\pdfoutput=1

\documentclass{amsart}

\usepackage[mathscr]{eucal}
\usepackage{amssymb}
\usepackage[usenames,dvipsnames]{xcolor} 
\usepackage[normalem]{ulem}
\usepackage{amsthm}
\usepackage{bbold}
\usepackage{enumitem}
\usepackage{amsmath}
\usepackage{wasysym}
\usepackage{tikz-cd}

\usepackage[unicode]{hyperref} 

\hypersetup{colorlinks=true,citecolor=Brown,urlcolor=Brown,linkcolor=Brown}

\numberwithin{equation}{section}
\usepackage[all]{xy}
\xyoption{line}
\usepackage{graphicx}

\swapnumbers 

\newtheorem{Thm}[equation]{Theorem}
\newtheorem*{Thm*}{Theorem}
\newtheorem{Prop}[equation]{Proposition}
\newtheorem{Lem}[equation]{Lemma}
\newtheorem{Cor}[equation]{Corollary}

\theoremstyle{remark}
\newtheorem{Def}[equation]{Definition}

\newtheorem{Not}[equation]{Notation}
\newtheorem{Exa}[equation]{Example}

\newtheorem{Cons}[equation]{Construction}

\newtheorem{Rem}[equation]{Remark}

\tikzset{
    labelrotatebelow/.style={anchor=north, rotate=90, inner sep=1.0mm}
}
\tikzset{
    labelrotateabove/.style={anchor=south, rotate=90, inner sep=1.0mm}
}

\newcommand{\nc}{\newcommand}
\nc{\dmo}{\DeclareMathOperator}

\nc{\Beren}[1]{{\color{MidnightBlue}#1}}
\nc{\Irkali}[1]{{\color{OliveGreen}#1}}
\nc{\Christian}[1]{{\color{Violet}#1}}
\nc{\Iout}[1]{\Irakli{\sout{#1}}}
\nc{\Bout}[1]{\Beren{\sout{#1}}}
\nc{\Cout}[1]{\Christian{\sout{#1}}}

\nc{\bbullet}{{\scriptscriptstyle\hspace{-1pt}\bullet}}
\nc{\bullett}{{\scriptscriptstyle\bullet}\hspace{-1pt}}
\nc{\LF}{L\hspace{-0.2ex}F}
\nc{\SpG}{\Sp^G}
\nc{\Prst}{{\cat P}\mathrm{r^{st}}}
\nc{\Mack}{\mathcal{M}ack}
\nc{\SC}{S\cat C}
\dmo{\DM}{DM}
\nc{\DMQ}{\DM_Q}
\dmo{\DerKal}{DMack}
\dmo{\Der}{D}
\dmo{\DMot}{DMot}
\dmo{\rmH}{H}
\dmo{\piu}{\underline{\pi}}
\dmo{\Sphere}{\mathbb{S}}
\nc{\HA}{{\rmH \hspace{-0.2em}\bbA}}
\nc{\HZ}{{\rmH \hspace{-0.2em}\bbZ}}
\nc{\HZbar}{{\rmH \hspace{-0.2em}\underline{\bbZ}}}
\nc{\Fp}{{\bbF_{\hspace{-0.1em}p}}}
\nc{\HFp}{{\rmH \hspace{-0.15em}\bbF_{\hspace{-0.1em}p}}}
\nc{\DHZG}{\Der(\HZ_G)}
\nc{\DHZH}{\Der(\HZ_H)}
\nc{\DHZK}{\Der(\HZ_K)}
\nc{\DHZGN}{\Der(\HZ_{G/N})}
\nc{\DHZGG}{\Der(\HZ_{G/G})}
\nc{\DHZCp}{\Der(\HZ_{C_p})}
\nc{\DHZGprime}{\Der(\HZ_{G'})}
\nc{\DHZ}{\Der(\HZ)}
\nc{\frakp}{\mathfrak{p}}
\nc{\frakq}{\mathfrak{q}}
\nc{\Z}{\mathbb{Z}}
\nc{\SSG}{\text{sSet}_*^G}
\nc{\sSet}{\text{sSet}}

\dmo{\Con}{Conj}
\dmo{\Id}{Id}
\dmo{\Loc}{Loc}
\dmo{\rmK}{\textrm{\rm K}}
\dmo{\Spc}{Spc}
\dmo{\thick}{thick}
\nc{\thickt}[1]{\thick_\otimes\langle #1 \rangle}
\dmo{\cone}{cone}
\dmo{\End}{End}
\dmo{\Mor}{Mor}
\dmo{\Hom}{Hom}
\dmo{\id}{id}
\dmo{\incl}{incl}
\dmo{\Img}{Im}
\dmo{\im}{im}
\dmo{\Ker}{Ker}
\dmo{\ind}{ind}
\dmo{\CoInd}{coind}
\dmo{\res}{res}
\dmo{\infl}{infl}
\dmo{\triv}{triv}
\dmo{\Tel}{Tel} 
\dmo{\Mod}{Mod}%
\dmo{\opname}{op}
\dmo{\SH}{SH}
\dmo{\smallb}{b}
\dmo{\Spec}{Spec}
\dmo{\supp}{supp}
\nc{\SHc}{{\SH^c}}
\nc{\SHp}{{\SH_{(p)}}}
\nc{\SHcp}{{\SH^c_{(p)}}}
\nc{\SHG}{\SH(G)}
\nc{\SHGp}{\SH(G)_{(p)}}
\nc{\SHGc}{\SHG^c}
\nc{\SHGcp}{\SHG^c_{(p)}}
\nc{\quadtext}[1]{\quad\textrm{#1}\quad}
\nc{\qquadtext}[1]{\qquad\textrm{#1}\qquad}
\nc{\adj}{\dashv}
\nc{\adjto}{\rightleftarrows}
\nc{\bbL}{\mathbb{L}}
\nc{\bbA}{\mathbb{A}}
\nc{\bbN}{\mathbb{N}}
\nc{\bbQ}{\mathbb{Q}}
\nc{\bbZ}{\mathbb{Z}}
\nc{\bbF}{\mathbb{F}}
\nc{\cat}[1]{\mathscr{#1}}
\nc{\ie}{{\sl i.e.}, }
\nc{\into}{\mathop{\rightarrowtail}}
\nc{\inv}{^{-1}}
\nc{\isoto}{\mathop{\overset{\sim}\to}}
\nc{\isotoo}{\mathop{\overset{\sim}\too}}
\nc{\onto}{\mathop{\twoheadrightarrow}}
\nc{\too}{\mathop{\longrightarrow}\limits}
\nc{\mapstoo}{\longmapsto}
\nc{\adh}[1]{\overline{#1}}
\nc{\adhpt}[1]{\adh{\{#1\}}}
\nc{\aka}{{a.\,k.\,a.}\ }
\nc{\calF}{\mathcal{F}}
\nc{\eg}{{\sl e.\,g.}}
\nc{\Homcat}[1]{\Hom_{\cat #1}}
\nc{\hook}{\hookrightarrow}
\nc{\ideal}[1]{\langle #1\rangle}
\nc{\ihom}{{\underline{\hom}}}
\nc{\Mid}{\,\big|\,}
\nc{\MMod}{\,\text{-}\Mod}%
\nc{\op}{^{\opname}}
\nc{\oto}[1]{\overset{#1}\to}
\nc{\otoo}[1]{\overset{#1}{\,\too\,}}
\nc{\sminus}{\!\smallsetminus\!}
\nc{\poplus}[1]{^{\oplus #1}}%
\nc{\potimes}[1]{^{\otimes #1}}
\nc{\sbull}{{\scriptscriptstyle\bullet}}
\nc{\SET}[2]{\big\{\,#1\Mid#2\,\big\}}
\nc{\SpcK}{\Spc(\cat K)}
\nc{\then}{\Rightarrow}
\nc{\unit}{\mathbb{1}}
\nc{\xra}{\xrightarrow}
\nc{\phigeom}[1]{\widetilde{\Phi}^{#1}}
\nc{\phigeomb}[1]{\Phi^{#1}}
\dmo{\Oname}{O}
\dmo{\proper}{proper}
\dmo{\lenormal}{\unlhd}
\dmo{\lnormal}{\lhd}
\nc{\normal}{\trianglelefteq}
\nc{\Op}{\Oname^p}
\nc{\Oq}{\Oname^q}
\dmo{\Sp}{Sp}
\dmo{\Ho}{Ho}
\dmo{\Fin}{Fin}
\dmo{\add}{add}
\dmo{\Fun}{Fun}
\dmo{\CAlg}{CAlg}
\dmo{\CMon}{CMon}
\dmo{\CC}{\cat C} 
\dmo{\DD}{\cat D}
\dmo{\OO}{\mathcal{O}}
\dmo{\Map}{Map}
\dmo{\Span}{Span}
\dmo{\N}{N}
\dmo{\Cat}{Cat}
\dmo{\colim}{colim}
\dmo{\Ch}{Ch}
\dmo{\A}{\mathbb{A}^{eff}}
\nc{\AGeff}{\mathbb{A}_G^{\mathrm{eff}}}
\nc{\BGeff}{\mathcal{B}_G^{\mathrm{eff}}}
\nc{\BG}{{\mathcal{B}_G}}
\nc{\NBGeff}{{\N}{\BGeff}}
\dmo{\Ab}{Ab}
\dmo{\Set}{Set}
\dmo{\ev}{ev}
\dmo{\Spcl}{Spcl}
\nc{\Funadd}{\Fun_{\add}}
\dmo{\proj}{proj}
\dmo{\cof}{cof}

\newcounter{enum-resume-hack}

\begin{document}


\title{The spectrum of derived Mackey functors}
\author{Irakli Patchkoria}
\author{Beren Sanders}
\author{Christian Wimmer}
\date{\today}

\begin{abstract}
	We compute the spectrum of the category of derived Mackey functors (in the sense of Kaledin) for all finite groups.  We find that this space captures precisely the top and bottom layers (i.e.~the height infinity and height zero parts) of the spectrum of the equivariant stable homotopy category.  Due to this truncation of the chromatic information, we are able to obtain a complete description of the spectrum for all finite groups, despite our incomplete knowledge of the topology of the spectrum of the equivariant stable homotopy category.  From a different point of view, we show that the spectrum of derived Mackey functors can be understood as the space obtained from the spectrum of the Burnside ring by ``ungluing'' closed points.  In order to compute the spectrum, we provide a new description of Kaledin's category, as the derived category of an equivariant ring spectrum, which may be of independent interest.  In fact, we clarify the relationship between several different categories, establishing symmetric monoidal equivalences and comparisons between the constructions of Kaledin, the spectral Mackey functors of Barwick, the ordinary derived category of Mackey functors, and categories of modules over certain equivariant ring spectra.  We also illustrate an interesting feature of the ordinary derived category of Mackey functors that distinguishes it from other equivariant categories relating to the behavior of its geometric fixed points.
\end{abstract}


\thanks{Second-named author supported by NSF grant~DMS-1903429.}

\maketitle

\tableofcontents

\section{Introduction}\label{sec:introduction} 
A compelling yet not completely understood phenomenon in hypertopical \mbox{algebra} is the impression that some stable homotopy theories appear (at least intuitively) to be ``linearizations'' of other stable homotopy theories.  For example, the derived category of the integers $\Der(\bbZ)$ can intuitively be regarded as a kind of ``linearization'' of the stable homotopy category of spectra $\SH$.  Although it is not our present goal to make this notion of ``linearization'' precise, one can readily find additional examples; consider, for example, the relationship between the derived category of motives $\DMot(k)$ over a field and the corresponding motivic stable homotopy category $\SH(k)$.  In this paper we are interested in the ``linearization'' of the \mbox{$G$-equivariant} stable homotopy category $\SH(G)$ for $G$ a finite group.

As the homotopy groups of a $G$-spectrum $X \in \SH(G)$ naturally form a graded \mbox{$G$-Mackey} functor, it seems plausible that the linearization of $\SH(G)$ would be some kind of derived category of Mackey functors.  The category of $G$-Mackey functors $\Mack(G)$ forms an abelian category, so we can certainly consider its derived category $\Der(\Mack(G))$, but Kaledin \cite{Kaledin11} argues that $\Der(\Mack(G))$ is not the ``correct'' definition of the derived category of Mackey functors.  He introduces a new triangulated category of ``derived Mackey functors'' $\DerKal(G)$, which contains $\Mack(G)$ as a subcategory, but with better behavior than $\Der(\Mack(G))$.  Part of Kaledin's argument against $\Der(\Mack(G))$ is that it does not behave in a way analogous to the equivariant stable homotopy category $\SH(G)$.  From our point of view, it is Kaledin's category $\DerKal(G)$ that is the correct ``linearization'' of $\SH(G)$, rather than the ordinary derived category $\Der(\Mack(G))$.

In this paper we will compute the tensor triangular spectrum of the compact objects in Kaledin's category of derived $G$-Mackey functors and explain its close relationship with the spectrum of the stable homotopy category of compact \mbox{$G$-spectra} (of which we have a fairly good understanding due to \cite{BalmerSanders17,BHNNNS19}).  We will find that the spectrum of $\DerKal(G)^c$ captures precisely the top and bottom chromatic layers of the spectrum of $\SH(G)^c$.  For example, the following diagram depicts the relationship between the two spaces for $G=C_p$ the cyclic group of order $p$:

\vspace{0.5em}
\resizebox{0.9\textwidth}{!}
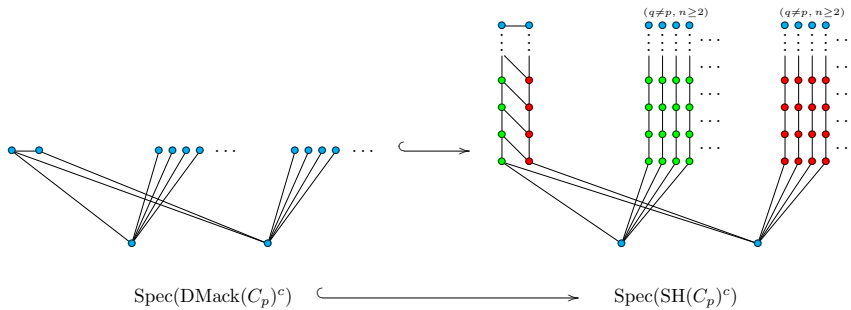
{\label{fig:SHCp}
\xy
	(-60,-25)*{\Spec(\DerKal(C_p)^c)};
	(25,-25)*{\Spec(\SH(C_p)^c)};
	{\ar@{^{(}->} (-40,-25)*{};(7,-25)*{}};
{\ar@{^{(}->} (-25,2)*{};(-13,2)*{}};
{\ar@{-} (15,-15)*{};(20,0)*{}};
{\ar@{-} (15,-15)*{};(22.5,0)*{}};
{\ar@{-} (15,-15)*{};(25,0)*{}};
{\ar@{-} (15,-15)*{};(27.5,0)*{}};
{\ar@{-} (40,-15)*{};(45,0)*{}};
{\ar@{-} (40,-15)*{};(47.5,0)*{}};
{\ar@{-} (40,-15)*{};(50,0)*{}};
{\ar@{-} (40,-15)*{};(52.5,0)*{}};
{\ar@{-} (15,-15)*{};(-7,0)*{}};
{\ar@{-} (40,-15)*{};(-7,0)*{}};
{\ar@{-} (40,-15)*{};(-2,0)*{}};
{\ar@{-} (-2,0)*{};(-7,5)*{}};
{\ar@{-} (-2,5)*{};(-7,10)*{}};
{\ar@{-} (-2,10)*{};(-7,15)*{}};
{\ar@{-} (-2,15)*{};(-6.5,19.5)*{}};
{\ar@{-} (-2,25)*{};(-7,25)*{}};
{\ar@{-} (-7,0)*{};(-7,5)*{}};
{\ar@{-} (-7,5)*{};(-7,10)*{}};
{\ar@{-} (-7,10)*{};(-7,15)*{}};
{\ar@{-} (-7,15)*{};(-7,19.5)*{}};
 (-7,0)*{\color{green}\bullet};
 (-7,0)*{\circ};
  (-7,5)*{\color{green}\bullet};
  (-7,5)*{\circ};
  (-7,10)*{\color{green}\bullet};
  (-7,10)*{\circ};
  (-7,15)*{\color{green}\bullet};
  (-7,15)*{\circ};
  (-7,23)*{\vdots};
  (-7,25)*{\color{cyan}\bullet};
  (-7,25)*{\circ};
%
{\ar@{-} (-2,0)*{};(-2,5)*{}};
{\ar@{-} (-2,5)*{};(-2,10)*{}};
{\ar@{-} (-2,10)*{};(-2,15)*{}};
{\ar@{-} (-2,15)*{};(-2,19.5)*{}};
 (-2,0)*{\color{red}\bullet};
 (-2,0)*{\circ};
  (-2,5)*{\color{red}\bullet};
  (-2,5)*{\circ};
  (-2,10)*{\color{red}\bullet};
  (-2,10)*{\circ};
  (-2,15)*{\color{red}\bullet};
  (-2,15)*{\circ};
  (-2,23)*{\vdots};
  (-2,25)*{\color{cyan}\bullet};
  (-2,25)*{\circ};
%
 (25,27.5)*{\scriptscriptstyle (q\neq p,\, n\ge2)};
{\ar@{-} (20,0)*{};(20,5)*{}};
{\ar@{-} (20,5)*{};(20,10)*{}};
{\ar@{-} (20,10)*{};(20,15)*{}};
{\ar@{-} (20,15)*{};(20,19.5)*{}};
 (20,0)*{\color{green}\bullet};
 (20,0)*{\circ};
  (20,5)*{\color{green}\bullet};
  (20,5)*{\circ};
  (20,10)*{\color{green}\bullet};
  (20,10)*{\circ};
  (20,15)*{\color{green}\bullet};
  (20,15)*{\circ};
  (20,23)*{\vdots};
  (20,25)*{\color{cyan}\bullet};
  (20,25)*{\circ};
{\ar@{-} (22.5,0)*{};(22.5,5)*{}};
{\ar@{-} (22.5,5)*{};(22.5,10)*{}};
{\ar@{-} (22.5,10)*{};(22.5,15)*{}};
{\ar@{-} (22.5,15)*{};(22.5,19.5)*{}};
 (22.5,0)*{\color{green}\bullet};
 (22.5,0)*{\circ};
  (22.5,5)*{\color{green}\bullet};
  (22.5,5)*{\circ};
  (22.5,10)*{\color{green}\bullet};
  (22.5,10)*{\circ};
  (22.5,15)*{\color{green}\bullet};
  (22.5,15)*{\circ};
  (22.5,23)*{\vdots};
  (22.5,25)*{\color{cyan}\bullet};
  (22.5,25)*{\circ};
{\ar@{-} (25,0)*{};(25,5)*{}};
{\ar@{-} (25,5)*{};(25,10)*{}};
{\ar@{-} (25,10)*{};(25,15)*{}};
{\ar@{-} (25,15)*{};(25,19.5)*{}};
     (25,0)*{\color{green}\bullet};
	 (25,0)*{\circ};
  (25,5)*{\color{green}\bullet};
  (25,5)*{\circ};
  (25,10)*{\color{green}\bullet};
  (25,10)*{\circ};
  (25,15)*{\color{green}\bullet};
  (25,15)*{\circ};
  (25,23)*{\vdots};
  (25,25)*{\color{cyan}\bullet};
  (25,25)*{\circ};
{\ar@{-} (27.5,0)*{};(27.5,5)*{}};
{\ar@{-} (27.5,5)*{};(27.5,10)*{}};
{\ar@{-} (27.5,10)*{};(27.5,15)*{}};
{\ar@{-} (27.5,15)*{};(27.5,19.5)*{}};
(27.5,0)*{\color{green}\bullet};
(27.5,0)*{\circ};
  (27.5,5)*{\color{green}\bullet};
  (27.5,5)*{\circ};
  (27.5,10)*{\color{green}\bullet};
  (27.5,10)*{\circ};
  (27.5,15)*{\color{green}\bullet};
  (27.5,15)*{\circ};
  (27.5,23)*{\vdots};
  (27.5,25)*{\color{cyan}\bullet};
  (27.5,25)*{\circ};
  (31.25,2.5)*{\hdots};
  (31.25,7.5)*{\hdots};
  (31.25,12.5)*{\hdots};
  (31.25,17.5)*{\hdots};
  (31.25,22.5)*{\hdots};
%
 (50,27.5)*{\scriptscriptstyle (q\neq p,\,n\ge2)};
{\ar@{-} (45,0)*{};(45,5)*{}};
{\ar@{-} (45,5)*{};(45,10)*{}};
{\ar@{-} (45,10)*{};(45,15)*{}};
{\ar@{-} (45,15)*{};(45,19.5)*{}};
 (45,0)*{\color{red}\bullet};
 (45,0)*{\circ};
  (45,5)*{\color{red}\bullet};
  (45,5)*{\circ};
  (45,10)*{\color{red}\bullet};
  (45,10)*{\circ};
  (45,15)*{\color{red}\bullet};
  (45,15)*{\circ};
  (45,23)*{\vdots};
  (45,25)*{\color{cyan}\bullet};
  (45,25)*{\circ};
{\ar@{-} (47.5,0)*{};(47.5,5)*{}};
{\ar@{-} (47.5,5)*{};(47.5,10)*{}};
{\ar@{-} (47.5,10)*{};(47.5,15)*{}};
{\ar@{-} (47.5,15)*{};(47.5,19.5)*{}};
(47.5,0)*{\color{red}\bullet};
(47.5,0)*{\circ};
  (47.5,5)*{\color{red}\bullet};
  (47.5,5)*{\circ};
  (47.5,10)*{\color{red}\bullet};
  (47.5,10)*{\circ};
  (47.5,15)*{\color{red}\bullet};
  (47.5,15)*{\circ};
  (47.5,23)*{\vdots};
  (47.5,25)*{\color{cyan}\bullet};
  (47.5,25)*{\circ};
{\ar@{-} (50,0)*{};(50,5)*{}};
{\ar@{-} (50,5)*{};(50,10)*{}};
{\ar@{-} (50,10)*{};(50,15)*{}};
{\ar@{-} (50,15)*{};(50,19.5)*{}};
  (50,0)*{\color{red}\bullet};
  (50,0)*{\circ};
  (50,5)*{\color{red}\bullet};
  (50,5)*{\circ};
  (50,10)*{\color{red}\bullet};
  (50,10)*{\circ};
  (50,15)*{\color{red}\bullet};
  (50,15)*{\circ};
  (50,23)*{\vdots};
  (50,25)*{\color{cyan}\bullet};
  (50,25)*{\circ};
{\ar@{-} (52.5,0)*{};(52.5,5)*{}};
{\ar@{-} (52.5,5)*{};(52.5,10)*{}};
{\ar@{-} (52.5,10)*{};(52.5,15)*{}};
{\ar@{-} (52.5,15)*{};(52.5,19.5)*{}};
  (52.5,0)*{\color{red}\bullet};
  (52.5,0)*{\circ};
  (52.5,5)*{\color{red}\bullet};
  (52.5,5)*{\circ};
  (52.5,10)*{\color{red}\bullet};
  (52.5,10)*{\circ};
  (52.5,15)*{\color{red}\bullet};
  (52.5,15)*{\circ};
  (52.5,23)*{\vdots};
  (52.5,25)*{\color{cyan}\bullet};
  (52.5,25)*{\circ};
  (56.25,2.5)*{\hdots};
  (56.25,7.5)*{\hdots};
  (56.25,12.5)*{\hdots};
  (56.25,17.5)*{\hdots};
  (56.25,22.5)*{\hdots};
  (15,-15)*{\color{cyan}\bullet};
  (15,-15)*{\circ};
  (40,-15)*{\color{cyan}\bullet};
  (40,-15)*{\circ};
{\ar@{-} (-75,-15)*{};(-70,2)*{}};
{\ar@{-} (-75,-15)*{};(-67.5,2)*{}};
{\ar@{-} (-75,-15)*{};(-65,2)*{}};
{\ar@{-} (-75,-15)*{};(-62.5,2)*{}};
{\ar@{-} (-50,-15)*{};(-45,2)*{}};
{\ar@{-} (-50,-15)*{};(-42.5,2)*{}};
{\ar@{-} (-50,-15)*{};(-40,2)*{}};
{\ar@{-} (-50,-15)*{};(-37.5,2)*{}};
{\ar@{-} (-75,-15)*{};(-97,2)*{}};
{\ar@{-} (-50,-15)*{};(-97,2)*{}};
{\ar@{-} (-50,-15)*{};(-92,2)*{}};
{\ar@{-} (-92,2)*{};(-97,2)*{}};
(-92,2)*{\color{cyan}\bullet};
(-92,2)*{\circ};
(-97,2)*{\color{cyan}\bullet};
(-97,2)*{\circ};
(-70,2)*{\color{cyan}\bullet};
(-70,2)*{\circ};
(-67.5,2)*{\color{cyan}\bullet};
(-67.5,2)*{\circ};
(-65,2)*{\color{cyan}\bullet};
(-65,2)*{\circ};
(-62.5,2)*{\color{cyan}\bullet};
(-62.5,2)*{\circ};
(-57.75,2)*{\hdots};
(-45,2)*{\color{cyan}\bullet};
(-45,2)*{\circ};
(-42.5,2)*{\color{cyan}\bullet};
(-42.5,2)*{\circ};
(-40,2)*{\color{cyan}\bullet};
(-40,2)*{\circ};
(-37.5,2)*{\color{cyan}\bullet};
(-37.5,2)*{\circ};
(-32.5,2)*{\hdots};
(-75,-15)*{\color{cyan}\bullet};
(-75,-15)*{\circ};
(-50,-15)*{\color{cyan}\bullet};
(-50,-15)*{\circ};
\endxy

}

\vspace{1em}
\noindent
Although there remain unresolved questions about the topology of $\Spec(\SH(G)^c)$ for nonabelian groups $G$, we are able to obtain a complete description of the space $\Spec(\DerKal(G)^c)$ for all finite groups because the unresolved chromatic interactions in the topology of $\Spec(\SH(G)^c)$ get truncated away at the top and bottom chromatic layers.

Our description of $\Spec(\DerKal(G)^c)$ is achieved in Theorem~\ref{thm:spec-as-set}, Theorem~\ref{thm:topology} and Proposition~\ref{prop:noetherian-space}.  The corresponding classification of thick tensor-ideals is included as Theorem~\ref{thm:classification}.  Moreover, the precise relationship with $\Spec(\SHG^c)$ is formulated in Corollary~\ref{cor:comparison-with-SHG}.

There is also a very satisfying relationship with the spectrum of the Burnside ring.  Recall that $\Spec(A(G))$ consists of a number of copies of $\Spec(\bbZ)$, one for each conjugacy class of subgroups, but with certain closed points glued together.  We will see that the spectrum of the category of derived Mackey functors is precisely the space obtained from the spectrum of the Burnside ring by ungluing these closed points.  More precisely, $\Spec(\DerKal(G)^c)$ consists of a number of \emph{disjoint} copies of $\Spec(\bbZ)$, one for each conjugacy class of subgroups, with topological interaction between the closed points describing the gluing that occurs in $\Spec(A(G))$. The following picture illustrates this for $G=C_p$:

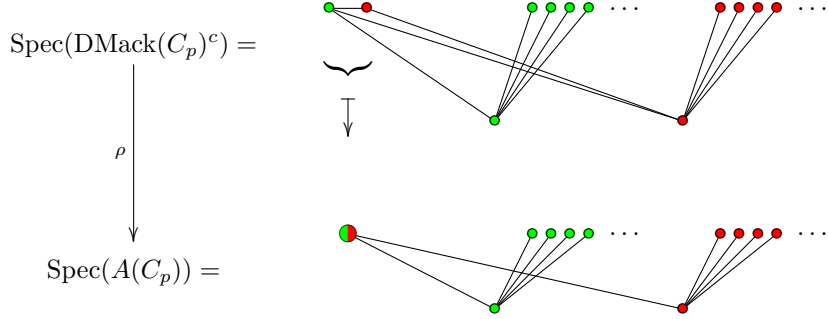
\[\label{fig:ACp}
\xy
{\ar@{->}_{\rho} (-33,-7.5)*{};(-33,-31)*{}};
(-33,-5)*{\Spec(\DerKal(C_p)^c) = };
(-33,-35)*{\Spec(A(C_p)) = };
(32.25,0)*{\hdots};
(57.25,0)*{\hdots};
{\ar@{-} (15,-15)*{};(20,0)*{}};
{\ar@{-} (15,-15)*{};(22.5,0)*{}};
{\ar@{-} (15,-15)*{};(25,0)*{}};
{\ar@{-} (15,-15)*{};(27.5,0)*{}};
{\ar@{-} (40,-15)*{};(45,0)*{}};
{\ar@{-} (40,-15)*{};(47.5,0)*{}};
{\ar@{-} (40,-15)*{};(50,0)*{}};
{\ar@{-} (40,-15)*{};(52.5,0)*{}};
{\ar@{-} (15,-15)*{};(-7,0)*{}};
{\ar@{-} (40,-15)*{};(-7,0)*{}};
{\ar@{-} (40,-15)*{};(-2,0)*{}};
 {\ar@{-} (-7,0)*{};(-2,0)*{}};
 (-7,0)*{\color{green}\bullet};
 (-7,0)*{\circ};
 (-2,0)*{\color{red}\bullet};
 (-2,0)*{\circ};
%
 (20,0)*{\color{green}\bullet};
 (20,0)*{\circ};
 (22.5,0)*{\color{green}\bullet};
 (22.5,0)*{\circ};
     (25,0)*{\color{green}\bullet};
	 (25,0)*{\circ};
(27.5,0)*{\color{green}\bullet};
(27.5,0)*{\circ};
 (45,0)*{\color{red}\bullet};
 (45,0)*{\circ};
(47.5,0)*{\color{red}\bullet};
(47.5,0)*{\circ};
  (50,0)*{\color{red}\bullet};
  (50,0)*{\circ};
  (52.5,0)*{\color{red}\bullet};
  (52.5,0)*{\circ};
  (15,-15)*{\color{green}\bullet};
  (15,-15)*{\circ};
  (40,-15)*{\color{red}\bullet};
  (40,-15)*{\circ};
%
%
  (-4.5,-8)*{\underbrace{\ }};
{\ar@{|->} (-4.5,-12)*{};(-4.5,-17)*{}};
%
%
{\ar@{-} (15,-40)*{};(20,-30)*{}};
{\ar@{-} (15,-40)*{};(22.5,-30)*{}};
{\ar@{-} (15,-40)*{};(25,-30)*{}};
{\ar@{-} (15,-40)*{};(27.5,-30)*{}};
{\ar@{-} (40,-40)*{};(45,-30)*{}};
{\ar@{-} (40,-40)*{};(47.5,-30)*{}};
{\ar@{-} (40,-40)*{};(50,-30)*{}};
{\ar@{-} (40,-40)*{};(52.5,-30)*{}};
{\ar@{-} (15,-40)*{};(-4.5,-30)*{}};
{\ar@{-} (40,-40)*{};(-4.5,-30)*{}};
(-4.5,-30)*{{\scriptscriptstyle{\color{green}\LEFTCIRCLE}\kern-.80em{\color{red}\RIGHTCIRCLE}} };
(-4.5,-30)*{\Circle};
(20,-30)*{\color{green}\bullet};
(20,-30)*{\circ};
(22.5,-30)*{\color{green}\bullet};
(22.5,-30)*{\circ};
(25,-30)*{\color{green}\bullet};
(25,-30)*{\circ};
(27.5,-30)*{\color{green}\bullet};
(27.5,-30)*{\circ};
(32.25,-30)*{\hdots};
(45,-30)*{\color{red}\bullet};
(45,-30)*{\circ};
(47.5,-30)*{\color{red}\bullet};
(47.5,-30)*{\circ};
(50,-30)*{\color{red}\bullet};
(50,-30)*{\circ};
(52.5,-30)*{\color{red}\bullet};
(52.5,-30)*{\circ};
(57.25,-30)*{\hdots};
(15,-40)*{\color{green}\bullet};
(15,-40)*{\circ};
(40,-40)*{\color{red}\bullet};
(40,-40)*{\circ};
\endxy

\]

\vspace{1em}
\noindent
Here the two closed points for the prime number $p$ (a green one for the trivial subgroup and a red one for the whole group $G$) are glued together in the spectrum of the Burnside ring but remain distinct in the spectrum of the category of derived Mackey functors.  A precise statement of the relationship between the two spaces is provided by Corollary~\ref{cor:comparison-with-AG} and additional examples are illustrated in \ref{exa:D8}--\ref{exa:S3}.

The category of derived Mackey functors thus lies cleanly between the equivariant stable homotopy category and the Burnside ring.  It is a chromatic truncation of the former and an equivariant refinement of the latter.  This clarifies the two distinct features noticed in \cite{BalmerSanders17} that distinguish $\Spec(\SH(G)^c)$ from $\Spec(A(G))$: the appearance of the chromatic filtration and the ungluing of the closed points.

\begin{center}
$\ast\ast\ast$
\end{center}

One feature common to the examples of linearization mentioned above 
\[\SH \rightsquigarrow \Der(\bbZ) \qquad\text{ and }\qquad \SH(k) \rightsquigarrow \DMot(k)\]
is that the linearized category can be interpreted as the derived category of modules over a suitable Eilenberg--MacLane spectrum. Indeed, 
\[\SH=\Der(\Sphere) \to \Der(\HZ)\cong \Der(\bbZ) \quad\text{ and }\quad \SH(k) \to \Der(\HZ_{mot})\cong\DMot(k)\]
where $\HZ$ is the ordinary Eilenberg--MacLane spectrum of the integers and $\HZ_{mot}$ is the motivic ring spectrum representing motivic cohomology.

We will apply a similar perspective to the linearization of $\SH(G)$ and give an alternative description of Kaledin's category of derived Mackey functors as the derived category $\DHZG$ of a commutative equivariant ring spectrum~$\HZ_G:=\triv_G(\HZ)$ (see Definition~\ref{def:HZG}).  In fact, we will establish symmetric monoidal equivalences between three categories: the derived category of $\HZ_G$-modules, the category of $\HZ$-valued spectral $G$-Mackey functors in the sense of Barwick \cite{Barwick17}, and the category of derived $G$-Mackey functors in the sense of Kaledin \cite{Kaledin11}.  These equivalences will all arise from equivalences of the underlying symmetric monoidal $\infty$-categories (see Corollary~\ref{cor: first comparison}, Proposition~\ref{prop: Comparison2} and Theorem~\ref{thm:equivalence}).  The comparison with Kaledin's constructions requires some technical care but the main key is Theorem~\ref{thm:burnside} and Corollary~\ref{cor:monoidal-burnside} which establish that Barwick's effective Burnside $\infty$-category is an $\infty$-categorical localization of Kaledin's Waldhausen type construction on the category of finite $G$-sets.

A common theme throughout the second half of the paper is the power of \mbox{$\infty$-categorical} monadicity theorems (e.g.~the Barr--Beck--Lurie Theorem) and the symmetric monoidal universal characterization of the $\infty$-category of $G$-spectra established by Gepner--Meier \cite{GepnerMeier} and Robalo \cite{Robalo15}.  As part of the analysis, we will obtain a symmetric monoidal equivalence between the $\infty$-category of \mbox{$G$-spectra} and the $\infty$-category of spectral $G$-Mackey functors (see Proposition~\ref{prop:universal-to-spectral-mackey} and Remark~\ref{rem:equiv-with-orthogonal}) which may be of independent interest.

\begin{center}
$\ast\ast\ast$
\end{center}

The details of the construction of $\DHZG$ will be given in Section~\ref{sec:highly-structured} and the equivalence $\DHZG \cong \DerKal(G)$ will be established in Section~\ref{sec:equivalence}.  The actual computation of $\Spec(\DHZG^c)$ appears in Section~\ref{sec:computation}.  Finally, we discuss in Section~\ref{sec:ordinary-derived} what goes wrong if one attempts to apply our method to compute the spectrum of the ordinary derived category of the abelian category of $G$-Mackey functors.  As observed by Greenlees and Shipley \cite{GreenleesShipley14}, this amounts to studying modules over the Eilenberg--MacLane $G$-spectrum associated to the Burnside ring Mackey functor: $\Der(\Mack(G))\cong \Der(\HA_G)$.  Ultimately things break down because the Eilenberg--MacLane spectra $\HA_G$ do not behave well with respect to geometric fixed points.  In fact, our explicit computations in Section~5 illustrate (and give a different perspective on) Kaledin's comments in \cite{Kaledin11} about the pathological behavior of the ordinary derived category of Mackey functors.  More precisely, we show that the target category of the geometric fixed point functor $\Phi^{H}$ for $\Der(\Mack(G))$ depends on the subgroup $H\le G$.  This is quite different than what happens for equivariant categories like $\SH(G)$ and $\DerKal(G)$ where the geometric fixed point functors $\Phi^{H}$ all land in the same category, namely the category associated to the trivial group.  This enables us, in those examples, to pull back information from the well-understood nonequivariant world.

\subsection*{Acknowledgements\,:}
We thank Rune Haugseng, Denis Nardin and Thomas Nikolaus for helpful conversations.  We also thank EPFL and the University of Bonn for their hospitality and for providing the venues where this project first got off the ground.
We also thank John Greenlees 
for reminding us that the dihedral group of order 8 has \emph{two} conjugacy classes of Klein-4 subgroups
(see Example~\ref{exa:D8}).
\bigbreak
\section{Computation of the spectrum}\label{sec:computation} 
We will assume familiarity with the description of $\SH(G)$ and the computation of its spectrum from \cite{BalmerSanders17}.  Following the approach in that work, we will begin by listing the essential features of the category of derived Mackey functors which are needed for the computation of its spectrum.  The main point is that there is a well-behaved notion of geometric fixed point functor which aligns with the corresponding notion for the equivariant stable homotopy category.  The crucial feature which leads to such well-behaved geometric fixed point functors is presented in \ref{it:finite-loc} below.  This feature of Kaledin's category of derived Mackey functors is not shared by the ordinary derived category of the abelian category of Mackey functors (as will be discussed in Section~\ref{sec:ordinary-derived}).
\begin{enumerate}[label=(\Alph*)]
\item\label{it:DHZG}
	For each finite group $G$, we have a tensor triangulated category $\DHZG$ and an adjunction $F_G:\SHG \adjto \DHZG:U_G$ where the left adjoint $F_G$ is a tensor triangulated functor.  The tensor triangulated category $\DHZG$ is rigidly compactly generated by $\SET{F_G(G/H_+)}{H \le G}$.  Consequently, $U_G$ is conservative (=reflects isomorphisms) and $F_G$ preserves compact (=rigid) objects.  Since $F_G$ preserves compact objects, it induces a tensor triangulated functor $\SHG^c \to \DHZG^c$ and hence a continuous map 
		\[
			\Spec(\DHZG^c) \xra{\Spec(F_G)} \Spec(\SHG^c).
		\]
\item \label{it:trivial}
	For $G=\{1\}$ the trivial group, $\DHZG \cong \DHZ$ and the adjunction in~\ref{it:DHZG} is the usual adjunction $F:\SH \adjto \DHZ:U$ induced by the unit $\Sphere \to \HZ$ of the Eilenberg--MacLane spectrum $\HZ$.
\item \label{it:change-of-group}
	For any homomorphism $\alpha:G \to G'$ of finite groups, there is an associated tensor triangulated functor $\alpha^*:\DHZGprime \to \DHZG$ which preserves coproducts, such that both squares in
		\[
		\begin{tikzcd}
			\SH(G') \ar[d,shift right, "F_{G'}"'] \ar[r,"\alpha^*"] & \SHG \ar[d, shift right, "F_G"'] \\
			\DHZGprime \ar[u,shift right, "U_{G'}"'] \ar[r, "\alpha^*"'] & \DHZG \ar[u, shift right, "U_G"']
		\end{tikzcd}
		\]
	commute up to natural isomorphism. As $\alpha^* : \DHZGprime \to \DHZG$ is a tensor triangulated functor between rigidly compactly generated tensor triangulated categories, it preserves compact (=rigid) objects and hence induces a continuous map
		\[
			\Spec(\DHZG^c) \xra{\Spec(\alpha^*)} \Spec(\DHZGprime^c).
		\]
	For a quotient $\alpha :G \to G/N$, we call $\infl_{G/N}^G := \alpha^*$ the inflation functor and for an inclusion $\alpha :H \hookrightarrow G$, we call $\res_H^G := \alpha^*$ the restriction functor. Moreover, we set $\triv_G := \infl_{G/G}^G$, regarded as a functor $\DHZ \to \DHZG$.
\item\label{it:composite-of-maps}
	For a composition $G \xra{\alpha} G' \xra{\beta} G''$ of group homomorphisms, we have a natural isomorphism $(\beta \circ \alpha)^* \cong \alpha^* \circ \beta ^*$. For example, $\res^G_K \cong \res^H_K \circ \res^G_H$ for $K \le H \le G$.
\item\label{it:adjoint}
	For any $H \le G$, the restriction functor $\res^G_H:\DHZG \to \DHZH$ has a left adjoint $\ind_H^G:\DHZH\to\DHZG$. 
\item\label{it:finite-loc}
	For any normal subgroup $N \lenormal G$, the composite
		\[
			\qquad \DHZGN \xra{\infl_{G/N}^G} \DHZG \to \DHZG / \Loc_{\otimes}\langle F_G(G/H_+) \;|\; H \not\supseteq N \rangle
		\]
	is an equivalence. In other words, $\DHZGN$ is a particular finite localization of $\DHZG$, obtained by killing the generators of $\DHZG$ associated to those subgroups which do not contain $N$. Define the ``geometric fixed point'' functor $\phigeom{N,G}:\DHZG \to \DHZGN$ to be the composite
		\[
			\DHZG \to \DHZG / \Loc_{\otimes}\langle F_G(G/H_+) \;|\; H \not\supseteq N \rangle \cong \DHZGN.
		\]
	By construction, it is split by inflation: $\phigeom{N,G} \circ \infl_{G/N}^G \cong \Id_{G/N}$. In particular, taking $N=G$, we have the geometric fixed point functor
		\[
			\phigeomb{G} : \DHZG \xra{\phigeom{G,G}} \DHZGG \cong \DHZ
		\]
	which is (up to equivalence) the localization obtained by killing all generators except $F_G(G/G_+) = \unit_{\DHZG}$.
\item\label{it:phiHG}
	For $H \le G$, we define $\phigeomb{H,G} : \DHZG \to \DHZ$ as the composite
		\[
			\DHZG \xra{\res^G_H} \DHZH \xra{\phigeomb{H}} \DHZ.
		\]
	These are tensor triangulated functors (preserving compact objects). In particular, $\phigeomb{H,G}$ induces a continuous map
		\[
			\Spec(\DHZ^c) \xra{\Spec(\phigeomb{H,G})} \Spec(\DHZG^c).
		\]
	For each prime ideal $\frakp \in \Spec \bbZ \cong \Spec(\DHZ^c)$ and $H \le G$, define $\cat P_G(H,\frakp) \in \Spec(\DHZG^c)$ by
		\[
			\cat P_G(H,\frakp) := \Spec(\phigeomb{H,G})(\frakp)=(\phigeomb{H,G})^{-1}(\frakp).
		\]
\item\label{it:conjugacy}
	For an inner automorphism $c_g := (-)^g : G \xra{\sim} G$, the induced functor $c_g^* : \DHZG \xra{\sim} \DHZG$ is naturally isomorphic to the identity functor.  It then follows from \ref{it:change-of-group} that for any subgroup $H \le G$, the left-hand triangle in
		\[
		\begin{tikzcd}
			 & \Der(\HZ_{H^g}) \ar[dd,"c_g^*", "\sim"'] \ar[dr,bend left=20, "\phigeomb{H^g}"] &\\
			\DHZG \ar[ur, bend left=20, "\res^G_{H^g}"] \ar[dr, bend right=20, "\res^G_H"'] & & \DHZ \\
			 & \DHZH \ar[ur,bend right=20,"\phigeomb{H}"']&
		\end{tikzcd}
		\]
	commutes up to natural isomorphism. The right-hand triangle also commutes up to natural isomorphism since $c_g^*(F_{H^g}(H^g/K^g)) \cong F_H(H/K)$ for any $K \le H \le G$ (again by \ref{it:change-of-group}).  Thus the functor $\phigeomb{H,G} : \DHZG \to \DHZ$ only depends, up to natural isomorphism, on the $G$-conjugacy class of the subgroup $H$.  That is, $\phigeomb{H^g,G} \cong \phigeomb{H,G}$ for any $g \in G$.  As naturally isomorphic functors induces the same map on spectra, it follows that $\cat P_G(H,\frakp) = \cat P_G(K,\frakp)$ if $H \sim_G K$.
\item \label{it:trivHZ}
	Finally, hinting at the reasons behind our choice of notation, $U_G(\unit) \cong \triv_G(\HZ)$ as commutative monoids in $\SHG$, where the right-hand side is the Eilenberg--MacLane spectrum of the integers regarded as a $G$-spectrum with trivial action.  This is the most ``explicit'' fact about our categories $\DHZG$ that we will need.
\end{enumerate}

\begin{Rem}
	The details of the construction of $\DHZG$ and verification of the above facts \ref{it:DHZG} through \ref{it:trivHZ} will be give in Section~\ref{sec:highly-structured}.  For the rest of the present section, we will use the above properties as a black-box in order to compute the spectrum of $\DHZG^c$ and describe its relationship with the spectrum of $\SHG^c$.
\end{Rem}

\begin{Rem}\label{rem:tensor-idempotents}
	The ``geometric fixed point'' functor $\phigeom{N,G}:\DHZG \to \DHZGN$ in \ref{it:finite-loc} is nothing but the finite localization associated to the thick tensor ideal of compact objects generated by the $F_G(G/H_+)$ for $H \not\supseteq N$.  As such it has an associated idempotent triangle $e_{\cat F[\not\supseteq N]} \to \unit \to f_{\cat F[\not\supseteq N]} \to \Sigma e_{\cat F[\not\supseteq N]}$ in $\DHZG$ and can be conveniently understood simply as tensoring with the right idempotent:
		\[
			f_{\cat F[\not\supseteq N]}\otimes -:\DHZG \to f_{\cat F[\not\supseteq N]}\otimes \DHZG \cong \DHZGN.
		\]
	Moreover, the latter equivalence is explicitly given by 
		\[
			f_{\cat F[\not\supseteq N]}\otimes \DHZG \hookrightarrow \DHZG \xra{(-)^N} \DHZGN
		\]
	where $(-)^N$ denotes the right adjoint of inflation (which exists since inflation preserves coproducts by assumption). In other words, it follows formally from the definition \ref{it:finite-loc} that the geometric fixed point functor is given as 
		\begin{equation}\label{eq:geom-formula}
			\phigeom{N,G}(X) \cong (f_{\cat F[\not\supseteq N]} \otimes X)^N.
		\end{equation}
	For further discussion of these tensor idempotents, see \cite[Section~5]{BalmerSanders17} and \cite{BalmerFavi11}.
\end{Rem}

\begin{Lem}\label{lem:geom-commute}
	For any $N \lenormal G$, the diagram
		\begin{equation}\label{eq:geom-compat}
		\begin{aligned}
		\begin{tikzcd}
			\SH(G) \ar[r,"\phigeom{N,G}"] \ar[d,"F_G"'] & \SH(G/N) \ar[d,"F_{G/N}"] \\
			\DHZG \ar[r,"\phigeom{N,G}"] & \DHZGN
		\end{tikzcd}
		\end{aligned}
		\end{equation}
	commutes up to isomorphism.
\end{Lem}

\begin{proof}
	Let $f \in \SH(G)$ denote the right idempotent for the finite localization $\phigeom{N,G}:\SH(G)\to\SH(G/N)$.  Applying \cite[Proposition~5.11]{BalmerSanders17} to the functor $F_G:\SH(G) \to \DHZG$ and recalling the definitions in~\ref{it:finite-loc} and \cite[(H)]{BalmerSanders17}, we see that $F_G(f) \in \DHZG$ is the right idempotent for the finite localization $\phigeom{N,G}:\DHZG\to\DHZGN$. Moreover, the middle square of 
		\[
		\begin{tikzcd}
			\SH(G/N) \ar[r,"\infl_{G/N}^G"] \ar[d,"F_{G/N}"'] & \SH(G) \ar[rr,bend left=15, "\phigeom{N,G}"] \ar[r] \ar[d,"F_G"'] & f\otimes \SH(G) \ar[d,"F_G"] \ar[r,xshift=6.20pt,"\cong"'] & \SH(G/N) \ar[d,"F_{G/N}"] \\
			\DHZGN \ar[r,"\infl_{G/N}^G"] & \DHZG \ar[rr, bend right=15, "\phigeom{N,G}"'] \ar[r] & F_G(f)\otimes \DHZG \ar[r,"\cong"] & \DHZGN
		\end{tikzcd}
		\]
	evidently commutes. The left-hand square commutes by \ref{it:change-of-group} and the horizontal composites are the identity.  It then follows formally that the right-hand square also commutes.
\end{proof}

\begin{Rem}
	It follows from Lemma~\ref{lem:geom-commute}, \ref{it:trivial}, \ref{it:change-of-group} and the definitions in \ref{it:phiHG} that
		\begin{equation}\label{eq:geom-compat2}
		\begin{aligned}
		\begin{tikzcd}
			\SHG \ar[r, xshift=-2.1pt, "\phigeomb{H}"] \ar[d,"F_G"'] & \SH \ar[d,"F"] \\
			\DHZG \ar[r,"\phigeomb{H}"] & \DHZ
		\end{tikzcd}
		\end{aligned}
		\end{equation}
	commutes up to isomorphism for any $H \le G$.
\end{Rem}

\begin{Lem}\label{lem:nested-phigeom}
	For any $N \lenormal G$, we have $\phigeomb{G} \cong \phigeomb{G/N} \circ \phigeom{N,G}$.
\end{Lem}

\begin{proof}
	Consider the two localizing $\otimes$-ideals $\cat L_1 \subset \cat L_2$ of $\DHZG$ given by
		\[
			\cat L_1 := \Loc_\otimes\langle F_G(G/H_+) \mid H \not\supseteq N\rangle
			\quad\text{ and }\quad
			\cat L_2 :=\Loc_{\otimes}\langle F_G(G/H_+) \mid H \lneq G \rangle.
		\]
	As Verdier quotients can be ``nested'', the localization
		\begin{equation}\label{eq:first-nest}
			\DHZG \rightarrow \DHZG/\cat L_2 
		\end{equation}
	coincides with the composite
		\begin{equation}\label{eq:second-nest}
			\DHZG \xrightarrow{q} \DHZG / \cat L_1 \rightarrow (\DHZG / \cat L_1)/ q(\cat L_2).
		\end{equation}
	Note that $\DHZG \to \DHZG/\cat L_2 \cong \DHZ$ is $\phigeomb{G}$ while $\DHZG \to \DHZG/\cat L_1 \cong \DHZGN$ is $\phigeom{N,G}$.  We just need to show that the quotient
		\[
			\DHZGN \to \DHZGN / \phigeom{N,G}(\cat L_2)
		\]
	is nothing but the finite localization defining $\phigeomb{G/N}:\DHZGN \to \DHZ$. Now $\phigeom{N,G}(\cat L_2)$ is the localizing tensor-ideal generated by $\SET{\phigeom{N,G}(F_G(G/H_+))}{H \lneq G}$ (see \cite[Prop.~2.31]{Verdier96} and \cite[Cor.~3.2.11]{Neeman01}).  This coincides with the localizing tensor-ideal generated by $\SET{F_{G/N}((G/N)/(H/N)_+)}{N \le H \lneq G}$ since
		\begin{eqnarray*}
			\phigeom{N,G}(F_G(G/H_+)) \cong F_{G/N}(\phigeom{N,G}(G/H_+)) \cong \begin{cases}
			 0 & \text{if } H \not\supseteq N \\
			 F_{G/N}({\scriptstyle \overline{G}/\overline{H}_+ }) & \text{if } H \supseteq N
			\end{cases}
		\end{eqnarray*}
	by Lemma~\ref{lem:geom-commute} and \cite[Cor.~II.9.9]{LewisMaySteinbergerMcClure86}.
\end{proof}

\begin{Lem}\label{lem:phigeom-and-restriction}
	For $N \le K \le G$ with $N \lenormal G$, we have $\phigeom{N,K}\circ \res_K^G \cong \res_{K/N}^{G/N}\circ \phigeom{N,G}$.
\end{Lem}

\begin{proof}
	Let $f_{\cat F[\not\supseteq N],G}$ denote the right idempotent in $\DHZG$ for $\phigeom{N,G}$ as in Remark~\ref{rem:tensor-idempotents}.  By \cite[Prop.~5.11]{BalmerSanders17}, its restriction $\res^G_K(f_{\cat F[\not\supseteq N],G})$ is the right idempotent for a finite localization of $\DHZK$, namely the localization associated to the compact thick tensor ideal generated by ${\SET{\res^G_K(F_G(G/H_+))}{H \le G, H \not\supseteq N}}$.  Using \ref{it:change-of-group} and the Mackey formula, this coincides with the thick tensor ideal generated by $\SET{F_K(K/L_+)}{L \le K, L \not\supseteq N}$.  In other words, $\res^G_K(f_{\cat F[\not\supseteq N],G}) \cong f_{\cat F[\not\supseteq N],K}$ is the idempotent in $\DHZK$ for $\phigeom{N,K}$.  Now, by \ref{it:composite-of-maps}, $\infl_{K/N}^K \circ \res_{K/N}^{G/N} \cong \res^G_K \circ \infl_{G/N}^G$.  Applying this equation to $\phigeom{N,G}(X)$ and post-composing by $\phigeom{N,K}$ we obtain
	\begin{align*}
		\res^{G/N}_{K/N}(\phigeom{N,G}(X)) 
		&\cong \phigeom{N,K}(\res^G_K(\infl^G_{G/N}(\phigeom{N,G}(X)))) \\
			&\cong \lambda^{N,K}(f_{\cat F[\not\supseteq N],K}\otimes \res^G_K(\infl^G_{G/N}(\phigeom{N,G}(X))))\\
			&\cong \lambda^{N,K}(\res^G_K(f_{\cat F[\not\supseteq N],G})\otimes \res^G_K(\infl^G_{G/N}(\phigeom{N,G}(X))))\\
			&\cong \lambda^{N,K}(\res^G_K(f_{\cat F[\not\supseteq N],G}\otimes \infl^G_{G/N}(\phigeom{N,G}(X))))\\
			&\cong \lambda^{N,K}(\res^G_K(f_{\cat F[\not\supseteq N],G}\otimes X))\\
			&\cong \lambda^{N,K}(f_{\cat F[\not\supseteq N],K}\otimes \res^G_K(X))\\
		&\cong \phigeom{N,K}(\res^G_K(X))
	\end{align*}
	where $\lambda^{N,K}$ denotes the right adjoint of $\infl_{K/N}^K$. Here we have used that 
	\[
		f_{\cat F[\not\supseteq N]} \otimes \infl_{G/N}^G(\phigeom{N,G}(X)) \cong f_{\cat F[\not\supseteq N]} \otimes X
	\]
	for any $X$ in~$\DHZG$. This follows from the fact that
	\[
		f_{\cat F[\not\supseteq N]}\otimes \DHZG \hookrightarrow \DHZG \xra{(-)^N} \DHZGN
	\]
	is an equivalence with quasi-inverse 
	\[
		\DHZGN \xra{\infl_{G/N}^G} \DHZG \to f_{\cat F[\not\supseteq N]} \otimes \DHZG
	\]
	as explained in Remark~\ref{rem:tensor-idempotents}.
\end{proof}

\begin{Lem}\label{lem:res}
	Let $K \le H \le G$. The map
	\[
		\Spec(\res_H^G) : \Spec(\DHZH^c) \to \Spec(\DHZG^c)
	\]
	sends $\cat P_H(K,\frakp)$ to $\cat P_G(K,\frakp)$.
\end{Lem}

\begin{proof}
	This is immediate from the definition of $\phigeomb{K,H}$ and $\phigeomb{K,G}$ (see \ref{it:phiHG}) and the relation $\res_K^G \cong \res_K^H \circ \res_H^G$ (see \ref{it:composite-of-maps}).
\end{proof}

\begin{Prop}\label{prop:image-of-res}
	For any $H \le G$, the image of the map
	\[
		\Spec(\res_H^G) : \Spec(\DHZH^c) \to \Spec(\DHZG^c)
	\]
	coincides with $\supp(F_G(G/H_+))$.
\end{Prop}

\begin{proof}
	Since the restriction functor $\res_H^G$ preserves coproducts \ref{it:change-of-group}, its left adjoint $\ind_H^G$ \ref{it:adjoint} necessarily preserves compact objects (see~\cite[Thm.~5.1]{Neeman96} for instance). Hence the adjunction $\ind_H^G \dashv \res_H^G$ restricts to an adjunction
	\[
		\ind_H^G :\DHZH^c \adjto \DHZG^c : \res_H^G
	\]
	on the subcategories of compact objects. Moreover, as the category $\DHZG$ is rigidly compactly generated \ref{it:DHZG}, its subcategory of compact objects $\DHZG^c$ is a \emph{rigid} category (i.e.~all objects are dualizable) so the duality $D$ provides an equivalence between $\DHZG^c$ and its opposite category.  It follows that $D\ind_H^G D$ is \emph{right} adjoint to $\res_H^G$ on the categories of compact objects.  We can then invoke \cite[Thm.~1.7]{Balmer17} to conclude that the image of $\Spec(\res^G_H)$ equals $\supp(D\ind_H^G D\unit)$.  This coincides with $\supp(\ind_H^G(\unit))$ since $D\unit=\unit$ and $\supp(DX)=\supp(X)$ (by \cite[Prop.~2.7]{Balmer07} for instance). Finally, by~\ref{it:change-of-group} we have $\res_H^G \circ U_G \cong U_H \circ \res^G_H$.  Taking left adjoints, $\ind_H^G \circ F_H \cong F_G \circ \ind_H^G$ so $\ind_H^G(\unit) \cong \ind_H^G(F_H(\unit)) \cong F_G(\ind_H^G(\unit)) \cong F_G(G/H_+)$.
\end{proof}

\begin{Lem}\label{lem:map-for-inflation}
	Let $N \le K \le G$ with $N \lenormal G$. The map
	\[
		\Spec(\infl_{G/N}^G) : \Spec(\DHZG^c) \to \Spec(\DHZGN^c)
	\]
	sends $\cat P_G(K,\frakp)$ to $\cat P_{G/N}(K/N,\frakp)$.
\end{Lem}

\begin{proof}
	Unravelling the definitions and factoring the composite $K\to G\to G/N$ as $K\to K/N \to G/N$, property \ref{it:composite-of-maps} reduces our claim to the assertion that $\phigeomb{G} \circ \infl_{G/N}^G \cong \phigeomb{G/N}$. This follows from Lemma~\ref{lem:nested-phigeom} since $\infl_{G/N}^G$ splits $\phigeom{N,G}$.
\end{proof}

\begin{Lem}\label{lem:map-for-phigeom}
	Let $N \le K\le G$ with $N \lenormal G$. The map
	\[
		\Spec(\phigeom{N,G}) : \Spec(\DHZGN^c) \to \Spec(\DHZG^c)
	\]
	sends $\cat P_{G/N}(K/N,\frakp)$ to $\cat P_G(K,\frakp)$.
\end{Lem}

\begin{proof}
	This follows from Lemma~\ref{lem:nested-phigeom} and Lemma~\ref{lem:phigeom-and-restriction} and the definitions.
\end{proof}

\begin{Cor}\label{cor:quotient}
	Let $N \le K\le G$ with $N \lenormal G$. Then $\cat P_G(K,\frakp) \subseteq \cat P_G(G,\frakq)$ if and only if $\cat P_{G/N}(K/N,\frakp) \subseteq \cat P_{G/N}(G/N,\frakq)$.
\end{Cor}

\begin{proof}
	The induced maps on spectra preserve inclusions.  Thus $(\Rightarrow)$ follows from Lemma~\ref{lem:map-for-inflation} and $(\Leftarrow)$ follows from Lemma~\ref{lem:map-for-phigeom}.
\end{proof}

\begin{Rem}
	Recall the prime ideals of the nonequivariant stable homotopy category of finite spectra $\SH^c$. They are of the form $\cat C_{p,n}$ where~$p$ is a prime number and $1 \le n \le \infty$ is a ``chromatic'' number.  Recall that $\cat C_{p,1} = \SH^{c,\textrm{tor}} =:\cat C_{0,1}$ is the subcategory of finite torsion spectra, independently of~$p$.  It is the unique generic point of $\Spec(\SH^c)$, while the points $\cat C_{p,\infty}$ are the closed points.  Similarly recall that the prime ideals of $\SH(G)^c$ are of the form $\cat P(H,p,n)$ for $H \le G$, $p$ a prime number and $1 \le n \le \infty$. Again, $\cat P(H,p,1)$ is independent of $p$ and sometimes written $\cat P(H,1)$ or $\cat P(H,0,1)$.
\end{Rem}

\begin{Prop}\label{prop:nonequiv-compat}
	The map $\Spec(\DHZ^c) \xra{\Spec(F)} \Spec(\SH^c)$ sends a prime ideal $\frakp \in \Spec \bbZ \cong \Spec(\DHZ^c)$ to $\cat C_{p,\infty}$ if $\frakp = (p)$ and to $\cat C_{0,1}$ if $\frakp=(0)$.
\end{Prop}

\begin{proof}
	It follows from the Hurewicz theorem that the functor $\HZ \wedge - : \SH \to \SH$ is conservative on compact objects.  That is, if $X \in \SH^c$ then $\HZ \wedge X = 0$ if and only if $X=0$.  As a corollary, the functor $F:\SH^c \to \DHZ^c$ is conservative: if $F(X) = 0$ then $\HZ \wedge X \cong UF(X) = 0$, hence $X=0$.  Thus by \cite[Thm.~1.2]{Balmer17}, the induced map on spectra $\varphi:=\Spec(F):\Spec(\DHZ^c) \to \Spec(\SH^c)$ hits all the closed points $\cat C_{p,\infty}$ of $\Spec(\SH^c)$.  Now, the unit map $\Sphere \to \HZ$ induces an isomorphism of rings $\pi_0(\Sphere) \to \pi_0(\HZ)$ which, under the usual identifications of both sides with the ring of integers, is just the identity.  This is precisely the map on endomorphism rings $\End_{\SH}(\unit) \to \End_{\DHZ}(\unit)$ induced by the functor $F:\SH^c \to \DHZ^c$.  Since the comparison map $\rho:\Spec(\cat K) \to \Spec(\End_{\cat K}(\unit))$ of \cite[Section 5]{Balmer10b} is natural, we have a commutative diagram
	\[\begin{tikzcd}[row sep=scriptsize]
			\Spec(\DHZ^c) \ar[dd,"\rho"', "\cong"] \ar[r,"\varphi"]& \Spec(\SH^c) \ar[dd,"\rho"] \\\\
			\Spec(\pi_0(\HZ)) \ar[d,"\cong"] \ar[r,"\cong"] & \Spec(\pi_0(\Sphere)) \ar[d,"\cong"]\\
			\Spec(\bbZ) \ar[r,equals]	& \Spec(\bbZ)
	\end{tikzcd}\]
	and the left-hand comparison map is just the usual identification of the spectrum of $\DHZ^c \cong \Der(\bbZ)^c$ with the spectrum of the integers.  Thus the top map $\varphi$ sends the prime $(0)$ in $\Spec(\DHZ^c)$ to a point in the fiber (with respect to $\rho$) of $(0)$ in $\Spec(\SH^c)$.  There is only one such point in the fiber, namely $\cat C_{0,1} = \SH^{c,\textrm{tor}}$.  On the other hand, the prime $(p)$ in $\Spec(\DHZ^c)$ maps to a point in the fiber of $(p)$ in $\Spec(\SH^c)$.  Since all the closed points of $\Spec(\SHc)$ are hit, the closed point $\cat C_{p,\infty}$ in the fiber over $(p)$ must be hit.  Since $(p)$ in $\Spec(\DHZ^c)$  is the only point mapping to the fiber over $(p)$, the only possibility is that it maps to the closed point $\cat C_{p,\infty}$.
\end{proof}
\begin{Cor}\label{cor:comp-with-SHG}
	For any $H \le G$, the map $\Spec(\DHZG^c) \xra{\Spec(F_G)} \Spec(\SHG^c)$ sends $\cat P_G(H,p)$ to $\cat P_G(H,p,\infty)$ and $\cat P_G(H,0)$ to $\cat P_G(H,0,1)$.
\end{Cor}

\begin{proof}
	This follows immediately from~\eqref{eq:geom-compat2} and Proposition~\ref{prop:nonequiv-compat}.
\end{proof}

\begin{Rem}\label{rem:endo-unit}
	By formal nonsense, the adjunction in \ref{it:DHZG} provides an isomorphism 
	\[
		\End_{\DHZG}(\unit) \simeq \pi_0(U_G(\unit))
	\]
	of commutative rings, where $U_G(\unit)$ is regarded as a commutative monoid in $\SHG$ via the induced lax monoidal structure on the functor $U_G$.  Moreover, the map on endomorphism rings $\pi_0(\unit) = \End_{\SHG}(\unit) \to \End_{\DHZG}(\unit) \simeq \pi_0(U_G(\unit))$ induced by the functor $F_G$ is just post-composition by the unit of $U_G(\unit)$.  Property \ref{it:trivHZ} asserts that we have an isomorphism $U_G(\unit) \simeq \triv_G(\HZ)$ of commutative monoids in $\SHG$. The map on endomorphism rings can then be identified with the map
	\[
		\pi_0( (\triv_G(\Sphere)^G) \to \pi_0( \HZ \wedge \triv_G(\Sphere)^G)
	\]
	induced by the unit of the Eilenberg--MacLane spectrum $\HZ$. This is an isomorphism by the Hurewicz theorem since the spectrum $\triv_G(\Sphere)^G$ is connective (being a wedge sum of suspension spectra by the tom Dieck splitting theorem).  In this way, we have an identification $A(G) \simeq \End_{\SHG}(\unit) \simeq \End_{\DHZG}(\unit)$ between the Burnside ring and the endomorphism ring of the unit in $\DHZG$.
\end{Rem}

\begin{Cor}\label{cor:comp-with-AG}
	The comparison map $\rho : \Spec(\DHZG^c) \to \Spec(A(G))$ sends $\cat P(H,p)$ to $\frakp(H,p) \in \Spec(A(G))$ and sends $\cat P(H,0)$ to $\frakp(H,0) \in \Spec(A(G))$.
\end{Cor}

\begin{proof}
	Naturality of the comparison map gives a commutative diagram
	\[
	\begin{tikzcd}[row sep=scriptsize]
		\Spec(\DHZG^c) \ar[dd,"\rho"'] \ar[r,"\Spec(F_G)"] &\Spec(\SHGc) \ar[dd,"\rho"] \\\\
		\Spec(\End_{\DHZG}(\unit)) \ar[d,"\cong"] \ar[r,"\cong"] & \Spec(\End_{\SHG}(\unit)) \ar[d,"\cong"] \\
		\Spec(A(G))	\ar[r, equals]& \Spec(A(G))
	\end{tikzcd}
	\]
	and the claim follows from Corollary~\ref{cor:comp-with-SHG} and \cite[Proposition~6.7]{BalmerSanders17} (see Remark~\ref{rem:endo-unit}).  The spectrum of the Burnside ring is recalled in \cite[Section 3]{BalmerSanders17}.
\end{proof}

\begin{Thm}\label{thm:spec-as-set}
	Let $G$ be a finite group. Every prime ideal of $\DHZG^c$ is of the form $\cat P(H,\frakp)$ for some $H \le G$ and $\frakp \in \Spec \bbZ$.  Moreover, the prime $\cat P(H,\frakp)$ is completely determined by the $G$-conjugacy class of $H$ and the prime ideal $\frakp$. That is, $\cat P(H,\frakp) = \cat P(K,\frakq)$ if and only if $H \sim_G K$ and $\frakp = \frakq$.
\end{Thm}

\begin{proof}
	We will prove the theorem by induction on the order $|G|$.  By construction~\ref{it:finite-loc}, the geometric fixed points $\phigeomb{G}:\DHZG \to \DHZ$ is a finite localization.  Hence by the Neeman--Thomason localization theorem \cite[Thm.~2.1]{Neeman92b}, the induced map 
	\[
		\Spec(\DHZ^c) \xra{\Spec(\phigeomb{G})} \Spec(\DHZG^c)
	\]
	is a homeomorphism onto the subset $V(F_G(G/H_+) \mid H \lneq G) \subseteq \Spec(\DHZG^c)$ consisting of those primes $\cat P \in \Spec(\DHZG^c)$ which contain $F_G(G/H_+)$ for all $H \lneq G$. In particular, 
	\[
		\SET{\cat P(G,\frakp)}{\frakp \in \Spec \bbZ} = V(F_G(G/H_+) \mid H \lneq G).
	\]
	The complement is thus given by
	\begin{align*}
		\Spec(\DHZG^c) \setminus \SET{\cat P(G,\frakp)}{\frakp \in \Spec \bbZ}
		&= \bigcup_{H \lneq G} \supp(F_G(G/H_+))\\
		&= \bigcup_{H \lneq G} \im(\Spec(\res_H^G))
	\end{align*}
	where the last equality is given by Proposition~\ref{prop:image-of-res}. By the inductive hypothesis, every prime in $\DHZH$ (for $H \lneq G$) is of the form $\cat P_H(K,\frakp)$ for some $K \le H$ and $\frakp \in \Spec(\bbZ)$.  By Lemma~\ref{lem:res}, it gets mapped to $\cat P_G(K,\frakp)$ under $\Spec(\res_H^G)$.  This completes the proof that every prime is of the required form. The uniqueness statement follows from~\ref{it:conjugacy}, Corollary~\ref{cor:comp-with-SHG} and \cite[Theorem~4.14]{BalmerSanders17}.
\end{proof}

\begin{Rem}\label{rem:covered}
	In other words, $\Spec(\DHZG^c)$ is covered by copies of $\Spec(\bbZ)$, one copy for each conjugacy class of subgroups $H \le G$.  These copies are \emph{disjoint}, so as a set $\Spec(\DHZG^c)$ is just the disjoint union of these copies of $\Spec(\bbZ)$.  However, the copies of $\Spec(\bbZ)$ are related by the topology of $\Spec(\DHZG^c)$.  Our next task is to determine this topology.  This will follow from a series of reductions culminating in Theorem~\ref{thm:topology}.
\end{Rem}

\begin{Rem}
	Understanding the topology boils down to understanding the inclusions among the primes (i.e.~understanding the irreducible closed sets) and the comparison map to the spectrum of the Burnside ring (see Remark~\ref{rem:endo-unit} and Corollary~\ref{cor:comp-with-AG}) greatly restricts the possible inclusions (Lemma~\ref{lem:restrict-inclusions} below).
\end{Rem}

\begin{Rem}
	A subgroup $H \le G$ is said to be a \emph{$p$-subnormal subgroup} if there exists a subnormal tower from $H$ to $G$ all of whose subquotients have order $p$.  We refer the reader to \cite[Section 3]{BalmerSanders17} for more details.
\end{Rem}

\begin{Lem}\label{lem:restrict-inclusions}
	Let $K,H \le G$ be two subgroups and $\frakp, \frakq \in \Spec \bbZ$.  Suppose $\cat P_G(K,\frakp) \subseteq \cat P_G(H,\frakq)$ in $\DHZG^c$. Then:
	\begin{enumerate}
		\item If $\frakp=(0)$ then $\frakq=(0)$ and $K \sim_G H$ (in which case the inclusion is an equality).
		\item If $\frakp=(p)$ then $K$ is $G$-conjugate to a $p$-subnormal subgroup of $H$ and $\frakq=(p)$ or $(0)$.
	\end{enumerate}
\end{Lem}

\begin{proof}
	This follows from Corollary~\ref{cor:comp-with-SHG} and \cite[Proposition~6.9]{BalmerSanders17} together with Corollary~\ref{cor:comp-with-AG} and what is known about the inclusions among the prime ideals of the Burnside ring (e.g.~from \cite[Theorem~3.6]{BalmerSanders17} or the original \cite{Dress69}).
\end{proof}

\begin{Rem}\label{rem:top-strat}
	On the other hand, we know that $\cat P_G(H,p) \subseteq \cat P_G(H,0)$ for any subgroup ${H \le G}$ since the map $\Spec(\phigeomb{H}):\Spec(\DHZ^c) \to \Spec(\DHZG^c)$ preserves inclusions and the identification $\Spec(\DHZ^c) \cong \Spec(\bbZ)$ reverses inclusions.  Armed with Lemma~\ref{lem:restrict-inclusions} and this observation, all that remains to determine the topology of $\Spec(\DHZG^c)$ is to understand when we have an inclusion $\cat P_G(K,p) \subseteq \cat P_G(H,p)$ when $K$ is \mbox{$G$-conjugate} to a $p$-subnormal subgroup of $H$.  We will show that this inclusion always holds (see Proposition~\ref{prop:psubnormal-inclusion} and Theorem~\ref{thm:topology} below).  To prove this we will use a series of reductions which ultimately reduces the problem to the case $G=C_p$, the cyclic group of order $p$.
\end{Rem}

\begin{Rem}
	The following result explains how vanishing of the Tate construction relates to the geometry of the Balmer spectrum.
\end{Rem}

\begin{Prop}\label{prop:idempotents-and-splitting}
	Let $\cat T$ be a rigidly-compactly generated tensor triangulated category and let $\cat K := \cat T^c$ denote its subcategory of compact-rigid objects. For any Thomason subset $Y \subset \Spec(\cat K)$, let $\cat K_{Y} = \SET{x \in \cat K}{\supp(x) \subseteq Y}$ be the corresponding thick tensor-ideal of $\cat K$, and let $e_{Y} \to \unit\to f_{Y} \to \Sigma e_{Y}$ be the idempotent triangle in $\cat T$ for the associated finite localization. For any object $x \in \cat K$, the following are equivalent:
	\begin{enumerate}
		\item	The Tate construction $t_{Y}(x) = 0$ vanishes.
		\item	The exact triangle
				\[
					e_Y \otimes x \to x \to f_Y \otimes x \to \Sigma e_Y \otimes x
				\]
				splits; that is, $f_Y \otimes x \to \Sigma e_Y \otimes x$ is the zero map.
		\item	The support of $x$ is a disjoint union of closed sets
				\[
					\supp(x) = Z_1 \sqcup Z_2
				\] 
				with $Z_1 \subseteq Y$ and $Z_2 \cap Y = \emptyset$.
		\item	$\supp(x) \cap (\Spec(\cat K) \setminus Y)$ is Thomason.
	\end{enumerate}
\end{Prop}

\begin{proof}
	Recall that $t_Y : \cat T \to \cat T$ is defined by $t_Y := [f_Y,\Sigma e_Y \otimes -]$ where $[-,-]$ denotes the internal hom in $\cat T$ (see \cite[Definition~5.7]{BalmerSanders17} or \cite{Greenlees01}).

	(a) $\Rightarrow$ (b): The kernel of $t_Y : \cat T \to \cat T$ is a thick subcategory of $\cat T$ which is closed under tensoring with compact-rigid objects. Thus, $t_Y(x) = 0$ iff $t_Y(Dx \otimes x) = [f_Y \otimes x, \Sigma e_Y \otimes x] = 0$. This implies (b) since $\Hom_{\cat T}(a,b) \cong \Hom_{\cat T}(\unit,[a,b])$.

	(b) $\Rightarrow$ (c): If the exact triangle splits then $x \simeq (e_Y \otimes x) \oplus (f_Y \otimes x)$. In particular, the objects $e_Y \otimes x$ and $f_Y \otimes x$ are both contained in $\cat K$ (i.e.~are compact). Then defining $Z_1 := \supp(e_Y \otimes x)$ and $Z_2 := \supp(f_Y \otimes x)$, we have $\supp(x) = Z_1 \sqcup Z_2$, a disjoint union of closed sets. Finally, recall that $\Loc\langle \cat K_Y \rangle = e_Y \otimes \cat T$. Thus, for any $c \in \cat K$, $\supp(c) \subseteq Y$ is equivalent to $f_Y \otimes c=0$. Similarly, if $\supp(c) \cap Y = \emptyset$ then for any $d \in \cat K_Y$, $d \otimes c = 0$; hence $e_Y \otimes c=0$. Conversely, if $e_Y \otimes c=0$ then $d \otimes c \simeq d \otimes e_Y \otimes c =0$ for any $d \in \cat K_Y$. It follows that $\supp(c) \cap Y = \emptyset$ is equivalent to $e_Y \otimes c=0$. This proves (c) by considering $c:=e_Y \otimes x$ and $c:=f_Y \otimes x$.

	(c) $\Rightarrow$ (d): Observe that $\Spec(\cat K) \setminus Z_2 = Z_1 \cup (\Spec(\cat K)\setminus\supp(x))$ is a union of two quasi-compact subsets of $\Spec(\cat K)$, and hence is itself quasi-compact. The closed set $Z_2 = \supp(x) \cap (\Spec(\cat K)\setminus Y)$ thus has quasi-compact complement, and hence is a Thomason closed subset.  

	(d) $\Rightarrow$ (a): The hypothesis implies that
		\[
			\supp(x) = (\supp(x) \cap Y) \sqcup (\supp(x) \cap (\Spec(\cat K) \setminus Y))
		\]
		is a decomposition into disjoint Thomason sets. Then by the generalized Carlson theorem \cite[Theorem~2.11]{Balmer07} we have $x \simeq a\oplus b$ for two objects $a,b \in \cat K$ with $\supp(a) \subseteq Y$ and $\supp(b) \cap Y = \emptyset$. Then $t_Y(x) = t_Y(a) \oplus t_Y(b)$ vanishes since $f_Y \otimes Da =0$ and $e_Y \otimes b =0$.
\end{proof}

\begin{Prop}\label{prop:no-split-Cp}
	Let $e_G \to \unit \to f_G \to \Sigma e_G$ be the idempotent triangle in $\SHG$ associated to the trivial family of subgroups. (That is, $e_G = EG_+$ and $f_G = \widetilde{E}G$.) For $G=C_p$, this triangle does not split after tensoring with $\triv_G(\HFp)$; that is, the map
		\[
			\triv_G(\HFp) \otimes f_G \to \triv_G(\HFp) \otimes \Sigma e_G
		\]
	in $\SHG$ is not the zero map.
\end{Prop}

\begin{proof}
	For notational simplicity, let $\rmH:=\triv_G(\HFp)$.  If the map $\rmH \to \rmH \otimes f_G$ has a section in $\SH(G)$, then the map of $G$-Mackey functors $\piu_0(\rmH) \to \piu_0(\rmH \otimes f_G)$ would have a section:
	\begin{equation}\label{eq:mackeysection}\begin{aligned}\begin{tikzcd}
		\piu_0(\rmH \otimes f_G) \ar[r,dashed,"\exists\sigma"] \ar[rr,bend right=20, "\id"'] & \piu_0(\rmH) \ar[r] & \piu_0(\rmH \otimes f_G).
	\end{tikzcd}\end{aligned}\end{equation}
	The Mackey functor $\piu_0(\rmH)$ can be identified with $\bbA \otimes \Fp$ where $\bbA$ denotes the Burnside ring $G$-Mackey functor $G/H \mapsto A(H)$. On the other hand, for $G=C_p$, the Mackey functor $\piu_0(\rmH \otimes f_G)$ satisfies 
	\[
		\piu_0(\rmH \otimes f_G)(G/G) = \pi_0^G(\rmH \otimes f_G) \cong \pi_0( \Phi^G(\rmH)) \cong \pi_0(\HFp) \cong \Fp
	\]
	and
	\[
		\piu_0(\rmH \otimes f_G)(G/\{1\}) = \pi_0^{\{1\}}(\rmH \otimes f_G)\cong \pi_0( \HFp \otimes \res^G_{\{1\}}(f_G)) = 0.
	\]
	Now $A(C_p)$ is the ring $\bbZ[X]/(X^2-pX)$ with restriction $A(C_p) \to A(\{1\})=\bbZ$ given by $X \mapsto p$ and with transfer $\bbZ =A(\{1\}) \to A(C_p)$ given by $1 \mapsto X$. A splitting~\eqref{eq:mackeysection} of $C_p$-Mackey functors would thus look like
	\[\begin{tikzcd}
		\Fp \ar[r,"\sigma"] \ar[d,shift right] & \Fp[X]/(X^2) \ar[d, shift right] \ar[r] & \Fp \ar[d, shift right] \\
		0 \ar[r] \ar[u, shift right] & \Fp \ar[u, shift right] \ar[r] & 0 \ar[u,shift right]
	\end{tikzcd}\]
	where the vertical maps represent restriction and transfer. Since the right-hand map commutes with transfers, it must map $X$ to $0$ in $\Fp$.  Hence, in order for the composite to be the identity, the left-hand map $\sigma$ must map $1 \in \Fp$ to an element of the form $1+mX \in \Fp[X]/(X^2)$.  Since the middle restriction map sends $X$ to $0 \in \Fp$, the element $1 +mX$ is mapped to $1 \in \Fp$.  On the other hand, since $\sigma$ commutes with restrictions, $1+mX$ must be mapped to $0 \in \Fp$. This is a contradiction.
\end{proof}

\begin{Prop}\label{prop:Cp}
	Consider $G=C_p$. Then $\cat P(1,p) \subseteq \cat P(C_p,p)$ in $\DHZCp$.
\end{Prop}

\begin{proof}
	Recall that if $F:\cat K \to \cat L$ is a tensor triangulated functor and $\varphi := \Spec(F) : \Spec(\cat L) \to \Spec(\cat K)$ is the induced map on spectra, then for any ${x \in \cat K}$, $\supp_{\cat L}(F(x)) = \varphi^{-1}(\supp_{\cat K}(x))$. Then consider the mod-$p$ Moore spectrum $M(p)$. Its support in $\Spec(\SH^c)$ is precisely $\overline{\{ \cat C_{p,2} \}}=\{\cat C_{p,n} \mid 2 \le n \le \infty\}$. We can then pass to the $G$-equivariant stable homotopy category by giving the mod-$p$ Moore spectrum a trivial $G$-action. By \cite[Cor.~4.6]{BalmerSanders17}, the support of $\triv_G(M(p))$ in $\SHG^c$ is $\SET{\cat P(H,p,n)}{H \le G,\, 2 \le n \le \infty} \subset \Spec(\SHGc)$. Finally, using $F_G:\SHGc \to \DHZG^c$, we can consider $F_G(\triv_G(M(p)))$ in $\DHZG^c$. By Corollary~\ref{cor:comp-with-SHG}, its support is $\SET{\cat P(H,p)}{H \le G}$. For $G=C_p$ this is precisely two points: 
		\[
			\supp(F_G(\triv_G(M(p))) = \{ \cat P(1,p), \cat P(C_p,p)\} \subset \Spec(\DHZG^c).
		\]
	We know from Lemma~\ref{lem:restrict-inclusions} that $\cat P(1,p)$ is a closed point and that $\cat P(G,p)$ is either a closed point or else $\overline{\{\cat P(G,p)\}} = \{\cat P(1,p),\cat P(G,p)\}$. We claim that the latter holds i.e.~that $\cat P(1,p)$ is contained in the closure of $\cat P(G,p)$. To this end, let
		\[
			e_G \to \unit \to f_G \to \Sigma e_G
		\]
	be the idempotent triangle in $\SH(G)$ associated to the Thomason closed subset $\supp(G_+) = \SET{ \cat P_G(1,p,n)}{\text{all } p, n}$ (that is, all the primes for the trivial subgroup). By \cite[Prop.~5.11]{BalmerSanders17}, $F_G(e_G) \to \unit \to F_G(f_G) \to \Sigma F_G(e_G)$ is the idempotent triangle in $\DHZG$ associated to $Y:=\supp(F_G(G_+)) = \SET{\cat P(1,\frakp)}{\text{ all } \frakp}$. Note that if $\cat P(1,p) \not\subseteq \cat P(G,p)$ then
		\[
			\supp(F_G(\triv_G(M(p)))) =\{\cat P(1,p)\}\sqcup \{\cat P(G,p)\} = \overline{\{\cat P(1,p)\}} \sqcup \overline{\{\cat P(G,p)\}}
		\]
	is a disjoint union of closed sets $Z_1\sqcup Z_2$ with $Z_1\subseteq Y$ and $Z_1 \cap Y = \emptyset$. Invoking Proposition~\ref{prop:idempotents-and-splitting} and letting $Z:=\supp(F_G(\triv_G(M(p))))=\{\cat P(1,p),\cat P(G,p)\}$, we see that $\cat P(1,p) \subseteq \cat P(G,p)$ if and only if $Z=\overline{\{\cat P(G,p)\}}$ if and only if $Z$ is irreducible if and only if $Z$ is connected if and only if the idempotent triangle
		\[
			F_G(e_G) \to \unit \to F_G(f_G) \to \Sigma F_G(e_G)
		\]
	does not split after tensoring with $F_G(\triv_G(M(p)))$. Suppose for a contradiction that it did split. Then passing back to $\SHG$ by applying $U_G$ and using that $U_G F_G = \triv_G(\HZ) \otimes - $ by \ref{it:trivHZ} and the projection formula \cite[Prop.~2.15]{BalmerDellAmbrogioSanders16}, it would follow that the sequence $e_G \to \unit \to f_G \to \Sigma e_G$ splits after tensoring with $\triv_G(\HZ) \otimes \triv_G(M(p)) \simeq \triv_G(\HZ \otimes M(p)) \simeq \triv_G(\HFp)$ which contradicts Proposition~\ref{prop:no-split-Cp}.
\end{proof}

\begin{Cor}\label{cor:p-group}
	Let $K$ be a subgroup of a finite $p$-group $G$. Then $\cat P_G(K,p) \subseteq \cat P_G(G,p)$ in $\DHZG$.
\end{Cor}

\begin{proof}
	As $G$ is a $p$-group, there is a subnormal tower ${K = K_0 \lnormal_p \cdots \lnormal_p K_t = G}$ where each subquotient has order $p$.  By Proposition~\ref{prop:Cp}, $\cat P_{K_i/K_{i-1}}(1,p)$ is contained in $\cat P_{K_i/K_{i-1}}(K_i/K_{i-1},p)$ for all $i=1,\ldots,t$.  Hence by Corollary~\ref{cor:quotient}, $\cat P_{K_i}(K_{i-1},p) \subseteq \cat P_{K_i}(K_i,p)$ for all $i=1,\ldots,t$. Hence $\cat P_G(K_{i-1},p) \subseteq \cat P_G(K_i,p)$ for all $i=1,\ldots,t$ by Lemma~\ref{lem:res}, so that $\cat P_G(K,p)\subseteq \cat P_G(G,p)$.
\end{proof}

\begin{Prop}\label{prop:K-in-G}
	If $K$ is a $p$-subnormal subgroup of a finite group $G$, then $\cat P_G(K,p) \subseteq \cat P_G(G,p)$ in $\DHZG$.
\end{Prop}

\begin{proof}
	The fact that $K$ is $p$-subnormal in $G$ implies that $\Op(G) \subseteq K$ (see \cite[Lem.~3.3]{BalmerSanders17}). By Corollary~\ref{cor:quotient}, $\cat P_G(K,p) \subseteq \cat P(G,p)$ if and only if 
		\[
			\cat P_{G/\Op(G)}(K/\Op(G),p) \subseteq \cat P_{G/\Op(G)}(G/\Op(G),p).
		\]
	As $G/\Op(G)$ is a $p$-group, the claim follows from Corollary~\ref{cor:p-group}.
\end{proof}

\begin{Prop}\label{prop:psubnormal-inclusion}
	Let $K,H$ be subgroups of a finite group $G$. If $K$ is $G$-conjugate to a $p$-subnormal subgroup of $H$ then $\cat P_G(K,p) \subseteq \cat P_G(H,p)$ in $\DHZG$.
\end{Prop}

\begin{proof}
	By assumption $K \sim_G K'$ where $K' \le H$ is $p$-subnormal. By Proposition~\ref{prop:K-in-G}, $\cat P_H(K',p) \subseteq \cat P_H(H,p)$ in $\DHZH$. It then follows from Lemma~\ref{lem:res} that $\cat P_G(K,p) = \cat P_G(K',p) \subseteq \cat P_G(H,p)$ in $\DHZG$.
\end{proof}

\begin{Thm}\label{thm:topology}
	Let $G$ be a finite group, let $K,H \le G$ be subgroups and let $\frakp, \frakq \in \Spec \bbZ$. Then $\cat P_G(K,\frakp) \subseteq \cat P_G(H,\frakq)$ if and only if either
	\begin{enumerate}
		\item $\frakp = (p)$, $K$ is $G$-conjugate to a $p$-subnormal subgroup of $H$, and $\frakq = (p)$ or $(0)$; or
		\item $\frakp = (0)$, $\frakq=(0)$ and $K \sim_G H$ (in which case the primes are equal).
	\end{enumerate}
\end{Thm}

\begin{proof}
	This follows from Lemma~\ref{lem:restrict-inclusions}, Remark~\ref{rem:top-strat} and Proposition~\ref{prop:psubnormal-inclusion}.
\end{proof}

\begin{Rem}\label{rem:irreducibles}
	The irreducible closed subsets of $\Spec(\DHZG^c)$ can thus be completely described as follows:
	\begin{itemize}
		\item $\overline{\{ \cat P(H,p) \}} = \SET{\cat P(K,p)}{K \text{ a $p$-subnormal subgroup of } H}$; and
		\item $\overline{\{ \cat P(H,0) \}} = \{\cat P(H,0) \} \cup \bigcup_{p \text{ prime}} \overline{\{ \cat P(H,p)\}}$.
	\end{itemize}
	These irreducible closed subsets completely determine the topology:
\end{Rem}

\begin{Prop}\label{prop:noetherian-space}
	The space $\Spec(\DHZG^c)$ is noetherian. Consequently, the closed subsets are precisely the finite unions of irreducible closed sets (equivalently, the closures of finite subsets). Moreover, the Thomason subsets are just the specialization closed subsets, that is, arbitrary unions of closed sets.
\end{Prop}

\begin{proof}
	By Theorem \ref{thm:spec-as-set}, the space $\Spec(\DHZG^c)$ is covered by the images of the continuous maps $\Spec(\phigeomb{H,G}):\Spec(\DHZ^c)\to\Spec(\DHZG^c)$ for $H \le G$.  The claim that $\Spec(\DHZG^c)$ is noetherian follows from the fact that $\Spec(\Der(\HZ)^c)\cong\Spec(\bbZ)$ is noetherian, that a continuous image of a noetherian space is noetherian, and that any space covered by finitely many noetherian subspaces is noetherian.  For the second claim just note that every subspace of a noetherian space is noetherian and that a noetherian space has finitely many irreducible components.  The description of the Thomason subsets is immediate from the definition as every subspace of a noetherian space is quasi-compact (cf.~\cite[Remark 12]{BalmerICM}).  All of this is standard: see \cite[\href{https://stacks.math.columbia.edu/tag/0050}{Section 0050}]{stacks-project} or \cite[\S 2.2]{EGA1}.
\end{proof}

\begin{Rem}\label{rem:complete-determination}
	Theorem~\ref{thm:spec-as-set} and Theorem~\ref{thm:topology} together with Proposition~\ref{prop:noetherian-space} thus provide a complete description of the topological space $\Spec(\DHZG^c)$ for any finite group $G$. We now explain the precise relationship, alluded to in the introduction, between the three spaces: $\Spec(\DHZG^c)$, $\Spec(\SH(G)^c)$, and $\Spec(A(G))$. It may be helpful to refer to the examples depicted on pages~\pageref{fig:SHCp} and \pageref{fig:ACp} in the introduction.
\end{Rem}

\begin{Cor}\label{cor:comparison-with-SHG}
	For any finite group $G$, the map
		\[
			\Spec(F_G):\Spec(\DHZG^c) \hookrightarrow \Spec(\SH(G)^c)
		\]
	is a homeomorphism of $\Spec(\DHZG^c)$ onto its image, which is the subspace of $\Spec(\SH(G)^c)$ consisting of the chromatic height 0 and chromatic height $\infty$ points.
\end{Cor}

\begin{proof}
	This follows from Cor.~\ref{cor:comp-with-SHG} and our descriptions of the two spaces; in particular, from Thm.~\ref{thm:spec-as-set}, Thm.~\ref{thm:topology}, \cite[Thm.~4.14]{BalmerSanders17}, \cite[Prop.~6.9]{BalmerSanders17}, and \cite[Cor.~8.4]{BalmerSanders17}.
\end{proof}

\begin{Rem}
	From the second point of view, both $\Spec(\DHZG^c)$ and $\Spec(A(G))$ consist of a number of copies of $\Spec(\bbZ)$, one for each conjugacy class of subgroups of $G$, except that in $\Spec(A(G))$ the closed points $\frak{p}(K,p)$ and $\frak{p}(H,p)$ are glued together when $\Op(K) \sim_G \Op(H)$.  Stated differently, each point $\frak{p}(H,p)$ is identified with $\frak{p}(\Op(H),p)$.  For example, if $G$ is a $p$-group then $\Op(H)$ is trivial for every $H \le G$, so all the copies of $(p)$ --- one for each copy of $\Spec(\bbZ)$ --- are glued into a single point.  In contrast, if $p$ does not divide $|G|$ then $\Op(H) = H$ for all $H \le G$, so no gluing of the copies of $(p)$ occurs.  The picture given of $\Spec(A(C_p))$ on page~\pageref{fig:ACp} is indicative of the situation for any $p$-group.

	This gluing $\frak{p}(H,p)=\frak{p}(\Op(H),p)$ in $\Spec(A(G))$ manifests in $\Spec(\DHZG^c)$ by the fact that $\cat P(\Op(H),p)$ is contained in the closure of $\cat P(H,p)$.  Or, rather, the gluing in $\Spec(A(G))$ is explained by these topological relations in $\Spec(\DHZG^c)$.
\end{Rem}

\begin{Cor}\label{cor:comparison-with-AG}
	For any finite group $G$, the comparison map
		\[
			\rho : \Spec(\DHZG^c) \to \Spec(A(G))
		\]
	is a quotient map which identifies points of height $\ge 1$ whose closures intersect. In more detail, if $\cat P, \cat Q \in \Spec(\DHZG^c)$ are distinct points then $\rho(\cat P)=\rho(\cat Q)$ if and only if $\cat P$ and $\cat Q$ are points of height $\ge 1$ with $\overline{\{\cat P\}} \cap \overline{\{\cat Q\}} \neq \emptyset$ if and only if $\cat P = \cat P(H,p)$ and $\cat Q=\cat P(K,p)$ for some prime number $p$ and subgroups $H,K\le G$ such that $H \cap K$ is a $p$-subnormal subgroup of both $H$ and $K$.
\end{Cor}

\begin{proof}
	The points of height 0 are precisely the points $\cat P(H,0)$ while the points of height $\ge 1$ are the points $\cat P(H,p)$.  (They can have height greater than 1 because of the inclusions among them.) According to Corollary~\ref{cor:comp-with-AG}, \cite[Theorem~3.6]{BalmerSanders17} and Theorem~\ref{thm:spec-as-set}, $\rho(\cat P(H,0)) \neq \rho(\cat P(K,p))$ for any $p$, while $\rho(\cat P(H,0)) = \rho(\cat P(K,0))$ iff $H \sim_G K$ iff $\cat P(H,0) = \cat P(K,0)$.  Also, $\rho(\cat P(H,p)) \neq \rho(\cat P(K,q))$ for $p \neq q$, while $\rho(\cat P(H,p)) = \rho(\cat P(K,p))$ iff $\Op(H) \sim_G \Op(K)$.  Thus the only identifications made by $\rho$ are for primes $\cat P(H,p)$ and $\cat P(K,p)$ corresponding to the same prime number~$p$, precisely when $\Op(H) \sim_G \Op(K)$.  By Remark~\ref{rem:irreducibles}, the closure of $\cat P(H,p)$ consists of those $\cat P(K,p)$ for $K$ a $p$-subnormal subgroup of $H$.  So if $\Op(H) \sim_G \Op(K)$ then $\cat P(\Op(H),p) = \cat P(\Op(K),p)$ is a point in the intersection $\overline{\{\cat P(H,p)\}}\cap\overline{\{\cat P(K,p)\}}$.  Conversely, if $\cat P=\cat P(H,p)$ and $\cat Q=\cat P(K,q)$ are height $\ge 1$ points with $\overline{\{\cat P\}}\cap \overline{\{\cat Q\}} \neq \emptyset$ then $p=q$ and there exists a subgroup $L$ which is $G$-conjugate to a $p$-subnormal subgroup of $H$ and also $G$-conjugate to a $p$-subnormal subgroup of $K$.  It follows that $\Op(H) \sim_G \Op(L) \sim_G \Op(K)$ and hence that $\rho(\cat P)=\rho(\cat Q)$.

	Also note that if $\Op(H) \sim_G \Op(K)$ then $\Op(H) = \Op(K^g)$ for some $g \in G$ and hence $H \cap K^g$ is a $p$-subnormal subgroup of both $H$ and $K^g$.  The converse also holds.  We conclude that $\rho(\cat P) = \rho(\cat Q)$ if and only if $\cat P = \cat P(H,p)$ and $\cat Q = \cat P(K,p)$ for some prime number $p$ and subgroups $H,K\le G$ satisfying $\Op(H)=\Op(K)$ (equivalently, with $H \cap K$ a $p$-subnormal subgroup of both $H$ and $K$).

	Finally, one readily checks using Remark~\ref{rem:irreducibles} that the surjective continuous map~$\rho$ is a closed map, hence a quotient map.
\end{proof}

\begin{Exa}\label{exa:D8}
	Consider $G=D_8$ the dihedral group of order $8$. Its lattice of conjugacy classes of subgroups is
		
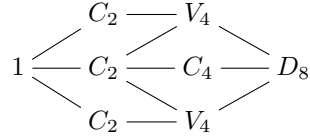
\[\xymatrix @=0.5pc{
			&&C_2 \ar@{-}[rr] && V_4 \ar@{-}[drr] \\
			1 \ar@{-}[urr] \ar@{-}[rr] \ar@{-}[drr] && C_2 \ar@{-}[rr] \ar@{-}[urr] \ar@{-}[drr] &&C_4 \ar@{-}[rr]&&D_8\\
			&&C_2  \ar@{-}[rr] && V_4 \ar@{-}[urr]&&
		}\]
	and the following diagram depicts the comparison map for the spectrum of $\Der(\HZ_{D_8})$ localized at the prime 2:
		
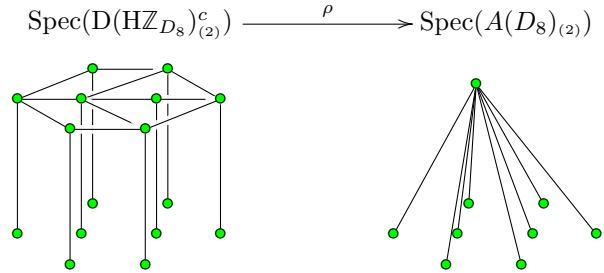
\[
\xy
(15,10)*{\Spec(\Der(\HZ_{D_8})_{\scriptscriptstyle (2)}^c)};
	{\ar@{->}^-{\rho} (30,10)*{};(52.5,10)*{}};
(65,10)*{\Spec(A(D_8)_{\scriptscriptstyle (2)})};
	{\ar@{-} (0,0)*{};(0,-18)*{}};
		{\ar@{-} (0,0)*{};(10,4)*{}};
		{\ar@{-} (0,0)*{};(7,-4)*{}};
		{\ar@{-} (10,4)*{};(20,4)*{}};
		{\ar@{-} (20,4)*{};(27,0)*{}};
	{\ar@{-} (10,4)*{};(10,-14)*{}};
	{\ar@{-} (8.5,0)*{};(8.5,-18)*{}};
	{\ar@{-} (7,-4)*{};(7,-22)*{}};
	{\ar@{-} (20,4)*{};(20,-14)*{}};
	{\ar@{-} (18.5,0)*{};(18.5,-18)*{}};
	{\ar@{-} (17,-4)*{};(17,-22)*{}};
	{\ar@{-} (27,0)*{};(27,-18)*{}};
		{\ar@[white] (8.5,0)*{};(20,4)*{}};
		{\ar@<-0.1mm>@[white] (8.5,0)*{};(20,4)*{}};
		{\ar@<-0.2mm>@[white] (8.5,0)*{};(20,4)*{}};
		{\ar@<-0.3mm>@[white] (8.5,0)*{};(20,4)*{}};
		{\ar@<-0.4mm>@[white] (8.5,0)*{};(20,4)*{}};
		{\ar@<-0.5mm>@[white] (8.5,0)*{};(20,4)*{}};
		{\ar@< 0.1mm>@[white] (8.5,0)*{};(20,4)*{}};
		{\ar@< 0.2mm>@[white] (8.5,0)*{};(20,4)*{}};
		{\ar@< 0.3mm>@[white] (8.5,0)*{};(20,4)*{}};
		{\ar@< 0.4mm>@[white] (8.5,0)*{};(20,4)*{}};
		{\ar@< 0.5mm>@[white] (8.5,0)*{};(20,4)*{}};
		{\ar@{-} (8.5,0)*{};(20,4)*{}};
		{\ar@{-} (8.5,0)*{};(18.5,0)*{}};
		{\ar@{-} (0,0)*{};(8.5,0)*{}};
		{\ar@[white] (8.5,0)*{};(17,-4)*{}};
		{\ar@<-0.1mm>@[white] (8.5,0)*{};(17,-4)*{}};
		{\ar@<-0.2mm>@[white] (8.5,0)*{};(17,-4)*{}};
		{\ar@<-0.3mm>@[white] (8.5,0)*{};(17,-4)*{}};
		{\ar@<-0.4mm>@[white] (8.5,0)*{};(17,-4)*{}};
		{\ar@<-0.5mm>@[white] (8.5,0)*{};(17,-4)*{}};
		{\ar@< 0.1mm>@[white] (8.5,0)*{};(17,-4)*{}};
		{\ar@< 0.2mm>@[white] (8.5,0)*{};(17,-4)*{}};
		{\ar@< 0.3mm>@[white] (8.5,0)*{};(17,-4)*{}};
		{\ar@< 0.4mm>@[white] (8.5,0)*{};(17,-4)*{}};
		{\ar@< 0.5mm>@[white] (8.5,0)*{};(17,-4)*{}};
		{\ar@{-} (8.5,0)*{};(17,-4)*{}};
		{\ar@[white] (17,0)*{};(25.5,0)*{}};
		{\ar@<-0.1mm>@[white] (17,0)*{};(25.5,0)*{}};
		{\ar@<-0.2mm>@[white] (17,0)*{};(25.5,0)*{}};
		{\ar@<-0.3mm>@[white] (17,0)*{};(25.5,0)*{}};
		{\ar@<-0.4mm>@[white] (17,0)*{};(25.5,0)*{}};
		{\ar@<-0.5mm>@[white] (17,0)*{};(25.5,0)*{}};
		{\ar@< 0.1mm>@[white] (17,0)*{};(25.5,0)*{}};
		{\ar@< 0.2mm>@[white] (17,0)*{};(25.5,0)*{}};
		{\ar@< 0.3mm>@[white] (17,0)*{};(25.5,0)*{}};
		{\ar@< 0.4mm>@[white] (17,0)*{};(25.5,0)*{}};
		{\ar@< 0.5mm>@[white] (17,0)*{};(25.5,0)*{}};
		{\ar@{-} (17,0)*{};(25.5,0)*{}};
		{\ar@[white] (17,-4)*{};(27,0)*{}};
		{\ar@<-0.1mm>@[white] (17,-4)*{};(27,0)*{}};
		{\ar@<-0.2mm>@[white] (17,-4)*{};(27,0)*{}};
		{\ar@<-0.3mm>@[white] (17,-4)*{};(27,0)*{}};
		{\ar@<-0.4mm>@[white] (17,-4)*{};(27,0)*{}};
		{\ar@<-0.5mm>@[white] (17,-4)*{};(27,0)*{}};
		{\ar@< 0.1mm>@[white] (17,-4)*{};(27,0)*{}};
		{\ar@< 0.2mm>@[white] (17,-4)*{};(27,0)*{}};
		{\ar@< 0.3mm>@[white] (17,-4)*{};(27,0)*{}};
		{\ar@< 0.4mm>@[white] (17,-4)*{};(27,0)*{}};
		{\ar@< 0.5mm>@[white] (17,-4)*{};(27,0)*{}};
		{\ar@{-} (17,-4)*{};(27,0)*{}};
		{\ar@[white] (7,-4)*{};(17,-4)*{}};
		{\ar@<-0.1mm>@[white] (7,-4)*{};(17,-4)*{}};
		{\ar@<-0.2mm>@[white] (7,-4)*{};(17,-4)*{}};
		{\ar@<-0.3mm>@[white] (7,-4)*{};(17,-4)*{}};
		{\ar@<-0.4mm>@[white] (7,-4)*{};(17,-4)*{}};
		{\ar@<-0.5mm>@[white] (7,-4)*{};(17,-4)*{}};
		{\ar@< 0.1mm>@[white] (7,-4)*{};(17,-4)*{}};
		{\ar@< 0.2mm>@[white] (7,-4)*{};(17,-4)*{}};
		{\ar@< 0.3mm>@[white] (7,-4)*{};(17,-4)*{}};
		{\ar@< 0.4mm>@[white] (7,-4)*{};(17,-4)*{}};
		{\ar@< 0.5mm>@[white] (7,-4)*{};(17,-4)*{}};
		{\ar@{-} (7,-4)*{};(17,-4)*{}};
(0,0)*{\color{green}\bullet};
(0,0)*{\circ};
(10,4)*{\color{green}\bullet};
(10,4)*{\circ};
(8.5,0)*{\color{green}\bullet};
(8.5,0)*{\circ};
(7,-4)*{\color{green}\bullet};
(7,-4)*{\circ};
(20,4)*{\color{green}\bullet};
(20,4)*{\circ};
(18.5,0)*{\color{green}\bullet};
(18.5,0)*{\circ};
(17,-4)*{\color{green}\bullet};
(17,-4)*{\circ};
(27,0)*{\color{green}\bullet};
(27,0)*{\circ};
(0,-18)*{\color{green}\bullet};
(0,-18)*{\circ};
(10,-14)*{\color{green}\bullet};
(10,-14)*{\circ};
(8.5,-18)*{\color{green}\bullet};
(8.5,-18)*{\circ};
(7,-22)*{\color{green}\bullet};
(7,-22)*{\circ};
(20,-14)*{\color{green}\bullet};
(20,-14)*{\circ};
(18.5,-18)*{\color{green}\bullet};
(18.5,-18)*{\circ};
(17,-22)*{\color{green}\bullet};
(17,-22)*{\circ};
(27,-18)*{\color{green}\bullet};
(27,-18)*{\circ};
	{\ar@{-} (61,2)*{};(50,-18)*{}};
	{\ar@{-} (61,2)*{};(60,-14)*{}};
	{\ar@{-} (61,2)*{};(58.5,-18)*{}};
	{\ar@{-} (61,2)*{};(57,-22)*{}};
	{\ar@{-} (61,2)*{};(70,-14)*{}};
	{\ar@{-} (61,2)*{};(68.5,-18)*{}};
	{\ar@{-} (61,2)*{};(67,-22)*{}};
	{\ar@{-} (61,2)*{};(77,-18)*{}};
%
%
(50,-18)*{\color{green}\bullet};
(50,-18)*{\circ};
(60,-14)*{\color{green}\bullet};
(60,-14)*{\circ};
(58.5,-18)*{\color{green}\bullet};
(58.5,-18)*{\circ};
(57,-22)*{\color{green}\bullet};
(57,-22)*{\circ};
(70,-14)*{\color{green}\bullet};
(70,-14)*{\circ};
(68.5,-18)*{\color{green}\bullet};
(68.5,-18)*{\circ};
(67,-22)*{\color{green}\bullet};
(67,-22)*{\circ};
(77,-18)*{\color{green}\bullet};
(77,-18)*{\circ};
(61,2)*{\color{green}\bullet};
(61,2)*{\circ}
\endxy

		\]
	Closure goes up and to the left. 
	For example, the closure of the point $\cat P(C_4,0)$ consists of four points:
		\[
			\overline{\{ \cat P(C_4,0) \}} = \{ \cat P(C_4,0), \cat P(C_4,2), \cat P(C_2,2), \cat P(\{1\},2) \}.
		\]
	Similarly, the closure of $\cat P(V_4,0)$ consists of five points (for either choice of $V_4$)
	and the closure of $\cat P(C_2,2)$ consists of two points (for any choice of $C_2$).
	Observe that there is a unique closed point in $\Spec(A(D_8)_{(2)})$ and that the fiber over this point is a copy of the lattice of conjugacy classes of subgroups of $D_8$. The reason is that for $p=2$, a subgroup $K$ is a $p$-subnormal subgroup of $H$ if and only if $K$ is a subgroup of $H$.  On the other hand, at a prime $p \neq 2$, there is no gluing in the Burnside ring and the comparison map is the simple bijection of eight copies of $\Spec(\bbZ_{(p)})$:
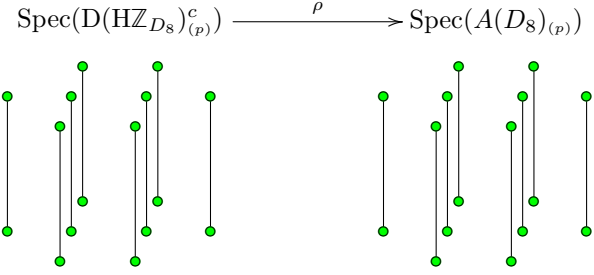
		\[
\xy
(15,10)*{\Spec(\Der(\HZ_{D_8})_{\scriptscriptstyle (p)}^c)};
	{\ar@{->}^-{\rho} (30,10)*{};(52.5,10)*{}};
(65,10)*{\Spec(A(D_8)_{\scriptscriptstyle (p)})};
	{\ar@{-} (0,0)*{};(0,-18)*{}};
	{\ar@{-} (10,4)*{};(10,-14)*{}};
	{\ar@{-} (7,-4)*{};(7,-22)*{}};
	{\ar@{-} (8.5,0)*{};(8.5,-18)*{}};
	{\ar@{-} (20,4)*{};(20,-14)*{}};
	{\ar@{-} (18.5,0)*{};(18.5,-18)*{}};
	{\ar@{-} (17,-4)*{};(17,-22)*{}};
	{\ar@{-} (27,0)*{};(27,-18)*{}};
	{\ar@{-} (50,0)*{};(50,-18)*{}};
	{\ar@{-} (60,4)*{};(60,-14)*{}};
	{\ar@{-} (58.5,0)*{};(58.5,-18)*{}};
	{\ar@{-} (57,-4)*{};(57,-22)*{}};
	{\ar@{-} (70,4)*{};(70,-14)*{}};
	{\ar@{-} (68.5,0)*{};(68.5,-18)*{}};
	{\ar@{-} (67,-4)*{};(67,-22)*{}};
	{\ar@{-} (77,0)*{};(77,-18)*{}};
(0,0)*{\color{green}\bullet};
(0,0)*{\circ};
(10,4)*{\color{green}\bullet};
(10,4)*{\circ};
(8.5,0)*{\color{green}\bullet};
(8.5,0)*{\circ};
(7,-4)*{\color{green}\bullet};
(7,-4)*{\circ};
(20,4)*{\color{green}\bullet};
(20,4)*{\circ};
(18.5,0)*{\color{green}\bullet};
(18.5,0)*{\circ};
(17,-4)*{\color{green}\bullet};
(17,-4)*{\circ};
(27,0)*{\color{green}\bullet};
(27,0)*{\circ};
(0,-18)*{\color{green}\bullet};
(0,-18)*{\circ};
(10,-14)*{\color{green}\bullet};
(10,-14)*{\circ};
(8.5,-18)*{\color{green}\bullet};
(8.5,-18)*{\circ};
(7,-22)*{\color{green}\bullet};
(7,-22)*{\circ};
(20,-14)*{\color{green}\bullet};
(20,-14)*{\circ};
(18.5,-18)*{\color{green}\bullet};
(18.5,-18)*{\circ};
(17,-22)*{\color{green}\bullet};
(17,-22)*{\circ};
(27,-18)*{\color{green}\bullet};
(27,-18)*{\circ};
(50,-18)*{\color{green}\bullet};
(50,-18)*{\circ};
(60,-14)*{\color{green}\bullet};
(60,-14)*{\circ};
(58.5,-18)*{\color{green}\bullet};
(58.5,-18)*{\circ};
(57,-22)*{\color{green}\bullet};
(57,-22)*{\circ};
(70,-14)*{\color{green}\bullet};
(70,-14)*{\circ};
(68.5,-18)*{\color{green}\bullet};
(68.5,-18)*{\circ};
(67,-22)*{\color{green}\bullet};
(67,-22)*{\circ};
(77,-18)*{\color{green}\bullet};
(77,-18)*{\circ};
(50,0)*{\color{green}\bullet};
(50,0)*{\circ};
(60,4)*{\color{green}\bullet};
(60,4)*{\circ};
(58.5,0)*{\color{green}\bullet};
(58.5,0)*{\circ};
(57,-4)*{\color{green}\bullet};
(57,-4)*{\circ};
(70,4)*{\color{green}\bullet};
(70,4)*{\circ};
(68.5,0)*{\color{green}\bullet};
(68.5,0)*{\circ};
(67,-4)*{\color{green}\bullet};
(67,-4)*{\circ};
(77,0)*{\color{green}\bullet};
(77,0)*{\circ};
\endxy

		\]
\end{Exa}

\begin{Rem}
	The above example for $G=D_8$ is illustrative of the situation for any $p$-group $G$.  Localized at a prime $q \neq p$ nothing interesting happens: both spaces are just a disjoint union of copies of $\Spec(\bbZ_{(q)})$.  Localized at the prime $p$, on the other hand, the spectrum of the Burnside ring has a single closed point and the fiber over that unique closed point is a copy of the lattice of conjugacy classes of subgroups of~$G$.  Again, the reason is that the relation ``is conjugate to a $p$-subnormal subgroup of'' reduces to ``is conjugate to a subgroup of''.  For example, a diligent reader can immediately write down the 2-local comparison map for the quaternion group $G=Q_8$ once they recall the lattice of (conjugacy classes of) subgroups of $Q_8$.
\end{Rem}

\begin{Exa}\label{exa:S3}
	Take $G=S_3$ the symmetric group on three symbols. Its lattice of conjugacy classes of subgroups is
		
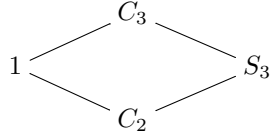
\[\xymatrix @=0.5pc{
			&& &C_3 \ar@{-}[drrr] \\
			1 \ar@{-}[urrr] \ar@{-}[drrr] &&&&&&S_3\\
			&&&C_2 \ar@{-}[urrr] &&  &&
		}\]
	and there are two primes of interest: $p=2,3$. Here is the 2-local comparison map 
		
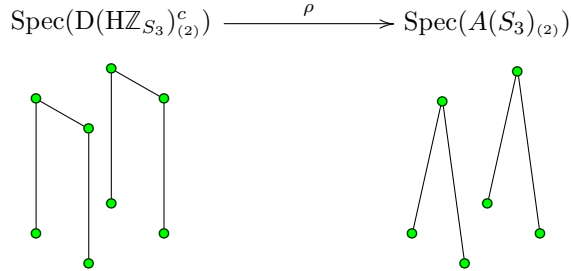
\[
\xy
(10,10)*{\Spec(\Der(\HZ_{S_3})_{\scriptscriptstyle (2)}^c)};
	{\ar@{->}^-{\rho} (25,10)*{};(47.5,10)*{}};
(60,10)*{\Spec(A(S_3)_{\scriptscriptstyle (2)})};
	{\ar@{-} (0,0)*{};(0,-18)*{}};
		{\ar@{-} (0,0)*{};(7,-4)*{}};
		{\ar@{-} (10,4)*{};(17,0)*{}};
	{\ar@{-} (10,4)*{};(10,-14)*{}};
	{\ar@{-} (7,-4)*{};(7,-22)*{}};
	{\ar@{-} (17,0)*{};(17,-18)*{}};
(0,0)*{\color{green}\bullet};
(0,0)*{\circ};
(10,4)*{\color{green}\bullet};
(10,4)*{\circ};
(7,-4)*{\color{green}\bullet};
(7,-4)*{\circ};
(17,0)*{\color{green}\bullet};
(17,0)*{\circ};
(0,-18)*{\color{green}\bullet};
(0,-18)*{\circ};
(10,-14)*{\color{green}\bullet};
(10,-14)*{\circ};
(7,-22)*{\color{green}\bullet};
(7,-22)*{\circ};
(17,-18)*{\color{green}\bullet};
(17,-18)*{\circ};
	{\ar@{-} (54,0)*{};(50,-18)*{}};
	{\ar@{-} (54,0)*{};(57,-22)*{}};
	{\ar@{-} (64,4)*{};(60,-14)*{}};
	{\ar@{-} (64,4)*{};(67,-18)*{}};
%
%
(50,-18)*{\color{green}\bullet};
(50,-18)*{\circ};
(60,-14)*{\color{green}\bullet};
(60,-14)*{\circ};
(57,-22)*{\color{green}\bullet};
(57,-22)*{\circ};
(67,-18)*{\color{green}\bullet};
(67,-18)*{\circ};
(64,3.5)*{\color{green}\bullet};
(64,3.5)*{\circ};
(54,-0.5)*{\color{green}\bullet};
(54,-0.5)*{\circ}
\endxy

		\]
	and here is the 3-local comparison map
		
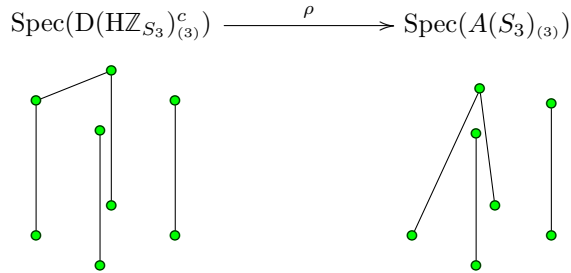
\[
\xy
(10,10)*{\Spec(\Der(\HZ_{S_3})_{\scriptscriptstyle (3)}^c)};
	{\ar@{->}^-{\rho} (25,10)*{};(47.5,10)*{}};
(60,10)*{\Spec(A(S_3)_{\scriptscriptstyle (3)})};
	{\ar@{-} (0,0)*{};(0,-18)*{}};
		{\ar@{-} (0,0)*{};(10,4)*{}};
	{\ar@{-} (10,4)*{};(10,-14)*{}};
	{\ar@{-} (8.5,-4)*{};(8.5,-22)*{}};
	{\ar@{-} (18.5,0)*{};(18.5,-18)*{}};
%
(0,0)*{\color{green}\bullet};
(0,0)*{\circ};
(10,4)*{\color{green}\bullet};
(10,4)*{\circ};
(8.5,-4)*{\color{green}\bullet};
(8.5,-4)*{\circ};
(18.5,0)*{\color{green}\bullet};
(18.5,0)*{\circ};
(0,-18)*{\color{green}\bullet};
(0,-18)*{\circ};
(10,-14)*{\color{green}\bullet};
(10,-14)*{\circ};
(8.5,-22)*{\color{green}\bullet};
(8.5,-22)*{\circ};
(18.5,-18)*{\color{green}\bullet};
(18.5,-18)*{\circ};
	{\ar@{-} (59,1.5)*{};(50,-18)*{}};
	{\ar@{-} (59,1.5)*{};(61,-14)*{}};
	{\ar@{-} (58.5,-4.5)*{};(58.5,-22)*{}};
	{\ar@{-} (68.5,-0.5)*{};(68.5,-18)*{}};
%
%
(50,-18)*{\color{green}\bullet};
(50,-18)*{\circ};
(61,-14)*{\color{green}\bullet};
(61,-14)*{\circ};
(58.5,-22)*{\color{green}\bullet};
(58.5,-22)*{\circ};
(68.5,-18)*{\color{green}\bullet};
(68.5,-18)*{\circ};
(59,1.5)*{\color{green}\bullet};
(59,1.5)*{\circ};
(68.5,-0.5)*{\color{green}\bullet};
(68.5,-0.5)*{\circ};
(58.5,-4.5)*{\color{green}\bullet};
(58.5,-4.5)*{\circ};
\endxy

		\]
	Note that although $C_2$ has index $3$ in $S_3$, it is not a $3$-subnormal subgroup (since it is not a normal subgroup). The reader might find it interesting to compare with Example~8.14 in \cite{BalmerSanders17}.
\end{Exa}

\begin{Rem}
	Finally, we can translate our computation of $\Spec(\DHZG^c)$ into a classification of the thick tensor-ideal subcategories of $\DHZG^c$:
\end{Rem}

\begin{Thm}\label{thm:classification}
	Let $G$ be a finite group. Say that a subset $Y \subseteq \Con(G) \times \Spec(\bbZ)$ is \emph{admissible} if it satisfies the following closure properties:
	\begin{enumerate}
		\item If $(H,0) \in Y$ then $(H,p) \in Y$ for all prime numbers $p$.
		\item If $(H,p) \in Y$ then $(K,p) \in Y$ for all (conjugacy classes of) $p$-subnormal subgroups $K$ of $H$.
	\end{enumerate}
	There is an inclusion-preserving bijection between the set of admissible subsets of $\Con(G)\times \Spec(\bbZ)$ and the collection of thick tensor-ideal subcategories of $\DHZG^c$ given by
	\[
		Y \longmapsto \SET{X \in \DHZG^c}{\phigeomb{H}(X) \in \frak p \text{ if } (H,\frak p) \not\in Y}
	\]
	and
	\[
		\cat C \longmapsto \SET{ (H,\frak p) \in \Con(G) \times \Spec(\bbZ)}{\phigeomb{H}(X) \not\in \frak p \text{ for some } X \in \cat C}.
	\]
\end{Thm}

\begin{proof}
	We have an identification of sets $\Spec(\DHZG^c) \cong \Con(G)\times \Spec(\bbZ)$ by Theorem~\ref{thm:spec-as-set}.  By Remark~\ref{rem:irreducibles}, the ``admissible'' subsets $Y \subseteq \Con(G) \times \Spec(\bbZ)$ are precisely those corresponding to specialization-closed subsets of $\Spec(\DHZG^c)$. Moreover, these are precisely the Thomason subsets by Proposition~\ref{prop:noetherian-space}. Finally, all thick tensor-ideals of $\DHZG^c$ are radical by \cite[Prop.~2.4]{Balmer07} since all objects are dualizable~\ref{it:DHZG}.  In this way, the theorem is just a translation of the abstract thick subcategory classification theorem of \cite[Theorem~4.10]{Balmer05a}.
\end{proof}

\section{Construction of \texorpdfstring{$\DHZG$}{D(HZG)} and equivariant spectra}\label{sec:highly-structured} 
Our next goal is to construct the tensor triangulated category $\DHZG$ and establish properties \ref{it:DHZG}--\ref{it:trivHZ} from Section~\ref{sec:computation}.

\begin{Def}
	For a finite group $G$, let $\SpG$ denote the $\infty$-category of genuine $G$-spectra as constructed in \cite[Appendix C]{GepnerMeier}.  It is presentable, stable, has a symmetric monoidal structure and comes equipped with a symmetric monoidal left adjoint
		\begin{equation}\label{eq:Gsuspension}
		\begin{tikzcd}
			\Sigma^{\infty} \colon \mathcal{S}_*^G \to \SpG
		\end{tikzcd}
		\end{equation}
	where $\mathcal{S}_*^G$ denotes the symmetric monoidal $\infty$-category of pointed $G$-spaces.  By Elmendorf's theorem, $\mathcal{S}_*^G$ is equivalent as a symmetric monoidal $\infty$-category to $\Fun(\mathcal{O}(G)\op,\mathcal{S}_*)$ with its pointwise monoidal structure, where $\mathcal{O}(G)$ denotes the usual orbit category of transitive $G$-sets and $G$-equivariant maps.
\end{Def}

\begin{Rem}
	The symmetric monoidal $\infty$-category $\SpG$ can be constructed in several different ways.  For example, it is equivalent to the underlying symmetric monoidal $\infty$-category \cite[\S1.3.4 and \S4.1.7]{HALurie} of the simplicial model category of orthogonal $G$-spectra $\Sp_G^O$ \cite{MandellMay02}.  In particular, its homotopy category $\SH(G):=\Ho(\SpG)$ is equivalent as a tensor triangulated category to the usual equivariant stable homotopy category of \cite{LewisMaySteinbergerMcClure86} and \cite[Appendix B]{HillHopkinsRavenel16}.  The approach taken by \cite[Appendix C]{GepnerMeier} instead defines $\SpG$ as the colimit of a diagram of copies of $\mathcal{S}_*^G$ parametrized by a certain poset of $G$-representations.  That this is equivalent to the other construction follows from \cite[Prop.~C.4 and Prop.~C.9]{GepnerMeier}.  The advantage of this definition is that it leads (following work of Robalo~\cite{Robalo15}) to a very convenient method for constructing symmetric monoidal functors on $\Sp^G$.  More precisely, as established in \cite[Corollary C.7]{GepnerMeier} and \cite[Corollary 2.22]{Robalo15}, the functor~\eqref{eq:Gsuspension} enjoys the following universal property:
\end{Rem}

\begin{Thm}[Robalo, Gepner--Meier] \label{thm:robalo}
	Let $G$ be a finite group and let $\rho_G$ denote the regular representation of $G$.  Given a presentably symmetric monoidal \mbox{$\infty$-category} $\cat D$ and a symmetric monoidal left adjoint $F:\mathcal{S}_*^G \to \cat D$ with the property that $F(S^{\rho_G})$ is invertible, there exists an essentially unique symmetric monoidal left adjoint $\overline{F}:\SpG \to \cat D$ such that $\overline{F} \circ \Sigma^\infty \simeq F$ as symmetric monoidal functors.
\end{Thm}

\begin{Rem}
	Here and in the sequel, when we say that a functor is essentially unique, we mean that the collection of such functors is parametrized by a contractible Kan complex.  Any two such choices will be equivalent (in a suitable $\infty$-category of functors) and will induce naturally isomorphic functors at the level of homotopy categories.
\end{Rem}

\begin{Rem}\label{rem:universal-Sp}
	We now recall the universal property of the stable $\infty$-category of nonequivariant spectra $\Sp$.  As explained in \cite[\S4.8.2]{HALurie}, the $\infty$-category~$\Prst$ of presentable stable $\infty$-categories and colimit preserving functors has a symmetric monoidal stucture whose commutative algebra objects are the presentably symmetric monoidal stable $\infty$-categories, that is, symmetric monoidal \mbox{$\infty$-categories} $\cat C^{\otimes}$ whose underlying $\infty$-category $\cat C$ is presentable and stable and has the property that the bifunctor $-\otimes-:\cat C \times \cat C \to \cat C$ commutes with small colimits in each variable.  The $\infty$-category of spectra $\Sp$ is the unit of $\Prst$ and consequently is the initial commutative algebra object in $\Prst$.  In other words, given any presentably symmetric monoidal stable $\infty$-category $\cat D^{\otimes} \in \CAlg(\Prst)$, there is an essentially unique symmetric monoidal functor $\Sp^{\otimes} \to \cat D^{\otimes}$ whose underlying functor $\Sp \to \cat D$ commutes with colimits \cite[Cor.~4.8.2.19]{HALurie}.
\end{Rem}

\begin{Exa}
	For any finite group $G$, there is an essentially unique symmetric monoidal functor
		\[
			\triv_G : \Sp \to \SpG
		\]
	which commutes with colimits.
\end{Exa}

\begin{Rem}\label{rem:module-induced}
	If $\cat C^{\otimes}$ is a presentably symmetric monoidal stable \mbox{$\infty$-category} then for any commutative algebra $A \in \CAlg(\cat C^{\otimes})$, the category of $A$-modules $A\MMod_{\cat C}^{\otimes}$ is also a presentably symmetric monoidal stable $\infty$-category \cite[Thm.~3.4.4.2]{HALurie}.  Moreover, the forgetful functor $U_A:A\MMod_{\cat C} \to \cat C$ has a left adjoint $F_A:\cat C \to A\MMod_{\cat C}$ which can be equipped with a symmetric monoidal structure, and the composite $U_AF_A$ is the functor $A\otimes -:\cat C \to \cat C$ (cf.~\cite[Cor.~4.2.3.7, Cor.~4.2.4.8, Thm.~4.5.2.1, \S 4.5.3]{HALurie}).  Furthermore, any symmetric monoidal functor $\theta:\cat C^{\otimes} \to \cat D^{\otimes}$ induces a functor $\CAlg(\cat C^{\otimes}) \to \CAlg(\cat D^{\otimes})$ between the $\infty$-categories of commutative algebra objects.  Moreover, for any $A \in \CAlg(\cat C^{\otimes})$ there is an induced symmetric monoidal functor $A\MMod_{\cat C}^\otimes \to \theta(A)\MMod_{\cat D}^\otimes$ for which both squares in
		\[
		\begin{tikzcd}
			\cat C \ar[r,"\theta"] \ar[d,shift right, "F"'] & \cat D \ar[d,shift right, "F"'] \\
			A\MMod_{\cat C} \ar[u,shift right,"U"'] \ar[r,"\theta"] & \theta(A)\MMod_{\cat D} \ar[u,shift right, "U"']
		\end{tikzcd}
		\]
	commute up to equivalence. Moreover, if $\theta : \cat C \to \cat D$ commutes with limits or colimits then so does the induced functor $\theta: A\MMod_{\cat C} \to \theta(A)\MMod_{\cat D}$ (cf.~\cite[Cor.~4.2.3.3 and Cor.~4.2.3.5]{HALurie}).
\end{Rem}

\begin{Def}\label{def:HZG}
	Let $\HZ \in \CAlg(\Sp)$ denote the Eilenberg--MacLane spectrum of the integers. For any finite group $G$, let $\HZ_G := \triv_G(\HZ) \in \CAlg(\SpG)$ and consider its $\infty$-category of modules $\HZ_G\MMod := \HZ_G\MMod_{\SpG}$. We define
		\[
			\DHZG := \Ho(\HZ_G\MMod)
		\]
	to be its homotopy category.
\end{Def}

\begin{Rem}\label{rem:free-forget-adj}
	By construction, we have a free-forgetful adjunction
		\[
		\begin{tikzcd}
			\SpG \ar[d, shift right, "F_G"'] \\
			\HZ_G\MMod \ar[u,shift right, "U_G"']
		\end{tikzcd}
		\]
	and $U_G(\unit) \cong \triv_G(\HZ)$ as commutative algebras in $\Sp^G$.
\end{Rem}

\begin{Rem}\label{rem:Amodules}
	If $\cat C^{\otimes} \in \CAlg(\Prst)$ is a presentably symmetric monoidal stable \mbox{$\infty$-category} then $\Ho(\cat C)$ has the structure of a triangulated category as well as a closed symmetric monoidal structure that is compatible with the triangulation (in the sense of \cite[Def.~A.2.1]{HoveyPalmieriStrickland97}).  Moreover, if $\cat C$ is compactly generated by a set of objects $\cat G$ then for any commutative algebra $A \in \CAlg(\cat C^{\otimes})$, the stable $\infty$-category $A\MMod_{\cat C}$ is compactly generated by the set $F_A(\cat G)$.  Note that an object of a stable $\infty$-category $\cat D$ is compact if and only if it is compact as an object of $\Ho(\cat D)$ in the usual triangulated category sense (see \cite[Prop.~1.4.4.1]{HALurie}).  Moreover, a set $\cat G$ of compact objects generates $\cat D$ under colimits if and only if $\cat G$ generates $\Ho(\cat D)$ in the usual triangulated category sense (see the proof of \cite[Cor.~1.4.4.2]{HALurie}).
\end{Rem}

\begin{Exa}\label{exa:generators}
	The stable $\infty$-category $\HZ_G\MMod$ is compactly generated by 
		\[
			\{ F_G(G/H_+) \mid H \le G\}.
		\]
	Moreover, since these compact generators are dualizable and the unit  $\unit$ is compact, it follows that an object of $\DHZG$ is dualizable if and only if it is compact (cf.~\cite[Theorem A.2.5]{HoveyPalmieriStrickland97}).  In other words, $\DHZG$ is a rigidly-compactly generated tensor triangulated category.
\end{Exa}

\begin{Rem}\label{rem:group-hom}
	Any homomorphism of groups $\alpha:G \to G'$ induces a left adjoint
		\begin{equation}\label{eq:group-hom-spaces}
			\alpha^*:\mathcal{S}_*^{G'} \to \mathcal{S}_*^G
		\end{equation}
	which can be constructed as follows. Restriction along $\alpha$ provides a functor from the category of $G'$-sets to the category of $G$-sets.  This functor always has a left adjoint which sends transitive $G$-sets to transitive $G'$-sets and hence restricts to a functor $\alpha_!:\mathcal{O}(G) \to \mathcal{O}(G')$ on the orbit categories.  Restricting along $\alpha_!$ provides the functor \eqref{eq:group-hom-spaces}.  Note that it preserves colimits hence is a left adjoint (since $\mathcal{S}_*$ is presentable).

	By Theorem~\ref{thm:robalo}, there is an essentially unique symmetric monoidal left adjoint
		\[
			\alpha^* : \Sp^{G'} \to \Sp^{G}
		\]
	such that the diagram
		\[
		\begin{tikzcd}
			\mathcal{S}_*^{G'} \ar[d,"\Sigma^\infty"'] \ar[r,"\alpha^*"] & \mathcal{S}_*^G \ar[d,"\Sigma^\infty"] \\
			\Sp^{G'} \ar[r,"\alpha^*"] & \Sp^G
		\end{tikzcd}
		\]
	commutes up to an equivalence of symmetric monoidal functors.

	By the universal property of Remark~\ref{rem:universal-Sp}, the composite
		\[
			\Sp \xra{\triv_{G'}} \Sp^{G'} \xra{\alpha^*} \Sp^G
		\]
	is equivalent as a symmetric monoidal functor to $\triv_G : \Sp \to \SpG$.  In particular, $\alpha^*(\HZ_{G'}) \cong \HZ_G$ as commutative algebras in $\SpG$.  Hence by Remark~\ref{rem:module-induced}, there is a symmetric monoidal left adjoint $\alpha^*:\HZ_{G'}\MMod \to \HZ_{G}\MMod$ such that both squares in 
		\[
		\begin{tikzcd}
			\Sp^{G'} \ar[d, shift right,"F_{G'}"'] \ar[r,"\alpha^*"]& \SpG \ar[d,shift right,"F_G"']\\
			\HZ_{G'}\MMod \ar[u,shift right,"U_{G'}"'] \ar[r,"\alpha^*"]& \HZ_G\MMod \ar[u, shift right,"U_G"']
		\end{tikzcd}
		\]
	commute up to equivalence.
\end{Rem}

\begin{Rem}\label{rem:composite-group-hom}
	If $G \xra{\alpha} G' \xra{\beta} G''$ is a composite of group homomorphisms, it follows from the constructions that the functor $(\beta \circ \alpha)^*$ is equivalent to $\alpha^* \circ \beta^*$ (in the relevant $\infty$-category of symmetric monoidal functors) at the level of spaces $\mathcal{S}_*^G$, spectra $\Sp^G$, and $\HZ_G\MMod$.
\end{Rem}

\begin{Rem}\label{rem:ind-name-adjoint}
	As usual, for a quotient $\alpha :G \to G/N$, we call $\infl_{G/N}^G := \alpha^*$ the inflation functor and for an inclusion $\alpha:H \hookrightarrow G$, we call $\res^G_H:=\alpha^*$ the restriction functor.  Since the functor $\res_H^G:\Sp^G \to \Sp^H$ preserves limits, the induced functor $\res_H^G:\HZ_G\MMod \to \HZ_H\MMod$ also preserves limits (Rem.~\ref{rem:module-induced}) and hence has a left adjoint (see \cite[Cor.~5.5.2.9 and Cor.~5.4.7.7]{HTTLurie}).
\end{Rem}

\begin{Rem}
	The smashing Bousfield localizations of a presentably symmetric monoidal stable $\infty$-category $\cat C$ correspond to the smashing Bousfield localizations of the tensor triangulated category $\Ho(\cat C)$.  More precisely, smashing Bousfield localizations correspond to idempotent commutative algebras (in $\cat C$ or $\Ho(\cat C)$ respectively) and the $\infty$-category of idempotent commutative algebras in $\cat C$ is equivalent to the ordinary category (in fact poset) of idempotent commutative algebras in~$\Ho(\cat C)$ (see \cite[\S 4.8.2]{HALurie} and \cite{BalmerFavi11}).
\end{Rem}

\begin{Prop}\label{prop:induced-localization}
	Let $\cat C^{\otimes} \in \CAlg(\Prst)$ be a presentably symmetric monoidal stable $\infty$-category and let $L:\cat C \to \cat C$ be a smashing Bousfield localization.  For any commutative algebra $A \in \CAlg(\cat C^{\otimes})$, the induced functor (Rem.~~\ref{rem:module-induced})
		\begin{equation}\label{eq:algebra-induced-localization}
			A\MMod_{\cat C} \to LA\MMod_{L\cat C}
		\end{equation}
	is a smashing Bousfield localization. Moreover, if $\cat C$ is rigidly-compactly generated and $L$ is the finite localization associated to a set $\cat G$ of compact objects in $\cat C$ then~\eqref{eq:algebra-induced-localization} is the finite localization associated to the set $F_A(\cat G)$ of compact objects in $A\MMod_{\cat C}$.
\end{Prop}

\begin{proof}
	Given two algebras $A,B \in \CAlg(\cat C)$, the extension-of-scalars $F_B : \cat C \to B\MMod_{\cat C}$ induces a functor $A\MMod_{\cat C} \to F_B(A)\MMod_{B\MMod_{\cat C}}$ by Remark~\ref{rem:module-induced} which under the equivalences
		\begin{align*}
			F_B(A)\MMod_{B\MMod_{\cat C}} &\simeq U_BF_B(A)\MMod_{\cat C} \\
			 &\simeq (B\otimes A)\MMod_{\cat C} \\
			 &\simeq (A\otimes B)\MMod_{\cat C} \\
			 &\simeq U_AF_A(B)\MMod_{\cat C}\\
			 &\simeq F_A(B)\MMod_{A\MMod_{\cat C}}
		\end{align*}
	is the extension-of-scalars $A\MMod_{\cat C} \to F_A(B)\MMod_{A\MMod_{\cat C}}$ associated to $F_A(B) \in \CAlg(A\MMod_{\cat C})$. (Nesting of module categories behaves as expected for symmetric monoidal $\infty$-categories: see \cite[\S3.4.1 and \S3.2.4]{HALurie}.) Next recall from \cite[Prop.~4.8.2.10]{HALurie} that up to symmetric monoidal equivalence, the smashing Bousfield localization $L:\cat C \to L\cat C$ is nothing but extension-of-scalars $\cat C \to L\unit\MMod_{\cat C}$ with respect to the idempotent algebra~$L\unit \in \CAlg(\cat C)$. Taking $B=L\unit$ above, we find that the the functor $A\MMod_{\cat C} \to LA\MMod_{L\cat C}$ induced by $L:\cat C \to L \cat C$ is the smashing Bousfield localization associated to the idempotent algebra $F_A(L\unit) \in \CAlg(A\MMod_{\cat C})$.

	To prove the second part, it suffices to check (at the level of homotopy categories) that $F_A(L\unit)$ is the smashing idempotent associated to the indicated finite localization.  This is the content of \cite[Proposition 5.11]{BalmerSanders17}.
\end{proof}

\begin{Rem}
	The following useful proposition is a variation on ideas that have appeared in  a few different places (e.g.~\cite[\S5.3]{MathewNaumannNoel17}).
\end{Rem}

\begin{Prop}\label{prop:monadicequivalence}
	Let $F:\cat C \to \cat D$ be a symmetric monoidal functor between presentably symmetric monoidal stable $\infty$-categories which admits a right adjoint~$G$.  If~$\cat C$ is generated by a set of compact-rigid objects then the functor $F$ is an equivalence if and only if 
		\begin{enumerate}
			\item the functor $F$ sends a set of compact generators of $\cat C$ to a set of compact generators of $\cat D$, and
			\item the commutative algebra $G(\unit_{\cat D})$ is equivalent to the commutative algebra $\unit_{\cat C}$.
		\end{enumerate}
\end{Prop}

\begin{proof}
	The $(\Rightarrow)$ direction is immediate once we recognize that the unit $\eta:\unit_{\cat C} \to GF(\unit_{\cat C})\simeq G(\unit_{\cat D})$ is a map of algebras.  On the other hand, hypothesis (a) implies that the right adjoint $G$ preserves colimits and is, moreover, conservative.  Hence the adjunction $F \dashv G$ is monadic by the Barr--Beck--Lurie Theorem \cite[Theorem~4.7.3.5]{HALurie}.  Now the projection formula $G(x)\otimes y \to G(x \otimes F(y))$ is an equivalence under our assumptions (as can be checked at the level of homotopy categories \cite[Prop.~2.15]{BalmerDellAmbrogioSanders16}) and the natural equivalence $G(\unit)\otimes y \simeq GF(y)$ provides an isomorphism between the monad associated to the algebra object $G(\unit)\in\CAlg(\cat C)$ and the monad~$GF$ of the adjunction (see the proof of \cite[Lemma~2.8]{BalmerDellAmbrogioSanders15}).  The functor~$F$ is thus, up to equivalence, just extension of scalars $\cat C \to G(\unit)\MMod_{\cat C}$ with respect to the algebra $G(\unit)$. This is an equivalence by the second hypothesis.
\end{proof}

\begin{Prop}\label{prop:HZGfiniteloc}
	Let $N \lenormal G$ be a normal subgroup and let 
		\[
			L:\HZ_G\MMod \to \HZ_G\MMod
		\]
	be the finite localization associated to the set $\SET{F_G(G/H_+)}{{H \not\supseteq N}}$. The composite
		\begin{equation}\label{eq:HZGcomp}
			\HZ_{G/N}\MMod \xra{\infl_{G/N}^G} \HZ_G\MMod \xra{L} L(\HZ_G\MMod)
		\end{equation}
	is an equivalence of symmetric monoidal $\infty$-categories.
\end{Prop}

\begin{proof}
	We first establish the analogous statement for $\Sp^{G/N}$ namely that if $L:\Sp^G \to \Sp^G$ is the finite localization associated to the set $\SET{G/H_+}{H \not\supseteq N}$ then the composite
		\begin{equation}\label{eq:SpGcomp}
			\Sp^{G/N} \xra{\infl_{G/N}^G} \Sp^G\xra{L} L(\Sp^G)
		\end{equation}
	is an equivalence of symmetric monoidal $\infty$-categories.  This is well-known at the level of homotopy categories (see \cite[Cor.~II.9.6]{LewisMaySteinbergerMcClure86}) but we will provide a proof that holds at the level of $\infty$-categories.  For simplicity of notation, let $f^*$ denote the composite \eqref{eq:SpGcomp} and let $f_*$ denote a right adjoint.  The functor $f^*$ has a symmetric monoidal structure (being a composite of such) and we just need to establish that it is an equivalence.  First note (e.g.~by Remark~\ref{rem:Amodules}) that the smashing localization~$L$ maps the set of compact generators $\SET{G/H_+}{H \le G}$ of $\Sp^G$ to a set of compact generators for $L(\Sp^G)$ which is thus compactly generated by $\SET{L(G/H_+)}{H \supseteq N}$.  Note that these are precisely the images under $f^*$ of the generators of $\Sp^{G/N}$.  We can thus invoke Proposition~\ref{prop:monadicequivalence} if we prove that the homomorphism of algebras $\Sphere_{G/N}=\unit \to f_*f^*\unit = (\widetilde{E\cat F}[{\not\supseteq}N])^N$ is an equivalence.  This can be checked in the homotopy category, where it is well-known (see \cite[Prop.~II.9.10.(ii)]{LewisMaySteinbergerMcClure86} or the proof of \cite[Thm.~6.11]{MathewNaumannNoel17}).

	The result now follows from Proposition~\ref{prop:induced-localization}.  Indeed, \eqref{eq:HZGcomp} is obtained from \eqref{eq:SpGcomp} by taking $A:= \HZ_{G/N}\in \CAlg(\Sp^{G/N})$ and invoking Remark~\ref{rem:module-induced} and Proposition~\ref{prop:induced-localization} (for the second functor).  Moreover, note (either directly or as the special case of Proposition~\ref{prop:induced-localization} correponding to a trivial finite localization) that if $\theta:\cat C \to \cat D$ is a symmetric monoidal equivalence, then the induced functor $A\MMod_{\cat C} \to \theta(A)\MMod_{\cat D}$ is a symmetric monoidal equivalence.  Thus, the fact that \eqref{eq:SpGcomp} is an equivalence implies that \eqref{eq:HZGcomp} is an equivalence.
\end{proof}

\begin{Rem}\label{rem:conjugate}
	For any $g \in G$, consider the inner automorphism $\alpha:=c_g = (-)^g : G \xra{\sim} G$.  The induced automorphism of the category of $G$-sets is naturally isomorphic to the identity functor.  Consequently, the induced functor $\alpha^*:\mathcal{S}_*^G \to \mathcal{S}_*^G$ of Remark~\ref{rem:group-hom} is equivalent (as a symmetric monoidal functor) to the identity functor.  It follows that the same is true for the induced functors $c_g^* : \SpG\to\SpG$ and $c_g^* : \HZ_G\MMod \to \HZ_G\MMod$.
\end{Rem}

\begin{Rem}
	To summarize, we have established \ref{it:DHZG} through \ref{it:trivHZ} as follows:
	\ref{it:DHZG} by Definition \ref{def:HZG}, Remark~\ref{rem:free-forget-adj} and Example \ref{exa:generators};
	\ref{it:trivial} by construction (Definition \ref{def:HZG});
	\ref{it:change-of-group} by Remark~\ref{rem:group-hom};
	\ref{it:composite-of-maps} by Remark~\ref{rem:composite-group-hom};
	\ref{it:adjoint} by Remark~\ref{rem:ind-name-adjoint};
	\ref{it:finite-loc} by Proposition~\ref{prop:HZGfiniteloc};
	\ref{it:phiHG} is just a definition;
	\ref{it:conjugacy} by Remark~\ref{rem:conjugate};
	and \ref{it:trivHZ} by Remark~\ref{rem:free-forget-adj}.
\end{Rem}

\begin{Rem}
	The approach we have taken is not the only way to construct the tensor triangulated category $\DHZG$.  An alternative approach is to consider the inflation functor $\epsilon_G^*:\Sp^O \to \Sp_G^O$ in the context of orthogonal $G$-spectra.  We can take a cofibrant replacement $\HZ^c \to \HZ$ of associative algebras with respect to the model structure given in \cite[Section III.7]{MandellMay02}.  The homotopy category $\Ho(\epsilon_G^*\HZ^c\MMod)$ is then a model for $\DHZG$. However, there are subtleties concerning the commutative structures on $\epsilon_G^* \HZ^c$.  One can show that $\epsilon_G^*\HZ^c$ is not a genuine $G$-$E_\infty$-ring spectrum (see Example \ref{exa: not genuine} below).  It does however have a naive $E_\infty$-structure (see \cite{Blumberg-Hill}) and this is good enough to produce the symmetric monoidal structure on $\DHZG$.  Alternatively one can do the same thing using the stable model category $\Sp_{\Sigma}^G$ of $G$-equivariant symmetric spectra of \cite{Hausmann} based on simplicial sets.  The latter is monoidally Quillen equivalent to $\Sp_G^O$ but has the advantage that it is combinatorial.  It turns out however that making this all work brings forth a number of point-set level technicalities and some of the delicate issues related to model categories of modules over naive equivariant $E_\infty$-ring spectra that are still not covered in the literature.  Instead we have chosen to construct $\DHZG$ using simple universal properties and stable $\infty$-categories.
\end{Rem}

\begin{Exa} \label{exa: not genuine}
	We show that $\epsilon_G^*\HZ^c$ for $G=C_2$ cannot be modelled by a genuine $C_2$-$E_\infty$-ring spectrum.  Indeed, if this were the case, then $\epsilon_{C_2}^*\HZ^c$ would admit a strictly commutative orthogonal spectrum model as a $C_2$-equivariant ring spectrum.  We would then have a commutative diagram in the homotopy category of spectra
		\[
		\begin{tikzcd}
			\HZ \ar[d,equals] \ar[r,"\Delta"] & \Phi^{C_2}(N \epsilon_{C_2}^*\HZ^c) \ar[r] \ar[d] & \Phi^{C_2}(\epsilon_{C_2}^*\HZ^c) \cong \HZ \ar[d] \\
			\HZ \ar[r,"\Delta"] & (N \epsilon_{C_2}^*\HZ^c)^{tC_2} \ar[r] & \HZ^{tC_2}
		\end{tikzcd}
		\]
	where $N$ is the Hill--Hopkins--Ravenel norm, the top $\Delta$ is the Hill--Hopkins--Ravenel diagonal \cite[Proposition B.209]{HillHopkinsRavenel16} and $(-)^{tC_2}$ is the Tate construction \cite{GreenleesMay95b,NikSch}.  The lower $\Delta$ is the Tate diagonal of \cite{NikSch}. It follows from \cite[Theorem IV.1.15]{NikSch} that the lower horizontal composite (which is the Tate valued Frobenius) splits as a sum containing all even Steenrod squares if we use the splitting of $\HZ^{tC_2}$ coming from the canonical $\HZ$-module structure. On the other hand, with respect to the same splitting, the right-hand vertical map is the inclusion of a summand and hence does not contain any non-trivial Steenrod squares. This gives a contradiction. 
\end{Exa}

\section{Equivalence with the categories of Kaledin and Barwick}\label{sec:equivalence} 
Our goal in this section is to prove that the category $\DHZG$ constructed in Section~\ref{sec:highly-structured} and whose spectrum was computed in Section~\ref{sec:computation} is equivalent to Kaledin's category of derived Mackey functors \cite{Kaledin11}.  We will achieve this in two steps by passing first through the category of spectral Mackey functors \cite{Barwick17,BGS20}.

\subsection*{Spectral Mackey Functors}
Recall the \emph{effective Burnside $\infty$-category} $\A(\Fin_G)$ introduced by \cite{Barwick17}. Its objects are finite $G$-sets and its \mbox{$n$-simplices} are \mbox{$n$-fold} spans of finite $G$-sets (see Definition \ref{def: effective Burnside category} below).  It is a semi-additive \mbox{$\infty$-category} with biproduct given by the disjoint union and with a symmetric monoidal structure provided by \cite[Section 2]{BGS20}. A \emph{spectral Mackey functor} is an additive functor $\A(\Fin_G) \to \Sp$.  The $\infty$-category of spectral Mackey functors
	\[
		\Fun_{\add}(\A(\Fin_G), \Sp)
	\] 
is a smashing localization of the functor category $\Fun(\A(\Fin_G),\Sp)$. As such, it is stable, presentable and can be equipped with the localized Day convolution product (see \cite{GlasDay} and \cite[Lemma~3.7]{BGS20}).  It also comes equipped with a symmetric monoidal left adjoint 
	\[
		\Sigma^{\infty} \colon \mathcal{S}_*^G \to \Fun_{\add}(\A(\Fin_G), \Sp)
	\]
which inverts representation spheres (see \cite[Proposition~A.10]{ClausenMathewNaumannNoel20pp}
and \cite[Appendix A]{Nardinthesis}) and whose right adjoint $\Omega^{\infty}$ is induced by the inclusion
	\[
		\mathcal{O}(G)^{op} \to \A(\Fin_G).
	\] 

\begin{Prop}\label{prop:universal-to-spectral-mackey}
	For any finite group $G$, the essentially unique colimit-preserving symmetric monoidal functor
		\begin{equation}\label{eq:spectralcomp}
			\bbF \colon \Sp^G \to \Fun_{\add}(\A(\Fin_G),\Sp)
		\end{equation}
	which commutes with $\Sigma^\infty$ is an equivalence of symmetric monoidal $\infty$-categories.
\end{Prop}

\begin{proof} 
	The existence and essential uniqueness of the functor $\bbF$ follows from the universal property of Theorem~\ref{thm:robalo}.  To prove the claim it suffices to check that~$\bbF$ is an equivalence of underlying $\infty$-categories.  Moreover, since $\bbF$ is an exact functor between stable $\infty$-categories, it suffices to check that the induced functor on homotopy categories
		\[
			\bbF \colon \Ho(\Sp^G) \to \Ho(\Fun_{\add}(\A(\Fin_G),\Sp))
		\]
	is an equivalence. Note that both homotopy categories are equivalent as triangulated categories to the homotopy category of orthogonal $G$-spectra. For $\Sp^G$ this is established in \cite[Appendix C]{GepnerMeier} while for spectral Mackey functors it goes back to \cite{GuillouMay17pp}.  We do not need the full strength of these results (just certain facts about the homotopy categories, such as the fact that they are both generated by suspension spectra of orbits) but morally what the following proof really does is establish that any system of endofunctors on the equivariant stable homotopy categories satisfying certain compatibility properties must be an equivalence. It is a rigidity result. With these comments in mind, let us continue.

	Note that the tensor triangulated category $\Ho(\Sp^G)$ is rigidly-compactly generated by 
		\[
			\{ \Sigma^{\infty}(G/H_+) \mid H \le G\}
		\]
	and, by construction, the universal functor $\bbF$ commutes with suspension spectra:
		\[
			\bbF(\Sigma^{\infty}(G/H_+)) \simeq \Sigma^{\infty}(G/H_+).
		\]
	Moreover, these suspension spectra rigidly-compactly generate the tensor triangulated category of spectral Mackey functors $\Ho(\Fun_{\add}(\A(\Fin_G),\Sp))$ by \cite[Lemma A.8]{Nardinthesis}. Since $\bbF$ preserves coproducts, a thick subcategory argument reduces the problem of showing that $\bbF$ is an equivalence to the problem of showing that
		\begin{equation}\label{eq:fullyfaithful}
			\bbF \colon [\Sigma^{\infty}(G/H_+), \Sigma^{\infty}(G/K_+)]_*^G \to  [\bbF(\Sigma^{\infty}(G/H_+)), \bbF(\Sigma^{\infty}(G/K_+))]_*^G 
		\end{equation}
	is an isomorphism for all $H, K \leq G$. Here we use the notation $[-,-]_*^G$ to denote graded morphisms in the homotopy categories.

	The degree zero morphisms $[\Sigma^{\infty}(G/H_+), \Sigma^{\infty}(G/K_+)]_0^G$ are in both cases generated by conjugations, restrictions and transfers subject to the same relations.  The functor $\bbF$ preserves conjugations and restrictions since they come from morphisms of $G$-spaces and $\bbF \Sigma^{\infty}\simeq\Sigma^{\infty}$.  Transfers are preserved since transfers are duals of restrictions and $\bbF$ commutes with duality (being a symmetric monoidal functor).  Hence we conclude that
		\[
			\bbF \colon [\Sigma^{\infty}(G/H_+), \Sigma^{\infty}(G/K_+)]_0^G \to  [\bbF(\Sigma^{\infty}(G/H_+)), \bbF(\Sigma^{\infty}(G/K_+))]_0^G
		\]
	is an isomorphism for any two subgroups $H$ and $K$.

	Now, to prove that \eqref{eq:fullyfaithful} is an isomorphism in general, it suffices by \cite[Proposition 5.1.1]{Pat16} to establish the $H=K$ case. For this, we use induction on the order of $H$.  The fact that $\Sp$ is the free stable $\infty$-category on one generator \cite[Corollary 1.4.4.6]{HALurie} implies that the diagram
		\[
		\begin{tikzcd}
			\Sp \ar[r,"\bbF"] \ar[d,"G_+ \wedge -"'] &  \Fun_{\add}(\A(\Fin_1),\Sp) \ar[d,"G_+ \wedge -"] \\
			\SpG \ar[r,"\bbF"] &  \Fun_{\add}(\A(\Fin_G),\Sp)
		\end{tikzcd}
		\]
	commutes up to equivalence and, moreover, that the top arrow is an equivalence. Furthermore, for both homotopy categories we have an isomorphism
		\[
			[\Sigma^{\infty}G_+, \Sigma^{\infty}G_+]_*^G \cong [\mathbb{S}, \mathbb{S}]_* \otimes [\Sigma^{\infty}G_+, \Sigma^{\infty}G_+]_0^G 
		\]
	induced by the functor $G_+\wedge-$. Combining the latter two results we conclude that
		\[
			\bbF \colon [\Sigma^{\infty}G_+, \Sigma^{\infty}G_+]_*^G \to  [\bbF(\Sigma^{\infty}G_+), \bbF(\Sigma^{\infty}G_+)]_*^G
		\]
	is an isomorphism. 

	To do the induction step we need the geometric fixed point functors. For $\Sp^G$ they are defined in the proof of Proposition~\ref{prop:HZGfiniteloc} while for the category of spectral Mackey functors $\Fun_{\add}(\A(\Fin_G),\Sp)$ they are defined in \cite[Example A.20]{Nardin16pp}. Both are smashing localizations, hence are symmetric monoidal, and the essential uniqueness of Theorem \ref{thm:robalo} implies that $\bbF \phigeom{N,G} \simeq \phigeom{N,G} \bbF$. For a subgroup $H \le G$ with normalizer $N(H)$ and Weyl group $W(H):=N(H)/H$, let $\phigeom{H,G}$ denote the composite $\phigeom{H,N(H)} \circ \res^G_{N(H)}$. We will use the notation $E\mathcal{P}(H)$ for the classifying space of the family of proper subgroups of $H$. Then consider the following diagram (in which we have suppressed the $\Sigma^{\infty}$ symbols): 

	\newsavebox{\tempdiagramm}
	\begin{lrbox}{\tempdiagramm}
	\(\begin{tikzcd}
		{[G/H_+, G \times_H E\mathcal{P}(H)_+]_*^G} \ar[d,"\bbF"] \ar[r,hook,"\proj_*"] &  {[G/H_+, G/H_+]_*^G} \ar[d,"\bbF"]  \ar[r, twoheadrightarrow, "\phigeom{H,G}"] & {[W(H)_+, W(H)_+]_*^{W(H)}} \ar[d,"\bbF"]  \\ 
		{[\bbF(G/H_+), \bbF(G \times_H E\mathcal{P}(H)_+)]_*^G} \ar[r,hook,"\bbF(\proj)_*"] &  {[\bbF(G/H_+), \bbF(G/H_+)]_*^G} \ar[r, twoheadrightarrow, "\phigeom{H,G}"] & {[\bbF(W(H)_+), \bbF(W(H)_+)]_*^{W(H)}.} 
	\end{tikzcd}\) 
	\end{lrbox}

	\smallskip
	\noindent\resizebox{\linewidth}{!}{\usebox{\tempdiagramm}}
	\smallskip

	\noindent
	It follows from \cite[Proposition 6.3.2]{Pat16} that the top row is a short exact sequence. Since $\bbF$ commutes with suspension spectra, so is the bottom row (as it is isomorphic to the analog of the top row in the homotopy category of spectral Mackey functors).  The left square commutes by the functoriality of $\bbF$ and the right square commutes by the fact that~$\bbF$ commutes with geometric fixed points.  Since $G \times_H E\mathcal{P}(H)$ has smaller isotropy than $H$ (see~\cite[Proof of Lemma 7.2.2]{Pat16}), the induction hypothesis implies that the left-hand map is an isomorphism.  On the other hand, the right-hand map is an isomorphism by the base case of the induction.  It follows that the middle map is also an isomorphism and this completes the proof.
\end{proof}

\begin{Rem}\label{rem:equiv-with-orthogonal}
	Combined with \cite[Proposition C.9]{GepnerMeier}, Proposition~\ref{prop:universal-to-spectral-mackey} provides a symmetric monoidal equivalence between the symmetric monoidal $\infty$-category of orthogonal $G$-spectra and the symmetric monoidal \mbox{$\infty$-category} of spectral \mbox{$G$-Mackey} functors.  Such an equivalence is also established by different methods in \cite[Appendix A]{ClausenMathewNaumannNoel20pp}.
\end{Rem}

\begin{Rem}
	Evaluation at the object $G/G \in \A(\Fin_G)$ admits a symmetric monoidal left adjoint
		\begin{equation*}
			F_1 \colon \Sp \to \Fun(\A(\Fin_G), \Sp)
		\end{equation*}
	whose composition with the localization
		\[
			L:\Fun(\A(\Fin_G),\Sp) \to \Funadd(\A(\Fin_G),\Sp)
		\]
	is the essentially unique symmetric monoidal left adjoint $\LF_1$ making the diagram
		\[
		\begin{tikzcd}
			 \mathcal{S}_* \ar[d,"\triv_G"']  \ar[r,"\Sigma^{\infty}"] &   \Sp \ar[d,"\LF_1"] \\ \mathcal{S}_*^G \ar[r,"\Sigma^{\infty}"] & \Fun_{\add}(\A(\Fin_G), \Sp) 
		\end{tikzcd}
		\]
	commute up to an equivalence of symmetric monoidal functors. This follows by taking right adjoints and noting that the inclusion functor $\{G/G\} \to \A(\Fin_G)$ factors through the  nerve of $\mathcal{O}(G)\op$. Moreover the universal property also implies that $\LF_1$ coincides with the composite
		\[
			\Sp \xra{\triv_G} \Sp^G \xra{\simeq} \Funadd(\A(\Fin_G),\Sp)
		\]
	up to equivalence of symmetric monoidal functors. An immediate consequence is:
\end{Rem}

\begin{Cor} \label{cor: DHZ is F1Zmod}
	There is an equivalence of symmetric monoidal $\infty$-categories
		\[
			\HZ_G\MMod \simeq \LF_1\HZ\MMod.
		\]
\end{Cor}

\begin{Rem}
	Our next task is to compare $\LF_1\HZ\MMod$ with $\HZ$-valued spectral Mackey functors: $\Funadd(\A(\Fin_G),\HZ\MMod)$.
\end{Rem}

\begin{Prop} \label{prop: general on modules} 
	Let $\cat C^{\otimes}$ and $\cat D^{\otimes}$ be symmetric monoidal $\infty$-categories with $\cat C$ small, $\cat D$ presentable and $-\otimes- \colon \cat D \times \cat D \to \cat D$ preserving colimits in each variable. Equip $\Fun(\cat C, \cat D)$ with the Day convolution.
	\begin{enumerate}
		\item	The left adjoint $F_\unit : \cat D \to \Fun(\cat C, \cat D)$ to evaluation at the unit $\unit \in \cat D$ has a symmetric monoidal structure.
		\item	Let $A \in \CAlg(\cat D)$ be a commutative algebra. There exists an essentially unique symmetric monoidal equivalence
					\[
						\Fun(\cat C, A\MMod_{\cat D}) \simeq {F_\unit A}\MMod_{\Fun(\cat C,\cat D)}
					\]
				where the left-hand side is equipped with the Day convolution. Moreover, this equivalence commutes with the free-forgetful adjunctions to $\Fun(\cat C,\cat D)$.
\end{enumerate}
\end{Prop}

\begin{proof}
	Part (a) is established by \cite[Corollary 3.8]{Niko16pp}. To prove part (b) note that the free-forgetful adjunction $F : \cat D \adjto A\MMod_{\cat D} : U$ induces an adjunction on functor categories
		\[
			F_* \colon \Fun(\cat C, \cat D)  \adjto \Fun(\cat C, A\MMod_{\cat D}) : U_*
		\]
	by post-composition. We then apply the monadic machinery used in the proof of Proposition~\ref{prop:monadicequivalence} (cf.~\cite[Corollary 4.8.5.21]{HALurie}).  The Day convolution product preserves colimits in each variable by \cite[Lemma 2.13]{GlasDay} and $F_*$ is symmetric monoidal by \cite[Corollary 3.7]{Niko16pp}.  Since equivalences and colimits in functor categories are detected pointwise, it follows that the right adjoint $U_*$ preserves colimits and is conservative. The projection formula can also be verified pointwise by using the equivalence 
		\[
			\Map_{\Fun(\cat C, \cat D)}({\mathcal F} \otimes U_*({\mathcal G}), {\mathcal H}) \simeq \Map_{\Fun(\cat C \times \cat C, \cat D)}(\otimes^{\cat D} \circ ({\mathcal F} \times U_*({\mathcal G})), {\mathcal H} \circ \otimes^{\cat C})
		\]
	of mapping spaces (see e.g.~\cite[Corollary 3.6]{Niko16pp}) and the fact that $U_*$ is also a left adjoint. Thus the adjunction $F_* \dashv U_*$ is monadic.

	Finally, to see that $F_\unit A$ is equivalent as a commutative algebra to the image under $U_*$ of the unit of the Day convolution on $\Fun(\cat C,A\MMod)$ just observe that in the following commutative diagram
		\[
		\begin{tikzcd}
			\cat D \ar[d,"UF"'] \ar[r,"F_\unit"] & \Fun(\cat C,\cat D) \ar[d,"U_*F_*"] \ar[r,"F_*"] & \Fun(\cat C,A\MMod_{\cat D}) \ar[d, "U_*"] \\
			\cat D \ar[r, "F_\unit"] & \Fun(\cat C,\cat D) \ar[r,equal] & \Fun(\cat C, \cat D)
		\end{tikzcd}
		\]
	the top row is symmetric monoidal and hence, in particular, preserves units.
\end{proof}

As a consequence we obtain

\begin{Prop} \label{prop: F1modules}
	There is an equivalence of symmetric monoidal $\infty$-categories
		 \[
		 	\LF_1\HZ\MMod \simeq \Fun_{\add}(\A(\Fin_G), \HZ\MMod).
		\]
\end{Prop}
 
\begin{proof}
	Proposition \ref{prop: general on modules} provides an equivalence of symmetric monoidal $\infty$-categories
		\begin{equation}\label{eq:firstequiv}
			F_1\HZ\MMod \simeq \Fun(\A(\Fin_G), \HZ\MMod)
		\end{equation}
	which commutes with the forgetful functors to $\Fun(\A(\Fin_G),\Sp)$. Then consider
		\[
		\begin{tikzcd}[column sep=small]
			& \Fun(\A(\Fin_G),\Sp) \ar[d,shift right] \ar[dl,bend right=10, shift right] \ar[r,"L"] & \Funadd(\A(\Fin_G),\Sp) \ar[d,shift right] \\
			\Fun(\A(\Fin_G),\HZ\MMod) \ar[ur,bend left=10,shift right] \ar[r,"\simeq"] & F_1\HZ\MMod \ar[u,shift right] \ar[r] & \LF_1\HZ\MMod
		\end{tikzcd}
		\]
	where $L$ denotes the smashing localization onto the full subcategory of additive functors. By Proposition~\ref{prop:induced-localization}, the induced functor $F_1\HZ\MMod \to \LF_1\HZ\MMod$ is also a smashing localization. Moreover, since the forgetful functor is conservative, an application of the projection formula shows that the local objects of the induced localization are precisely those that are sent under the forgetful functor to a local object of the original localization (that is, to an additive functor). As additivity is detected by the forgetful functors, we conclude that the bottom row is the smashing localization of $\Fun(\A(\Fin_G),\HZ\MMod)$ whose local objects are the additive functors. In particular, $\Funadd(\A(\Fin_G,\HZ\MMod)$ with its localized Day convolution is equivalent to $\LF_1\HZ\MMod$ as a symmetric monoidal $\infty$-category.
\end{proof}
 
\begin{Cor} \label{cor: first comparison} 
	There is an equivalence of symmetric monoidal $\infty$-categories
 		\[
			\HZ_G\MMod \simeq \Funadd(\A(\Fin_G),\HZ\MMod)
		\]
	and, consequently, an equivalence of tensor triangulated categories
		\[
			\DHZG \simeq \Ho(\Fun_{\add}(\A(\Fin_G), \HZ\MMod))
		\]
	for any finite group $G$.
\end{Cor}

\begin{proof}
	This follows from Corollary~\ref{cor: DHZ is F1Zmod} and Proposition~\ref{prop: F1modules}.
\end{proof}

\begin{center}
$\ast\ast\ast$
\end{center}
 
 \subsection*{Comparison to Kaledin's derived Mackey functors}
Our next task is to show that the tensor triangulated category 
	\[
		\Ho(\Fun_{\add}(\A(\Fin_G), \HZ\MMod))
	\]
is equivalent to Kaledin's category of derived Mackey functors from \cite{Kaledin11}. This section is based on the arguments of \cite[Sections 3--5]{Kaledin11} modified as appropriate to the $\infty$-categorical context to facilitate comparison with Barwick's spectral Mackey functors.

\begin{Rem}
	Kaledin \cite{Kaledin11} provides two different constructions of the triangulated category of derived Mackey functors $\DerKal(G)$.  His first definition is in terms of an $A_\infty$-category associated to the $(2,1)$-category of spans of finite \mbox{$G$-sets}, while a second approach uses a Waldhausen type construction on the category of finite \mbox{$G$-sets.} We will take the latter construction as our starting point (see Definition~\ref{def:kaledin} below).  Nevertheless, as the monoidal structure \cite{Kaledin11} constructs on $\DerKal(G)$ uses the $A_\infty$-approach, we will ultimately need to recall this construction as well (see Definition~\ref{def:DMQ}).
\end{Rem}
 
\begin{Def}
	For a small category $\CC$ with pullbacks, Kaledin \mbox{\cite[Section 4]{Kaledin11}} defines a category $\SC$ as a subcategory of the Grothendieck fibration associated with 
		\[
			\Fun(-,\CC){\op}:\Delta{\op}\rightarrow \Cat.
		\]
	In more detail, $\SC$ is the category whose objects are the pairs
		\[
			([n], X_0\longrightarrow X_1\longrightarrow\cdots \longrightarrow X_n)
		\]
	where $X_0\longrightarrow X_1\longrightarrow\cdots \longrightarrow X_n$ is a diagram in $\CC$ for $n\geq 0$. A morphism $(\alpha, f):([n],X_\bullet)\rightarrow ([m], Y_\bullet)$ consists of a map $\alpha:[n]\rightarrow [m]$ in $\Delta$ and $f:\alpha^\ast (Y_\bullet)\rightarrow X_\bullet$ such that for each $i\geq j$ the commutative square
		\begin{center}
		\begin{tikzcd}
			Y_{\alpha(j)}\arrow[r, "f_j"]\arrow[d]	&	X_j\arrow[d]\\
			Y_{\alpha(i)}\arrow[r, "f_i"]			&	X_i			
		\end{tikzcd}
		\end{center}
	is cartesian. The morphism is called \emph{special} if $\alpha$ is the inclusion of an end-segment and all the maps $f_i$ are isomorphisms. We denote the set of special morphisms by~$I$. 
\end{Def}

\begin{Rem}
	The case of interest is $\cat C = \Fin_G$ the category of finite $G$-sets.  Our first task is to relate Kaledin's category $S\Fin_G$ to Barwick's $\infty$-category $\A(\Fin_G)$.  This will be achieved in Theorem~\ref{thm:burnside} which will recognize $\A(\Fin_G)$ as the infinity-categorical localization of $S\Fin_G$ with respect to the class of special morphisms.  We begin with some general notation and then recall the definition of $\A(\Fin_G)$.
\end{Rem}

\begin{Not}
	Given an $\infty$-category $\DD$, we use the notation
		\[
			\Map_{\DD}(x,y)=\{x\}\times_{\DD}\Fun(\Delta^1,\DD)\times_{\DD}\{y\}
		\]
	for the mapping space between two objects $x,y \in \DD$ (see \cite[\S 1.2.2]{HTTLurie}). Furthermore, we will use the symbol $\lvert \DD\lvert$ for the Kan complex replacement (which is equivalent to inverting all arrows or to geometric realization) and $\DD^\sim\subset\DD$ will denote the $\infty$-groupoid core (the maximal Kan subcomplex). 
\end{Not}

\begin{Def} \label{def: effective Burnside category}
	Let $\cat C$ be a small category with pullbacks. The $\infty$-category $\A(\CC)$ is defined in simplicial degree $n$ to be the subset of \emph{cartesian diagrams}
		\[
			\A(\CC)_n \subset \Fun(\operatorname{Tw}([n]){\op}, \CC)
		\]
	where $\operatorname{Tw}([n])$ denotes the \emph{twisted arrow category} of $[n]$ with objects $(i,j)$ for $0\leq i \leq j\leq n$ and exactly one morphism $(i_1,j_1)\rightarrow (i_2,j_2)$ for $i_2\leq i_1$ and $j_1\leq j_2$. Cartesian here means that a diagram $X$ is contained in $\A(\CC)_n$ if and only if the square
		\[\xymatrix{
			X_{ij} \ar[r] \ar[d] & X_{kj} \ar[d] \\ X_{il} \ar[r]  & X_{kl}
		}\]
	is cartesian for all integers $0 \leq i \leq k \leq l \leq j \leq n$. 
\end{Def}

\begin{Rem}
	The mapping spaces in $\A(\CC)$ can be concretely identified as spans in the following way. Given objects $X$ and $Y$ of $\CC$, we denote by $\Span_{\CC}(X,Y)$ the category with objects the spans $X\leftarrow Z\rightarrow Y$ and morphisms $Z\rightarrow Z'$ lying over $X$ and $Y$. There is a natural equivalence of Kan complexes 
		\[
			\Map_{\A(\CC)}(X,Y)\overset{\simeq}{\longrightarrow}{\N}\Span_{\CC}(X,Y)^{\sim}
		\]
	(see \cite[3.7]{Barwick17}) which sends an $n$-simplex $f:\Delta^1\times\Delta^n\rightarrow \A(\CC)$ to the $n$-tuple of isomorphisms
		\[
			f\vert_{\Delta^1\times\{0\}}\overset{f_0}{\longrightarrow} f\vert_{\Delta^1\times\{1\}}\overset{f_1}{\longrightarrow} \cdots \overset{f_{n-1}}{\longrightarrow} f\vert_{\Delta^1\times\{n\}}
		\]
	in $\Span_{\CC}(X,Y)$. Here the map $f_i$ is the vertical composition in the commutative diagram
		\begin{center}
		\begin{tikzcd}
			X							&	Z_{i+1}\arrow[l]\arrow[r]							&	Y	\\
			X\arrow[u, equal]\arrow[d, equal]	&	W_i	\arrow[u,"\cong"']\arrow[d, "\cong"]\arrow[l]\arrow[r]	&	Y	\arrow[u, equal]\arrow[d, equal]	\\
			X							&	Z_{i}\arrow[l]\arrow[r]								&	Y
		\end{tikzcd}
		\end{center}
	in $\CC$ encoded by the two 2-simplices of $f\vert_{\Delta^1\times\{i,i+1\}}$ in the Burnside category (thinking of a square subdivided with a diagonal into two triangles), where the bottom and top row are the spans associated with $f\vert_{\Delta^1\times\{i\}}$ and $f\vert_{\Delta^1\times\{i+1\}}$.
\end{Rem}

\begin{Cons}\label{cons:comparison-to-burnside}
	We define a functor ${\N} \SC\rightarrow \A(\CC)$ as follows: On objects it sends $([n], X_\bullet)$ to $X_n$ and on morphisms it sends $(\alpha, f):([n],X_\bullet)\rightarrow ([m], Y_\bullet)$ to the span
		\[
			X_n\overset{f_n}{\longleftarrow} Y_{\alpha(n)}\longrightarrow Y_m\text{.}
		\]
	It sends a 2-simplex
		\[\xymatrix{
			([n_0], X_{0, \bullet}) \ar[r]^{\alpha_0} & ([n_1], X_{1, \bullet}) \ar[r]^{\alpha_1} &  ([n_2], X_{2, \bullet})
		}\]
	to
		
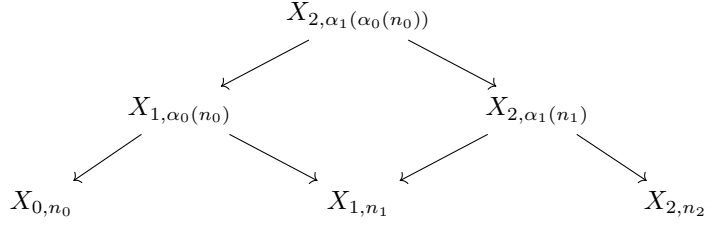
\begin{center}
		\begin{tikzcd}[column sep=small]
			&&	X_{2, \alpha_1(\alpha_0(n_0))}\arrow[ld]\arrow[rd]																			\\
			&	X_{1,\alpha_0(n_0)}\arrow[ld]\arrow[rd]		&&			X_{2,\alpha_1(n_1)}\arrow[ld]\arrow[rd]	&				\\
			X_{0,n_0}		&&			X_{1,n_1}		&&			X_{2,n_2}.		&&	
		\end{tikzcd}
		\end{center}
	A general $k$-simplex 
		\[\xymatrix{
			([n_0], X_{0, \bullet}) \ar[r]^{\alpha_0} & ([n_1], X_{1, \bullet}) \ar[r]^-{\alpha_1} & \dots  \ar[r]^-{\alpha_{k-1} } & ([n_k], X_{k, \bullet})
		}\] 
	is sent to the obvious diagram $Y$ with
		\[
			Y_{j, i+j} = X_{i+j, \alpha_{j+i-1}(\alpha_{j+i-2}(\cdots \alpha_{j} (n_j)))}
		\]
	for $0 \leq i \leq k-j$. The cartesian condition on morphisms in $\SC$ gives the cartesian condition for $\A(\CC)$. The functor ${\N} \SC\rightarrow \A(\CC)$ sends special morphisms to equivalences. 
\end{Cons}

\begin{Thm}\label{thm:burnside}
	The functor ${\N} \SC\rightarrow \A(\CC)$ exhibits the effective Burnside $\infty$-category as the $\infty$-categorical localization ${\N} \SC[I^{-1}]\simeq \A(\CC)$ of $\SC$ with respect to the class of special morphisms $I$.
\end{Thm}

In order to prove Theorem \ref{thm:burnside}, we consider the following more general situation:

\begin{Def}[{{\cite[Definition 4.3]{Kaledin11}}}] \label{Def: complementary}
	Let $\Phi$ be a small category with distinguished classes of morphisms $I$ and $P$. Then $\langle P,I\rangle$ forms a \emph{complementary pair} if the following axioms are satisfied:
	\begin{enumerate}
		\item	The classes $P$ and $I$ are closed under composition and contain all isomorphisms.
		\item	For every object $b\in \Phi$, the full subcategory $\Phi^I_{/b}\subset \Phi_{/b}$ of the slice category consisting of the maps in $I$ admits an initial object $i_b:\iota (b)\rightarrow b$.
		\item	Every map $f$ in $\Phi$ factors uniquely (up to unique isomorphism) as $f=i(f)\circ p(f)$ with $i(f)\in I$ and $p(f)\in P$.\footnote{Note that there is a typo in part (iii) of \cite[Definition~4.3]{Kaledin11}.}
		\item	For every pair of morphisms $b_1\overset{p}{\leftarrow}b\overset{i}{\rightarrow} b_2$ in $\Phi$ with $p\in P$ and $i\in I$, there exists a pushout square
					\begin{center}
					\begin{tikzcd}
						b	\arrow[r, "p"]\arrow[d, "i"]	&	b_1\arrow[d, "i'"]	\\
						b_2	\arrow[r, "p'"]					&	b_{12}
					\end{tikzcd}
					\end{center}
				with $i'\in I$ and $p'\in P$.
	\end{enumerate}
\end{Def}

\begin{Exa}\label{exa:complementary}
	Take $\Phi = \SC$, $I$ the class of special morphisms, and $P$ the collection of morphisms $(\alpha, f):([n],X_\bullet)\rightarrow ([m], Y_\bullet))$ such that $\alpha(0)=0$.  This pair  $\langle P,I\rangle$ is complementary by \cite[Lemma 4.8]{Kaledin11}.
\end{Exa}

\begin{Cons}
	Let $R\Phi$ be the category of diagrams $b_1\overset{i_1}{\longleftarrow} b \overset{i_2}{\longrightarrow} b_2$ such that $i_1, i_2\in I$. Using the canonical projections $\pi_1,\pi_2:R\Phi\rightarrow \Phi$ we construct for every presentable $\infty$-category $\DD$ an endofunctor
		\[
			\Spcl:\Fun(\Phi,\DD)\rightarrow \Fun(\Phi,\DD)
		\]
	as the composition $\Spcl=(\pi_1)_!\circ(\pi_2)^\ast$ of restriction along $\pi_2$ with left Kan extension along $\pi_1$. The common section $\delta:\Phi\rightarrow R\Phi$ of the projections $\pi_1, \pi_2$ provides a natural transformation $\tau:\Id\rightarrow \Spcl$ which is defined as
		\[
			\Id=\delta^\ast\circ\pi_2^\ast\rightarrow\delta^\ast\circ\pi_1^\ast\circ(\pi_1)_!\circ\pi_2^\ast=(\pi_1)_!\circ\pi_2^\ast.
		\]
\end{Cons}

\begin{Lem}\label{Lemma:SpFormula}
	The value of $\Spcl$ on a functor $F\in\Fun(\Phi, \DD)$ can be computed at $b\in\Phi$ as the colimit
		\[
			\Spcl(F)(b)\simeq\colim_{\Phi_b}\pi^\ast F
		\]
	where $\Phi_b$ is the full subcategory of the under category $\Phi_{\iota b/}$ with objects the maps $\iota b \rightarrow \tilde b\in I$ and $\pi:\Phi_b\rightarrow \Phi$ is the projection to $\Phi$.  Under this identification the natural transformation $\tau$ corresponds to the canonical map
		\[
			(\eta_F)(b):Fb\rightarrow \colim_{\Phi_b}\pi^\ast F
		\]
	into the colimit system at the initial map $\iota b \rightarrow b$.
\end{Lem}

\begin{proof}
	We use the pointwise formula for Kan extensions \cite[Lemma 4.3.2.13 and Definition 4.3.3.2]{HTTLurie}. It follows from the axioms that the projection $\pi_1:R\Phi\rightarrow \Phi$ is a weak op-fibration (cofibered functor). Hence the left Kan extension can be computed at $b\in\Phi$ as the colimit (using cofinality, see e.g.~\cite[Theorem 4.1.3.1]{HTTLurie})
		\[
			\Spcl(F)(b)=(\pi_1)_!\pi_2^\ast(F)(b)\simeq \colim_{\pi_1^{-1}(b)}\pi_2^\ast F
		\]
	over the fiber $\pi_1^{-1}(b)$. Moreover, there is an adjunction
		\[
			\pi_1^{-1}(b) \rightleftarrows \Phi_b
		\]
	where the right adjoint sends $b\leftarrow \tilde b\rightarrow \tilde{\tilde b}$ to the composite $\iota b\rightarrow \tilde b\rightarrow \tilde{\tilde b}$ with the unique arrow $\iota b\rightarrow \tilde b$ lying over $b$.
\end{proof}

\begin{Prop}
	Let $\Phi$ be a small category together with a complementary pair $\langle P,I\rangle$ of classes of maps and let $\DD$ be a presentable $\infty$-category. Then $\Spcl$ is a localization functor with local objects the functors $F\in \Fun^I(\Phi, \DD)$ that invert the maps in $I$. 
\end{Prop}

\begin{proof}
	We first claim that if $F\in\Fun(\Phi,\DD)$ inverts morphisms in $I$, then the map
		\[
			\tau_F:F\longrightarrow\Spcl(F)
		\]
	is an equivalence. This follows from the object-wise description of Lemma \ref{Lemma:SpFormula} together with the observation that the natural transformation from the constant functor is an equivalence $\underline {F(\iota c)}\simeq\pi^\ast F$ of functors $\Phi_c\rightarrow \DD$ whose source has an initial object \cite[Corollary 4.4.4.10]{HTTLurie}.  Moreover, for any $F\in \Fun(\Phi, \DD)$ the functor $\Spcl(F)$ inverts morphism in $I$, since for $i:c\rightarrow c'\in I$ there is a unique isomorphism $\iota c\cong \iota c'$ lying over $i$, which induces an equivalence of the relevant slice categories \cite[Section 4.3]{Kaledin11}. This shows that the essential image of $\Spcl$ is given by the $I$-inverting functors and that
		\[
			\tau_{\Spcl(F)}:\Spcl(F)\rightarrow \Spcl(\Spcl(F))
		\]
	is an equivalence for all $F:\Phi\rightarrow \DD$. It remains to show that the map
		\[
			\Spcl(\tau_F):\Spcl(F)\rightarrow \Spcl(\Spcl(F))
		\]
	is an equivalence. Under the zigzag of equivalences 
		\[
			\Spcl(G)=(\pi_1)_!(\pi_2)^\ast G\simeq (\pi_1)_!(\pi_1)^\ast G
		\]
	for $I$-inverting $G$ it is given by the unit map
		\[
			(\pi_1)_!\pi_2^\ast F\xrightarrow{(\pi_1)_!\eta_{\pi_2^\ast F}}(\pi_1)_!\pi_1^\ast (\pi_1)_!\pi_2^\ast F,
		\]
	which can be seen pointwise by evaluating $\eta_{\pi^\ast_2F}$ at the zig-zag of (horizontal) morphisms
		\begin{center}
		\begin{tikzcd}
			c_1							&	c_{12}\arrow[l]\arrow[r]							&	c_2	\\
			c_{12}	\arrow[u]\arrow[d]		&	c_{12}\arrow[u, equal]\arrow[d]\arrow[l, equal]\arrow[r]	&	c_2	\arrow[u, equal]\arrow[d, equal]	\\
			c_2							&	c_2\arrow[l, equal]\arrow[r, equal]		&	c_2	
		\end{tikzcd}
	\end{center}
	in $R\Phi$. The counit $\epsilon_{\Spcl(F)}$ is a left inverse and so it suffices to show that it is an equivalence. Furthermore, for any $I$-inverting functor $G$, under the zigzag of equivalences $\Spcl(G)\simeq (\pi_1)_!\pi_1^\ast G$, the natural map $\tau_G$ corresponds to
		\[
			G=\delta^\ast\pi_1^\ast G\xrightarrow{\delta^\ast\eta_{\pi_1^\ast G}}\delta^\ast\pi_1^\ast(\pi_1)_!\pi_1^\ast G=(\pi_1)_!\pi_1^\ast G.
		\]
	This is an equivalence (using the first paragraph of the proof) with a left inverse $\epsilon_G=\delta^\ast\pi_1^\ast \epsilon_{G}$, which is thus also an equivalence.
\end{proof}

\begin{Rem}
	A formal consequence of the Yoneda Lemma \cite[5.1.3]{HTTLurie} is that the mapping spaces in the $\infty$-categorical localization $(\N\Phi)[I^{-1}]$ can be identified as
		\[
			\Map_{(\N\Phi)[I^{-1}]}(c_1,c_2)\simeq \Spcl(\Phi(c_1,-))(c_2).
		\]
\end{Rem}

\begin{Def}
	Let $Q_I(b_1,b_2)$ be the category of cospans $b_1\overset{p}{\rightarrow}b\overset{i}{\leftarrow}b_2$ with $p\in P$ and $i\in I$. Forgetting the map $p$ and composing with the map $\iota b_2 \to b_2$ defines a functor $j:Q_I(b_1,b_2)\rightarrow \Phi_{b_2}$. Recall that we also have the functor $\pi : \Phi_{b_2} \to \Phi$. An object $c \in \CC$ is called \emph{simple} if the map $\iota c \to c$ is an isomorphism. 
\end{Def}

\begin{Lem}
	For a simple $b_2$, the natural transformation
		\[
			\underline{\Delta^0}\rightarrow j^\ast\pi^\ast\Phi(b_1,-),\quad (p,i)\mapsto p
		\]
	of functors $Q_I(b_1,b_2)\rightarrow \mathcal S$ exhibits the functor $\pi^\ast \Phi(b_1,-)$ as a left Kan extension
		\[
			\pi^\ast\Phi(b_1,-)\simeq j_!(\underline{\Delta^0})
		\]
	along $j$. Here $\underline{\Delta^0}$ is the constant functor associated with the point and $ \mathcal S$ the $\infty$-category of spaces.
\end{Lem}

\begin{proof}
	We fix a map $i':\iota b_2\rightarrow b'\in I$ and consider the slice category $j_{/i'}$ with objects the diagrams
		\begin{center}
		\begin{tikzcd}
			b_1	\arrow[r, "p"]					&	b\arrow[d]		&	b_2	\arrow[l, "i", swap]		&	  \iota b_2\arrow[l]\arrow[lld, "i'"]		\\
										&	b'	
		\end{tikzcd}
		\end{center}
	Sending such a diagram to the composite $b_1\rightarrow b'$ defines a functor to the discrete category $\Phi(b_1,b')=\pi^\ast\Phi(b_1,-)(i')$, which is a left adjoint by unique factorization (Property (c) of Definition \ref{Def: complementary}). In fact the right adjoint is given by the unique factorization and using that $b_2$ is simple. The claim now follows from the pointwise formula for Kan extensions \cite[Lemma 4.3.2.13 and Definition 4.3.3.2]{HTTLurie} and cofinality \cite[Theorem 4.1.3.1]{HTTLurie}. 
\end{proof}

\begin{Cor}\label{cor:corep}
	The localization of the corepresented functor $\Phi(c_1,-)$ evaluated at a simple object $c_2$ is given by
		\[
			\Spcl(\Phi(c_1,-))(c_2)\simeq \lvert Q_I(c_1,c_2)\rvert
		\]
	and the universal arrow is induced by the functor $\Phi(c_1,c_2)\rightarrow Q_I(c_1,c_2)$ that sends $f=i\circ p:c_1\rightarrow c_2$ to the cospan
		\[
			c_1\overset{p}{\longrightarrow}c\leftarrow \iota c \cong c_2\text{.}
		\]
\end{Cor}

\begin{Exa}
	In the case of the complementary pair from Example~\ref{exa:complementary} (where $\Phi = \SC$) the map from the above corollary is given as follows:
		\begin{align*}
			\SC(X,Y)=\CC(Y,X)&\longrightarrow Q_I(X,Y)	\\
			\quad (Y\rightarrow X)=(([0],X) \to ([0],Y)) &\longmapsto ([0],X)\rightarrow ([0],Y)=([0],Y)
		\end{align*}
	for simple objects $([0],X)$ and $([0],Y)$, i.e.\ $X,Y\in \CC$.
\end{Exa}

We can now show that the effective Burnside category $\A(\CC)$ is obtained from $\SC$ by inverting the special maps:
\begin{proof}[Proof of Theorem \ref{thm:burnside}]
	The functor ${\N} \SC\rightarrow \A(\CC)$ is clearly essentially surjective and so it suffices to consider its effect on mapping spaces. By inspection of the definitions one sees that the composite
		\[
			\Map_{{\N}\SC}(X,Y)\longrightarrow\Map_{\A(\CC)}(X,Y)\overset{\simeq}{\longrightarrow}\N\Span_{\CC}(X,Y)^{\sim}
		\]
	is induced by the map $\SC(X,Y)\rightarrow \Span_{\CC}(X,Y)$ that sends $X\overset{f}{\longleftarrow}Y$ to the span $X\overset{f}{\longleftarrow}Y=Y$. It factors as the composition
		\[
			\SC(X,Y)\longrightarrow Q_I(X,Y)\longrightarrow \Span_{\CC}(X,Y)^{\sim}
		\]
	of the map from Corollary \ref{cor:corep} with the second functor defined by sending a cospan 
		\[
			([0],X)\rightarrow ([n],Z_\bullet)\leftarrow([0],Y)
		\]
	to the span
		\[
			X\overset{f}{\longleftarrow}Z_0\rightarrow Z_n\cong Y\text .
		\]
	This is a right adjoint and so induces an equivalence on classifying spaces. Hence the above composite is a universal arrow. It follows from Corollary \ref{cor:corep} that 
		\[
			\Map_{({\N}\SC)[I^{-1}]}(X,Y)\simeq \Spcl(\SC(X,-))(Y)\overset{\simeq}{\longrightarrow}\Map_{\A(\CC)}(X,Y)
		\]
	is a weak equivalence. Finally, we note that any object in ${\N}\SC[I^{-1}]$ is equivalent to an object of the form $([0],X)$. 
\end{proof}

\begin{Rem}\label{rem:SCmonoidal}
	Let $\CC$ be a small category with pullbacks and a terminal object $\ast$. The cartesian product of $\CC$ does not induce a symmetric monoidal structure on $\SC$. It only yields a functor
		\[
			\SC\times_\Delta \SC\cong S(\CC\times \CC)\rightarrow \SC.
		\]
	However it passes to a product
		\begin{center}
		\begin{tikzcd}
			S(\CC\times \CC)	\arrow[r]\arrow[d]								&	\SC	\arrow[d]	\\
			(\SC)[I^{-1}]\times (\SC)[I^{-1}]		\arrow[r, dashed]			&	(\SC)[I^{-1}]
		\end{tikzcd}
		\end{center}
	on the localization. Here the left vertical arrow exhibits the target as the localization $S(\CC\times \CC)[I^{-1}]$ which follows from the description of the mapping spaces in the proof of Theorem \ref{thm:burnside}.  We now promote this to a full symmetric monoidal structure. 
\end{Rem}

\begin{Cons}\label{cons:monoidal}
	We use notation from \cite[Section 2.1.1]{HALurie}. The symbol $\langle n\rangle$ stands for the pointed set $\{\ast, 1,2, \dots, n\}$ and $ \langle n\rangle^\circ$ for the set $\{1,2, \dots, n\}$. We write $\Fin_\ast$ for the category of finite pointed sets.

	The product on $\CC$ can be rectified to a functor
		\[
			\CC^{\times}:\Fin_\ast\longrightarrow \Cat, \quad \langle n\rangle\mapsto \Fun^{\Pi}(P(\langle n\rangle^\circ),\CC)
		\]
	where $P(-)$ denotes the poset of all subsets and $\Fun^{\Pi}(P(\langle n\rangle^\circ), \CC)$ is the full subcategory of those functors $F:P(\langle n\rangle^\circ)\rightarrow \CC$ such that for every $S\subseteq \langle n\rangle^\circ$ restricting to the elements of $S$ exhibits $F(S)\cong \prod_{s\in S} F(\{s\})$ as a product ($F(\emptyset)=\ast$). In particular, evaluating at the elements of $\langle n\rangle^\circ$ induces an equivalence
		\[
			\Fun^{\Pi}(P(\langle n\rangle^\circ), \CC)\overset{\simeq}{\longrightarrow}\CC^n.
		\]
	The construction $S(-)[I^{-1}]$ preserves products and equivalences. Applying it objectwise, we thus obtain a functor
		\[
			\SC^\times[I^{-1}]:{\N}\Fin_\ast\rightarrow \Cat_\infty
		\]
	that still satisfies the Segal conditions $\SC^\times[I^{-1}](\langle n\rangle^\circ)\simeq (\SC)[I^{-1}]^n$ and thus encodes a symmetric monoidal structure on $\SC[I^{-1}]$. Similarly, we obtain a `pointwise' symmetric monoidal structure
		\[
			 \A^\times\hspace{-0.4ex}(\CC)=\A \circ \CC^\times:{\N}\Fin_\ast\rightarrow\Cat_\infty
		\]
	using the fact that $\A$ preserves products and sends equivalences of ordinary categories to equivalences of $\infty$-categories.
\end{Cons}

\begin{Cor}\label{cor:monoidal-burnside}
	The comparison functor ${\N}\SC\rightarrow \A(\CC)$ of Theorem \ref{thm:burnside} induces an equivalence $\SC^\times[I^{-1}]\xra{\sim}  \A^\times\hspace{-0.4ex}(\CC)$ of symmetric monoidal $\infty$-categories. 
\end{Cor}

\begin{Rem}
	The symmetric monoidal structure we have constructed on $\A(\CC)$ agrees with the one constructed in \cite[Section 2]{BGS20}. To explain this, we need to recall the relative nerve construction (see \cite[Section 3.2.5]{HTTLurie}). 
\end{Rem}

\begin{Def}
	Let $f:\DD \rightarrow \operatorname{Set}_\Delta$ be a functor from a small category $\DD$ to the category of simplicial sets. The \emph{relative nerve} of $f$ is the simplicial set $\N_f$ with an $n$-simplex consisting of a chain
	\[
		d_0\overset{\phi_0}{\longrightarrow}d_1\overset{\phi_1}\longrightarrow\cdots\overset{\phi_{n-1}}\longrightarrow d_n
	\]
	in $\DD$ together with a collection of diagrams $\{\tau_J:\Delta^J\rightarrow f(d_j)\}_{J\subseteq [n]}$ indexed by the non-empty subsets of $[n]$, where $j\in J$ denotes the maximal element of $J$. These are compatible in the following sense: For any inclusion of non-empty subsets $J' \subset J \subset [n]$, with maximal elements $j' \leq j$, the diagram
	\[\xymatrix{
		\Delta^{J'} \ar@{^{(}->}[d]  \ar[r]^{\tau_{J'}}  & f(d_{j'}) \ar[d] \\ \Delta^{J} \ar[r]^{\tau_J} & f(d_j)
	}\]
	commutes. 
\end{Def}

\begin{Rem}
	Let $\DD$ be a small category and $F:\DD \rightarrow \operatorname{Cat}$ a functor such that $f(d)$ has pullbacks for any object $d$ and for any $\phi \colon d \to d'$, the functor $F(\phi)$ preserves pullbacks. Then we get a functor $\A \circ F \colon \DD \to \operatorname{Set}_\Delta$. It follows from \cite[3.2.5.21]{HTTLurie} that the canonical map $\N_{{\A}{F}} \to {\N}{\DD}$ is a cocartesian fibration classified by the functor $\A \circ F \colon {\N}{\DD} \to \Cat_{\infty}$. Now suppose that $\CC$ is a small category with pullbacks and a terminal object. Consider the special case of the latter $\DD=\Fin_\ast$ and $F=\Fun^{\Pi}(P(\langle -\rangle^\circ), \CC)$. A direct but tedious inspection shows that $\N_{{\A}{F}}$ agrees with the total space $\A(\CC)^\otimes$ of the cocartesian fibration of \cite[Proposition 2.14 and Notation 2.6]{BGS20}. As a consequence we obtain:
\end{Rem}

\begin{Cor}\label{cor:Aeffmonoidalagrees}
	The cocartesian fibration $\A(\CC)^\otimes\rightarrow \N\Fin_\ast$ is classified by the functor
	\[
		\A^{\times}\hspace{-0.4ex}(\CC)=\A \circ \CC^\times:\N\Fin_\ast\rightarrow \Cat_\infty.
	\]
	Consequently, the symmetric monoidal structure on $\A(\cat C)$ defined in \cite{BGS20} agrees with the symmetric monoidal structure provided by Construction~\ref{cons:monoidal}.
\end{Cor}

\begin{Cor} \label{cor: Comparison1}
	For any finite group $G$, there is an equivalence of symmetric monoidal $\infty$-categories
	\[
		\Fun_{\add}(\A(\Fin_G), \HZ\MMod) \simeq \Fun_{\add}(S(\Fin_G)[I^{-1}], \HZ\MMod)
	\]
	where both sides are equipped with localized Day convolutions. 
\end{Cor}

\begin{proof} 
	Corollary~\ref{cor:Aeffmonoidalagrees} provides an equivalence of symmetric monoidal \mbox{$\infty$-categories}
	\[
		\Fun(\A(\Fin_G), \HZ\MMod) \simeq \Fun(S(\Fin_G)[I^{-1}], \HZ\MMod)
	\]
	where both sides are equipped with Day convolution products. Additive functors correspond under this equivalence and after localizing we get the desired compatibility of localized Day convolution products \cite[Lemma 3.7]{BGS20}. 
\end{proof}

\begin{Def}\label{def:kaledin}
	Consider the derived category $\Der(\Fun({S}{\Fin_G},\Ab))$ of the abelian category of all functors ${S}{\Fin_G}\to\Ab$.  Note that the objects of the derived category can be regarded as functors ${S}{\Fin_G} \to \Ch(\bbZ)$.  The triangulated category of derived Mackey functors in the sense of Kaledin can be defined as the full subcategory
	\[
		\DerKal(G) \subset \Der(\Fun({S}{\Fin_G},\Ab))
	\]
	consisting of all functors ${S}{\Fin_G} \to \Ch(\bbZ)$ which send the special morphisms of ${S}{\Fin_G}$ to quasi-isomorphisms and which are ``additive'' in the sense that their restriction along the embedding $\Fin_G\op \hookrightarrow {S}{\Fin_G}$ results in a functor $\Fin_G\op \to \Ch(\bbZ)$ that preserves products up to quasi-isomorphism. In Kaledin's notation this category is denoted $\mathcal{D}\mathcal{S}_{\add}(\Fin_G,\Ab)$ and defined in \cite[Def.~4.1 and Def.~4.11]{Kaledin11}.  It follows from \cite[Thm.~4.2 and Prop.~4.7]{Kaledin11} that it is equivalent to the category defined in \cite[Def.~3.3]{Kaledin11} (cf.~Remark~\ref{rem:derkal-equiv} below).
\end{Def}

\begin{Rem}\label{rem:derived-functor-category}
	For any small category $\cat C$, the abelian category 
	\[
		\Ch(\Fun(\cat C,\Ab))\equiv \Fun(\cat C,\Ch(\bbZ))
	\]
	can be equipped with the projective model structure. The underlying $\infty$-category $\Der^\infty(\Fun(\cat C,\Ab))$ is equivalent to the functor $\infty$-category $\Fun({\N}{\cat C},\Der^\infty(\bbZ))$, where $\Der^\infty(\bbZ)$ denotes the underlying $\infty$-category of $\Der(\bbZ)$.  In particular, we have an equivalence of triangulated categories
	\begin{equation}\label{eq:underlying-functor-category}
		\Der(\Fun(\cat C,\Ab)) \cong \Ho(\Fun({\N}{\cat C},\Der^\infty(\bbZ))).
	\end{equation}
	If $I \subset\Mor(\cat C)$ is a class of morphisms, then under the equivalence \eqref{eq:underlying-functor-category}, the full subcategory of $\Der(\Fun(\cat C,\Ab))$ consisting of functors that send morphisms in $I$ to quasi-isomorphisms is equivalent to the full subcategory of the right-hand side consisting of functors of $\infty$-categories ${\N}{\cat C} \to \Der^\infty(\bbZ)$ which send morphisms in $I$ to equivalences in $\Der^\infty(\bbZ)$.  Similarly, if $\cat C$ is semi-additive, the full subcategory of $\Der(\Fun(\cat C,\Ab))$ consisting of functors which are additive up to quasi-isomorphism is equivalent to the full subcategory of additive functors ${\N}{\cat C} \to \Der^\infty(\bbZ)$.  Also note that there is a symmetric monoidal equivalence $\Der^\infty(\bbZ) \cong \HZ\MMod$ by \cite[Thm.~7.1.2.13]{HALurie} which is an $\infty$-categorical version of a theorem of Schwede and Shipley \cite{SS03}.
\end{Rem}

\begin{Prop} \label{prop: Comparison2} 
	There is an equivalence of triangulated categories
	\[
		\Ho(\Funadd(\A(\Fin_G),\HZ\MMod)) \cong \DerKal(G)
	\]
	for any finite group $G$.
\end{Prop} 

\begin{proof}
	According to Definition~\ref{def:kaledin} and Remark~\ref{rem:derived-functor-category}, $\DerKal(G)$ is equivalent to the full subcategory of $\Ho(\Fun({\N}{S}{\Fin_G},\HZ\MMod))$ consisting of those functors which send the special morphisms in ${S}{\Fin_G}$ to equivalences in $\HZ\MMod$ and whose restriction along ${\N}{\Fin_G\op} \hookrightarrow {\N}{S}{\Fin_G}$ sends products to direct sums.  By the universal property of localization we have an equivalence of $\infty$-categories
	\[
		\Fun^I({\N}{S}{\Fin_G},\HZ\MMod) \simeq \Fun({\N}{S}(\Fin_G)[I^{-1}],\HZ\MMod)
	\]
	and since ${\N}{S}{\Fin_G}[I^{-1}] \simeq \A(\Fin_G)$ by Theorem~\ref{thm:burnside}, it remains to understand the additivity condition. To this end note that the composite 
	\[
		\Fin_G\op \hookrightarrow {S}{\Fin_G} \twoheadrightarrow {S}{\Fin_G}[I^{-1}] \simeq \A(\Fin_G)
	\]
	is the (opposite) of the usual embedding. This embedding $\Fin_G\op \hookrightarrow \A(\Fin_G)$ sends products to direct sums (see the proof of \cite[Prop.~4.3]{Barwick17}) and, consequently, since the embedding is surjective on objects, a functor $\A(\Fin_G) \to \HZ\MMod$ preserves direct sums if and only if the composite $\Fin_G\op \hookrightarrow \A(\Fin_G) \to \HZ\MMod$ sends products to direct sums. The diagram
	\[\begin{tikzcd}
		& {\N}{S}{\Fin_G} \ar[dr, twoheadrightarrow] \ar[rr,"F"] & & \HZ\MMod \\
		{\N}{\Fin_G\op} \ar[ur] \ar[rr] &&\A(\Fin_G) \ar[ur,"\overline{F}"']
	\end{tikzcd}\]
	then shows that for a functor $F:{\N}{S}{\Fin_G} \to \HZ\MMod$ which maps special morphisms to equivalences, the induced functor $\overline{F}$ is additive if and only if the restriction of $F$ to $\Fin_G\op$ preserves products. Putting everything together, we conclude that the homotopy category of 
	\[
		\Funadd({\N}{S}{\Fin_G}[I^{-1}],\HZ\MMod) \simeq \Funadd(\A(\Fin_G),\HZ\MMod)
	\]
	is equivalent to the triangulated category $\DerKal(G)$ of Definition~\ref{def:kaledin}.
\end{proof}

\begin{Cor} \label{cor: Comparison3}
	For any finite group $G$, there is an equivalence of triangulated categories
	\[
		\DHZG \cong \DerKal(G)
	\]
\end{Cor}

\begin{proof} 
	This follows from Corollary \ref{cor: first comparison}, Corollary \ref{cor: Comparison1} and Proposition \ref{prop: Comparison2}.
\end{proof}

\begin{Rem}
	The last thing that remains is to compare the monoidal structure of $\DHZG$ with the monoidal structure of $\DerKal(G)$ defined in \cite[Section 5.2]{Kaledin11}.  This monoidal structure is actually constructed on an equivalent triangulated category (denoted $\DerKal_Q(G)$ below) which Kaledin constructs using $A_\infty$-categories.  We briefly recall the construction.
\end{Rem}

\begin{Rem}\label{rem:Ainftycat}
	Every 2-category $\cat C$ gives rise to an associated $A_\infty$-category $\mathcal{B}(\cat C)$ which has the same objects and whose complex of morphisms (for a pair of objects $X$ and $Y$) is the simplicial chain complex $C_\bullett(\cat C(X,Y))$ of the nerve of the category $\cat C(X,Y)$.  See \cite[Sections 1.5--1.6]{Kaledin11} for details.  We can then consider its derived category of $A_\infty$-modules $\Der(\mathcal{B}(\cat C))$, that is, the derived category of $A_\infty$-functors from $\mathcal{B}(\cat C)$ to the dg-category $\Ch(\bbZ)$ (see \cite{Keller01} for instance).
\end{Rem}

\begin{Exa}\label{exa:1-cat-bar}
	Applied to an ordinary category $\cat C$, the resulting $A_\infty$-category $\mathcal{B}(\cat C)$ is quasi-isomorphic to the additive category $\bbZ[\cat C]$ (obtained by linearizing the hom sets) regarded as a dg-category whose morphism complexes are concentrated in degree~0.  In this case, an $A_\infty$-module is just an ordinary functor $\cat C \to \Ch(\Z)$ and the derived category $\Der(\mathcal{B}(\cat C))$ is simply the derived functor category $\Der(\Fun(\cat C,\Ab))$.
\end{Exa}

\begin{Def}\label{def:DMQ}
	Let $Q(\Fin_G)$ denote the $(2,1)$-category of spans of finite $G$-sets whose 2-morphisms are the isomorphisms of spans.  The usual embedding $\Fin_G\op \hookrightarrow Q(\Fin_G)$ can be regarded as a 2-functor and it induces a map between the associated $A_\infty$-categories $\bbZ[\Fin_G\op] \to \mathcal{B}(Q(\Fin_G))$.  Kaledin considers the full subcategory
	\[
		\DerKal_Q(G) \subset \Der(\mathcal{B}(Q(\Fin_G)))
	\]
	 of the derived category of $A_\infty$-modules consisting of those $A_\infty$-functors 
	\[
		\mathcal{B}(Q(\Fin_G))\to\Ch(\Z)
	\]
	that are additive in the sense that their restriction (to an ordinary functor) $\Fin_G\op \to \Ch(\Z)$ preserves products up to quasi-isomorphism (see \cite[Definition~3.2]{Kaledin11}).
\end{Def}

\begin{Rem}\label{rem:derkal-equiv}
	The category $\DerKal_Q(G)$ is equivalent to the triangulated category $\DerKal(G)$ defined in Definition~\ref{def:kaledin}. This is established in \cite[Section 4]{Kaledin11} but we reformulate the result in a way that suits our purposes. Regarding $S(\Fin_G)$ as a discrete $2$-category, there is a $2$-functor $\phi:S(\Fin_G)\to Q(\Fin_G)$ given on objects and morphisms exactly like the functor ${\N}{S(\Fin_G)} \to \A(\Fin_G)$ from Construction~\ref{cons:comparison-to-burnside} (see \cite[Section 4.2]{Kaledin11}).  This functor $\phi$ preserves products (disjoint unions) when restricted to $\Fin_G\op$ and sends morphisms in $I$ to equivalences. Further, it induces an $A_\infty$-functor $\Z S(\Fin_G) \to \mathcal{B}(Q(\Fin_G))$ where $\Z S(\Fin_G)$ is the linearization of the category $S(\Fin_G)$ (see Example~\ref{exa:1-cat-bar}). Restricting along this $A_\infty$-functor provides a triangulated functor
	\begin{equation}\label{eq:DMcomparison}
		\chi:\DerKal_Q(G) \to \DerKal(G)
	\end{equation}
	which one proves is an equivalence by combining \cite[Theorem~4.2]{Kaledin11} and \cite[Proposition~4.7]{Kaledin11}.
\end{Rem}

\begin{Cons}\label{cons:symkaledin}
	We now recall the symmetric monoidal structure that \cite[Section 5.2]{Kaledin11} constructs on $\DerKal_Q(G)$. The Cartesian product of finite $G$-sets induces a $2$-functor
		\[
			m \colon  Q(\Fin_G \times \Fin_G) \to {Q}(\Fin_G)
		\]
	which provides a functor
		\[
			m^* \colon \Der(\mathcal{B}({Q}(\Fin_G))) \to \Der(\mathcal{B}({Q}({\Fin_G} \times {\Fin_G}))).
		\]
	Let $m_!^{\add}$ denote the left adjoint of the composite functor
		\[
			\DerKal_Q(G) \hookrightarrow \Der(\mathcal{B}({Q}(\Fin_G))) \overset{m^*}\to \Der(\mathcal{B}({Q}({\Fin_G} \times {\Fin_G}))).
		\]
	Using the Alexander--Whitney map one can define an external product
		\[
			\Der(\mathcal{B}(Q(\Fin_G))) \times \Der(\mathcal{B}(Q(\Fin_G))) \overset{\boxtimes}\to  \Der(\mathcal{B}(Q(\Fin_G \times \Fin_G))))
		\]
	and the symmetric monoidal product on $\DerKal_Q(G)$ is defined to be the composite $m_!^{\add} \circ \boxtimes \circ (i\times i)$ where $i:\DerKal_Q(G) \hookrightarrow \Der(\mathcal{B}(Q(\Fin_G)))$ denotes the inclusion. The original category $\DerKal(G)$ then obtains a symmetric monoidal structure via the equivalence \ref{eq:DMcomparison}.
\end{Cons}

\begin{Thm}\label{thm:equivalence}
	There is an equivalence of tensor triangulated categories
		\[
			\DHZG \cong \DerKal(G)
		\]
	for any finite group $G$.
\end{Thm}

\begin{proof}
	We have already established a symmetric monoidal equivalence
		\begin{equation*}
			\HZ_G\MMod \simeq \Funadd(S(\Fin_G)[I^{-1}],\HZ\MMod)
		\end{equation*}
	which provides a triangulated equivalence $\DHZG \cong \DerKal(G)$ at the level of homotopy categories (Cor.~\ref{cor: Comparison3}). It remains to show that the triangulated equivalence $\chi:\DerKal_Q(G) \xra{\sim} \DerKal(G)$ of Remark~\ref{rem:derkal-equiv} is symmetric monoidal when $\DerKal(G)$ is equipped with the symmetric monoidal structure induced by the localized Day convolution product on $\Funadd(S(\Fin_G)[I^{-1}],\HZ\MMod)$.

	To this end, note that we have a strictly commutative diagram of $2$-functors
		\begin{equation}\label{eq:first-dia}
		\begin{tikzcd}
			S(\Fin_G) \times S(\Fin_G) \ar[r,"\phi \times \phi"] & Q(\Fin_G) \times Q(\Fin_G) \\
			S(\Fin_G \times \Fin_G) \ar[d,"m"] \ar[u] \ar[r,"\phi"] & Q(\Fin_G \times  \Fin_G) \ar[u,"\sim" labelrotatebelow] \ar[d,"m"] \\ S(\Fin_G) \ar[r,"\phi"] &  Q(\Fin_G)
		\end{tikzcd}
		\end{equation}
	where $m \colon S(\Fin_G \times \Fin_G) \to  S(\Fin_G)$ is induced from the Cartesian product. We can apply the $\mathcal{B}(-)$ construction (Rem.~\ref{rem:Ainftycat}) to obtain a corresponding diagram of $A_\infty$-categories.  For example, the bottom arrow of the diagram becomes the \mbox{$A_\infty$-functor} $\bbZ S(\Fin_G) \to \mathcal{B}Q(\Fin_G)$ which in turn provides a functor $\Der(\mathcal{B}Q(\Fin_G)) \to \Der(\bbZ S(\Fin_G)) \cong \Der(\Fun(S(\Fin_G),\Ab))$ by restriction (Example~\ref{exa:1-cat-bar}).  Moreover, we can add the following piece to the top of the diagram
		\begin{equation}\label{eq:extra-dia}
		\begin{tikzcd}[column sep=large]
			{\bbZ {S}(\Fin_G)} \otimes {\bbZ S(\Fin_G)} \ar[r,"\mathcal{B}(\phi) \otimes \mathcal{B}(\phi)"] & {\mathcal{B}Q(\Fin_G)} \otimes {\mathcal{B}Q(\Fin_G)} \\
			\bbZ[{{S}(\Fin_G)} \times {{S}(\Fin_G)}] \ar[u,"\sim" labelrotateabove] \ar[r,"\mathcal{B}(\phi\times\phi)"] & \mathcal{B}({Q(\Fin_G)} \times {Q(\Fin_G)}) \ar[u]
		\end{tikzcd}
		\end{equation}
	where $\otimes$ denotes tensor product of $A_\infty$-categories and where the right-hand vertical functor is induced by the Alexander--Whitney map. That the diagram commutes simply follows from the observation that the Alexander--Whitney map is an isomorphism in degree zero (given by the Cartesian product of basis sets). Finally, let's arm ourselves with the following commutative diagram of $\infty$-categories
		\begin{equation}\label{eq:extra2-dia}
		\begin{tikzcd}
			S(\Fin_G)\times S(\Fin_G) \ar[r,"q\times q"] & S(\Fin_G)[I^{-1}] \times S(\Fin_G)[I^{-1}] \\
			S(\Fin_G \times \Fin_G) \ar[d,"m"] \ar[u] \ar[r,"q"] & S(\Fin_G\times \Fin_G)[I^{-1}] \ar[u,"\sim" labelrotatebelow] \ar[d,"m"] \\
			S(\Fin_G) \ar[r,"q"] & S(\Fin_G)[I^{-1}]
		\end{tikzcd}
		\end{equation}
	where we have omitted nerves in the notation (cf.~Rem.~\ref{rem:SCmonoidal}). We are now prepared to compare the monoidal structures. Writing 
		\[
			\Der(\cat D) := \Der(\Fun(\cat D,\Ab)) \cong \Ho(\Fun({\N}{\cat D},\HZ\MMod))
		\]
	for an ordinary category $\cat D$ and then using the abbreviations
		\[
			\DM := \DerKal(G), \quad \DMQ := \DerKal_Q(G), \quad \text{ and } \qquad \cat C := \Fin_G,
		\]
	consider the following diagram

	\newsavebox{\tempdiagrammm}
	\begin{lrbox}{\tempdiagrammm}
	\(\begin{tikzcd}[column sep=small]
		\DMQ \times \DMQ \ar[r,hook,"i\times i"] \ar[ddr, draw=none, start anchor=center, end anchor=center]{}[description]{(1)} \ar[dd,"\sim"', "\chi\times\chi"] & \Der(\mathcal{B}Q \cat C) \times \Der(\mathcal{B}Q \cat C)  \ar[dd,"\sim"',"\phi^*\times\phi^*"] \ar[r,"\boxtimes"] & \Der(\mathcal{B}Q(\cat C\times \cat C))\ar[d,"\sim"] \ar[r,"m_{!}^{\add}"]\ar[dr, draw=none, start anchor=center, end anchor=center]{}[description]{(3)} & \DMQ \ar[d,"\sim"', "\chi"] \\
		& \ar[r, draw=none, start anchor=center, end anchor=center]{}[description]{(2)} &\Der(S(\cat C \times \cat C)[I^{-1}]) \ar[d,"\sim"] \ar[r,"\widetilde{m}_!^{\add}"]\ar[dr, draw=none, start anchor=center, end anchor=center]{}[description]{(4)} & \DM \\
		\DM \times \DM \ar[r,hook,"i\times i"] & \Der(\SC[I^{-1}])\times \Der(\SC[I^{-1}]) \ar[r,"\boxtimes"] & \Der(\SC[I^{-1}]\times \SC[I^{-1}]) \ar[r,"t_!"] & \Der(\SC[I^{-1}]). \ar[u,"\add"']
	\end{tikzcd}
	\)
	\end{lrbox}
	\smallskip

	\noindent\resizebox{\linewidth}{!}{\usebox{\tempdiagrammm}}

	\medskip
	\noindent
	Here $\widetilde{m}_!^{\add}$ denotes the left adjoint of 
		\[
			\DerKal(G) \hookrightarrow \Der(S(\Fin_G)[I^{-1}]) \xra{m^*} \Der(S(\Fin_G \times \Fin_G)[I^{-1}])
		\]
	(cf.~Cons.~\ref{cons:symkaledin}) and $t_!$ denotes the left adjoint of restriction along
		\[
			S(\Fin_G)[I^{-1}] \times S(\Fin_G)[I^{-1}] \to S(\Fin_G)[I^{-1}].
		\]
	The lower $\boxtimes$ is the external product $\Fun(\cat A,\cat B) \times \Fun(\cat A,\cat B) \to \Fun(\cat A \times \cat A,\cat B)$ for the symmetric monoidal $\infty$-categories $\cat A = S(\Fin_G)[I^{-1}]$ and $\cat B = \HZ\MMod$.  The composite $t_! \circ \boxtimes$ is the Day convolution on $\Fun(S(\Fin_G)[I^{-1}],\HZ\MMod)$ (see~\cite{Niko16pp}).

	The commutativity of the diagram (up to isomorphism) can be established as follows. Region (1) comes directly from the definition of the comparison functor~$\chi$ (Rem.~\ref{rem:derkal-equiv}) while the more involved region (2) can be checked using the commutative diagrams \eqref{eq:first-dia}, \eqref{eq:extra-dia} and \eqref{eq:extra2-dia}. For the commutativity of (3) note that
		\[
		\begin{tikzcd}
			\DerKal_Q(G) \ar[d,"\chi"] \ar[r,hook] & 
			\Der(\mathcal{B}Q(\Fin_G)) \ar[d,"\phi^*"] \ar[r,"m^*"]
			&\Der(\mathcal{B}Q(\Fin_G \times \Fin_G)) \ar[d,"\phi^*"] \\
			\DerKal(G) \ar[r,hook] & \Der(S(\Fin_G)[I^{-1}]) \ar[r,"m^*"] & \Der(S(\Fin_G \times \Fin_G)[I^{-1}])
		\end{tikzcd}
		\]
	commutes. Kaledin \cite[Section 4]{Kaledin11} establishes that all three vertical functors are equivalences. Hence, we can replace the top and bottom rows with their left adjoints and the diagram still commutes.  Finally, the commutativity of region~(4) is immediate from the definitions.  This completes the proof that the localized Day convolution on $\DerKal(G)$ coincides with the symmetric monoidal structure on $\DerKal(G) \cong \DerKal_Q(G)$ constructed by Kaledin.
\end{proof}

\section{Modules over the Burnside ring Mackey functor}\label{sec:ordinary-derived} 
Instead of considering modules over the equivariant ring spectrum $\HZ_G := \triv_G(\HZ)$, a natural alternative is to consider modules over the equivariant ring spectrum $\HA_G$, that is, the Eilenberg-MacLane $G$-spectrum associated to the Burnside $G$-Mackey functor $\bbA_G$.  As observed by Greenlees--Shipley \cite[Section 5]{GreenleesShipley14}, the derived category of $\HA_G$-modules is equivalent to the ordinary derived category of $G$-Mackey functors.  We will begin by providing a new proof of this fact --- one which takes the monoidal structures into account --- and then explain how the story changes with $\HZ_G$ replaced by $\HA_G$ (i.e.~with the category of derived Mackey functors replaced with the ordinary derived category of Mackey functors).

\begin{Rem}\label{rem:heart-of-SpG}
	The ordinary abelian category of $G$-Mackey functors $\Mack(G)$ is equivalent to the heart of the standard $t$-structure on $\Sp^G$.  Under this equivalence, every Mackey functor $\mathcal{M}\in \Mack(G)$ is associated to its Eilenberg-MacLane \mbox{$G$-spectrum} ${\rmH}{\mathcal{M}} \in \Sp^G$.  The inclusion ${\N}{\Mack(G)} \cong (\Sp^G)^\heartsuit \hookrightarrow \Sp^G$ is lax symmetric monoidal and hence induces a functor 
	\[
		{\N}{\CAlg}(\Mack(G)) = \CAlg({\N}{\Mack}(G)) \to \CAlg(\Sp^G).
	\]
	Thus the Eilenberg-MacLane $G$-spectrum of a commutative Green functor has the structure of a commutative algebra in the symmetric monoidal \mbox{$\infty$-category}~$\Sp^G$.
\end{Rem}

\begin{Rem}
	The homotopy category $\Ho(\A(\Fin_G))$ of the effective Burnside \mbox{$\infty$-category} (Def.~\ref{def: effective Burnside category}) is the ordinary effective Burnside category $\BGeff$ whose objects are finite $G$-sets and whose morphisms are isomorphism classes of spans.  It is a semi-additive category whose group completion is the usual Burnside category~$\BG$.  Since the category of abelian groups is additive, there is an equivalence $\Fun_{\add}(\BG,\Ab)\cong\Funadd(\BGeff,\Ab)$.  In other words, although $G$-Mackey functors are usually defined to be additive functors $\BG \to \Ab$, they can equivalently be defined as functors $\BGeff \to \Ab$ which are ``additive'' in the sense that they preserve biproducts (equivalently, preserve products).
\end{Rem}

\begin{Rem}\label{rem:representable-mackey-functors}
	For each finite $G$-set $T \in \BG$, we have the evaluation functor 
	\[
	\begin{aligned}
		\ev_T:\Mack(G)&\to\Ab \\
		 \mathcal{M} &\mapsto \mathcal{M}(T)
	 \end{aligned}
	\]
	which, by the Yoneda lemma, is representable: $\ev_T \cong \Hom_{\Mack(G)}(\mathcal{M}_T,-)$ where $\mathcal{M}_T := \Hom_{\BG}(T,-) \in \Mack(G)$ is the Mackey functor represented by $T \in \BG$.  
\end{Rem}

\begin{Exa}\label{exa:burnside-mackey}
	The Burnside $G$-Mackey functor $\bbA_G$ is the representable Mackey functor $\mathcal{M}_{G/G}$.  It is the unit for the Day convolution product on $\Mack(G)$ (see Example~\ref{exa:box-product} below).
\end{Exa}

\begin{Def}
	For any finite group $G$, let $\HA_G \in \CAlg(\Sp^G)$ denote the Eilenberg--MacLane $G$-spectrum associated to the Burnside $G$-Mackey functor $\bbA_G$ (see Remark~\ref{rem:heart-of-SpG} and Example~\ref{exa:burnside-mackey}). We let 
	\[
		\Der(\HA_G) := \Ho(\HA_G\MMod_{\Sp^G})
	\]
	denote the homotopy category of the $\infty$-category of $\HA_G$-modules.
\end{Def}

\begin{Lem}\label{lem:compact-generators-for-derived}
	The triangulated category $\Der(\Mack(G))$ is compactly generated by the set of Mackey functors $\SET{ \mathcal{M}_{G/H}}{H \le G}$ regarded as complexes concentrated in degree 0.
\end{Lem}

\begin{proof}
	Since every finite $G$-set is a finite coproduct of orbits, a Mackey functor $\mathcal{N}$ is zero in $\Mack(G)$ if and only if $\ev_{G/H}(\mathcal{N}) = 0$ for all $H \le G$.  Since the functor $\ev_{G/H}:\Mack(G) \to \Ab$ is exact, the representing object $\mathcal{M}_{G/H}$ is projective, and we have
		\begin{equation}\label{eq:mackeygenerators}
			\Hom_{\Der(\Mack(G))}(\mathcal{M}_{G/H}[n],\mathcal{N}_\bbullet) \cong H_n(\ev_{G/H}(\mathcal{N}_\bbullet)) \cong \ev_{G/H}(H_n(\mathcal{N}_\bbullet))
		\end{equation}
	for any complex of Mackey functors $\mathcal{N}_\bbullet$. A complex is thus zero in $\Der(\Mack(G))$ if and only if \eqref{eq:mackeygenerators} vanishes for all $H \le G$ and $n \in \bbZ$.  Morever, \eqref{eq:mackeygenerators} also shows that the $\mathcal{M}_{G/H}[0]$ are compact objects of $\Der(\Mack(G))$ since the right-hand side commutes with coproducts.
\end{proof}

\begin{Rem}
	If $\cat C$ is a small symmetric monoidal additive category, then the category of additive functors $\Funadd(\cat C,\Ab)$ is closed symmetric monoidal with respect to the additive Day convolution.  The tensor product is given by the coend
		\[
			(F \otimes_{\add} G)(c) = \int^{(c_1,c_2)} F(c_1) \otimes G(c_2) \otimes \cat C(c_1 \otimes c_2,c)
		\]
	which implicitly uses that the target category $\Ab$ is copowered over the enriching category ($\Ab$ itself).  This is not the same as the Day convolution on $\Fun(\cat C,\Ab)$ that does not use the $\Ab$-enrichment.  For example, the unit of the additive Day convolution on $\Funadd(\cat C,\Ab)$ is the functor which maps $c \in \cat C$ to the abelian group $\cat C(\unit,c)$, while the unit of the non-enriched Day convolution on $\Fun(\cat C,\Ab)$ is the functor which maps $c \in \cat C$ to the free abelian group generated by the set $\cat C(\unit,c)$.  
\end{Rem}

\begin{Exa}\label{exa:box-product}
	The category of Mackey functors $\Mack(G) = \Funadd(\BG,\Ab)$ is closed symmetric monoidal under the additive Day convolution (with respect to the symmetric monoidal structure on $\BG$ induced from the cartesian product of finite $G$-sets). This symmetric monoidal structure is sometimes called the ``box product'' of Mackey functors.  The unit is the Mackey functor corepresented by the unit of the monoidal structure on $\BG$, that is, the Burnside Mackey functor ${\bbA_G=\mathcal{M}_{G/G}}$.  We equip the derived category $\Der(\Mack(G))$ with the derived symmetric monoidal structure.
\end{Exa}

\begin{Thm}
	There is an equivalence of tensor triangulated categories
	\[
		\Der(\HA_G) \cong \Der(\Mack(G))
	\]
	for any finite group $G$.
\end{Thm}

\begin{proof} 
	For any pointed simplicial presheaf $Y \in \Fun(\mathcal{O}(G)\op, \sSet_*)$, the relative normalized simplicial chain complex $\widetilde{C}_{\bullett}(Y) := C_\bullett(Y{(-)}, \ast)$ can be regarded as a complex of coefficient systems. This provides a functor
		\begin{equation}\label{eq:bredonchains}
			\widetilde{C}_\bullett(-) : \Fun(\mathcal{O}(G)\op, \sSet_*) \to \Ch(\Fun(\mathcal{O}(G)\op, \Ab))
		\end{equation}
	which, using the Eilenberg--Zilber map, is lax symmetric monoidal with respect to the pointwise smash product on
		\[
			\Fun(\mathcal{O}(G)\op, \sSet_*)
		\]
	and the pointwise monoidal structure on 
		\[
			\Ch(\Fun(\mathcal{O}(G)\op,\Ab)) = \Fun(\mathcal{O}(G)\op,\Ch(\bbZ)).
		\]
	To pass from coefficient systems to Mackey functors, we use the induction functor $i_! \colon  \Fun(\mathcal{O}(G)\op, \Ab) \to \Mack(G)$, which is left adjoint to the restriction functor along $i : \mathcal{O}(G)\op \hookrightarrow \BG$. The induction functor $i_!$ is symmetric monoidal with respect to the pointwise monoidal structure on the category of coefficient systems and the box product monoidal structure on the category of Mackey functors (see Example~\ref{exa:box-product}). It then induces a symmetric monoidal functor
		\begin{equation}\label{eq:mackeyfy}
			i_! : \Ch(\Fun(\mathcal{O}(G)\op, \Ab)) \to \Ch(\Mack(G))
		\end{equation}
	on the categories of chain complexes.

	These categories are symmetric monoidal model categories when equipped with the projective model structures, and the functors \eqref{eq:bredonchains} and \eqref{eq:mackeyfy} are left Quillen functors. We denote the composite $i_!(\widetilde{C}_\bullett(-))$ by $\underline{\widetilde{C}}_\bullett(-,\bbA_G)$. This choice of notation can be explained as follows: If $X$ is a pointed $G$-simplicial set then $G/H \mapsto X^H$ gives a cofibrant pointed simplicial presheaf on $\mathcal{O}(G)$ and a straightforward calculation (using \cite[Section 3]{MPN06}, for example) shows that the associated complex of Mackey functors $i_!(\widetilde{C}_\bullett(X^{(-)}))$ is nothing but the Mackey-valued Bredon chain complex of $X$ with coefficients in $\bbA_G$ (see \cite[Section II.9]{tomDieck87}).

	All told, we have a lax symmetric monoidal left Quillen functor
		\[
			\underline{\widetilde{C}}_\bullett(-, \mathbb{A}_G) \colon \Fun(\mathcal{O}(G)\op, \sSet_*) \to  \Ch(\Mack(G))
		\]
	whose lax monoidal structure maps are weak equivalences. By \cite[Theorem~A.7]{NikSch} (see also \cite[Example 4.1.7.6]{HALurie}), we obtain a symmetric monoidal left adjoint 
		\begin{equation}\label{eq:mackeychains}
			\underline{\widetilde{C}}_\bullett(-, \mathbb{A}_G) \colon \mathcal{S}_*^G \to \Der^{\infty}(\Mack(G))
		\end{equation}
	between the underlying symmetric monoidal $\infty$-categories. We would like to show that \eqref{eq:mackeychains} induces a symmetric monoidal functor $\Sp^G \to \Der^\infty(\Mack(G))$. To invoke Theorem~\ref{thm:robalo}, we need to show that $\underline{\widetilde{C}}_\bullett(S^{\rho_G}, \mathbb{A}_G)$ is invertible. This can be checked at the level of homotopy categories. Indeed, given a finite $G$-CW spectrum $X$ (i.e.~a finite $G$-spectrum with a preferred cellular decomposition in the homotopy category), we can define the cellular chain complex $\underline{C}^{cell}_\bullett(X, \mathbb{A}_G)$ using the relative equivariant stable homotopy groups (see \cite[Section~4]{BDP17}) and it is straightforward to check that the following equivalences hold in $\Der(\Mack(G))$:
		\[
			\underline{C}^{cell}_\bullett(\Sigma^{\infty}(G/G_+), \mathbb{A}_G) \simeq \mathcal{M}_{G/G}[0]=\mathbb{A}_G[0]
		\]
	and
		\[
			\underline{C}^{cell}_\bullett(X \wedge Y, \mathbb{A}_G)  \simeq  \underline{C}^{cell}_\bullett(X, \mathbb{A}_G) \otimes \underline{C}^{cell}_\bullett(Y, \mathbb{A}_G).
		\]
	Moreover, it follows from \cite[Section 4]{BDP17} that $\underline{\widetilde{C}}_\bullett(S^{\rho_G}, \mathbb{A}_G)$ is quasi-isomorphic to  $\underline{C}^{cell}_\bullett(S^{\rho_G}, \mathbb{A}_G)$.  Since orbits are self-dual, a cellular structure on $S^{\rho_G}$ induces a cellular structure on $S^{-\rho_G}$ and the cellular chain complex $\underline{C}^{cell}_\bullett(S^{-\rho_G}, \mathbb{A}_G)$ is an inverse of $\underline{\widetilde{C}}_\bullett(S^{\rho_G}, \mathbb{A}_G)$ in $\Der(\Mack(G))$. We can thus invoke Theorem~\ref{thm:robalo} and assert that there is an essentially unique symmetric monoidal left adjoint
		\[
			L  \colon  \Sp^G \to \Der^{\infty}(\Mack(G))
		\]
	such that $L\circ \Sigma^\infty \simeq \underline{\widetilde{C}}_\bullett(-, \mathbb{A}_G)$. If $R$ denotes a right adjoint to $L$ then 
		\begin{equation}\label{eq:computing-right}
			\pi_*^H R(\mathcal{N}_\bbullet) \cong H_*(\mathcal{N}_\bbullet(G/H))
		\end{equation}
	for any complex of Mackey functors $\mathcal{N}_\bbullet$ since $L(\Sigma^{\infty} G/H_+) \simeq \mathcal{M}_{G/H}[0]$.  Moreover, since $L$ sends a set of compact generators to a set of compact generators (see Lemma~\ref{lem:compact-generators-for-derived}), the right adjoint $R$ commutes with colimits and is conservative.  Hence the adjunction is monadic by the Barr--Beck--Lurie Theorem \cite[Theorem~4.7.3.5]{HALurie}.  Since the projection formula holds, we conclude that there is a symmetric monoidal equivalence
		\[
			\Der^{\infty}(\Mack(G)) \simeq R(\unit) \MMod_{\Sp^G}
		\]
	and it remains to show that $R(\unit)$ is equivalent as a commutative algebra to~$\HA_G$. The unit map $\mathbb{S} \to R(\unit)=R(\bbA_G[0])$ is a morphism of commutative algebras and truncates to an isomorphism $\pi_0(\mathbb{S}) \xra{\sim} \pi_0(R(\bbA_G[0]))$ in the heart (see Rem.~\ref{rem:heart-of-SpG} and \cite[Exa.~2.2.1.10]{HALurie}), since both sides are abstractly isomorphic to $\mathbb{A}_G$ and the latter is the initial commutative Green functor.  This provides an isomorphism $\HA_G \xra{\sim} R(\unit)$ since $R(\unit)$ has non-trivial homotopy groups only in degree 0.
\end{proof}

\begin{Rem}
	Many of the basic features of the category $\Der(\HZ_G)$ developed in Section~\ref{sec:highly-structured} hold just as well for $\Der(\HA_G)$ by simply replacing all instances of $\HZ_G$ with~$\HA_G$.  Indeed, properties \ref{it:DHZG}--\ref{it:adjoint} all hold for $\Der(\HA_G)$.  The crucial property that does \emph{not} hold is property \ref{it:finite-loc} which says that we obtain $\DHZ$ if we kill off all the generators of $\Der(\HZ_G)$ associated to proper subgroups.  From the authors' point of view, the geometric fixed point functor~$\phigeomb{G}$ of an equivariant category \emph{is} this localization killing off all the generators for proper subgroups (morally, killing everything that comes by induction from proper subgroups).  For $\SH(G)$ the result of this localization is the category associated with the trivial subgroup: the nonequivariant stable homotopy category $\SH$.  Property \ref{it:finite-loc} asserts that the same is true for the category of derived Mackey functors $\Der(\HZ_G)$: the result of the localization is $\Der(\HZ_1)\cong\Der(\HZ)$.  We will show below that property \ref{it:finite-loc} fails for $\Der(\HA_G)$ even for the smallest nontrivial group $G=C_2$.  Nevertheless, Proposition~\ref{prop:induced-localization} does provide a general description of the localization, as follows:
\end{Rem}

\begin{Cor}
	For any finite group $G$, the finite localization of $\HA_G\MMod_{\Sp^G}$ associated to the set $\SET{F_G(G/H_+)}{H \lneq G}$ is, up to equivalence, the functor on module categories 
		\[
			\HA_G\MMod_{\Sp^G} \to \phigeomb{G}(\HA_G)\MMod_{\Sp}
		\]
	induced by the geometric fixed point functor $\phigeomb{G}:\Sp^G \to \Sp$. In particular,
		\[
			\Der(\HA_G) / \Loc_{\otimes}\langle F_G(G/H_+) \;|\; H \lneq G \rangle \cong \Der(\phigeomb{G}(\HA_G)).
		\]
\end{Cor}

\begin{Rem}
	In other words, the target of the ``geometric fixed point'' functor~$\phigeomb{G}$ associated to the category $\Der(\HA_G)$ is the derived category $\Der(\phigeomb{G}(\HA_G))$ of the ring spectrum $\phigeomb{G}(\HA_G) \in \CAlg(\Sp)$.  More generally, for any subgroup $H \le G$, the target of the geometric fixed point functor $\Phi^H$ on $\Der(\HA_G)$ is the derived category of the ring spectrum $\Phi^H(\HA_G) \cong \Phi^H(\HA_H) \in \CAlg(\Sp)$.  The heart of the issue is that the equivariant Eilenberg-MacLane spectra $\HA_G$ do not behave well with respect to geometric fixed points.  For categorical fixed points, it is immediate from the definition that the non-equivariant spectrum $(\HA_G)^H$  is the Eilenberg-MacLane spectrum ${\rmH}A(H)$ associated to the Burnside ring $A(H)$.  In contrast, the geometric fixed points $\Phi^H(\HA_G)$ seem more mysterious and more complicated.
\end{Rem}

\begin{Prop}\label{prop:counter-example}
	The homotopy ring $\pi_* \Phi^{C_2}\HA_{C_2}$ is isomorphic to the graded ring $\mathbb{Z}[x]/(2x)$ where $x$ has degree two.  In particular, $\Phi^{C_2}\HA_{C_2}$ is not equivalent to $\HZ$.
\end{Prop}

\begin{proof}
	Let $\HZbar$ denote the $C_2$-equivariant Eilenberg-MacLane spectrum associated to the constant Mackey functor $\underline{\mathbb{Z}}$.  It follows from the Tate square of \cite{GreenleesMay95b} and the Tate cohomology of $C_2$ that $\pi_* \Phi^{C_2}\HZbar$ is the polynomial algebra $\mathbb{F}_2[x]$ where $x$ has degree two.  Comparing the isotropy separation sequences for $\HZbar$ and $\HA_{C_2}$, we conclude that the canonical map $\Phi^{C_2} \HA_{C_2} \to \Phi^{C_2} \HZbar$ induces an isomorphism on homotopy groups in positive degrees. On the other hand, again using the isotropy separation sequence, we know that  $\pi_0$ of geometric fixed points is isomorphic to $\pi_0$ of genuine fixed points modulo proper transfers. Hence ${\pi_0 \Phi^{C_2}\HA_{C_2}\cong A(C_2)/[C_2] \cong \mathbb{Z}}$. This completes the proof. 
\end{proof}

\begin{Cor}
	The derived category of the integers $\Der(\HZ)$ is not equivalent as a triangulated category to the homotopy category of modules over $\Phi^{C_2}\HA_{C_2}$.
\end{Cor}

\begin{proof}
	If $P$ is any compact generator in $\DHZ$ then it is a perfect complex and hence quasi-isomorphic to a direct sum of finitely generated abelian groups.  This implies that the endomorphism ring spectrum of $P$ has bounded homotopy groups.  But the homotopy groups of $\Phi^{C_2}\HA_{C_2}$ are not bounded by Proposition~\ref{prop:counter-example}.
\end{proof}

\begin{Rem}
	This demonstrates that for the ordinary derived category of Mackey functors $\Der(\Mack(G))\cong\Der(\HA_G)$, the target category of the geometric fixed point functor $\Phi^H$ varies with the subgroup $H$.  For example, $\DHZ$ is always the target of the geometric fixed point functor $\Phi^{1} \cong \res^G_{1}$ associated to the trivial subgroup, but if $G$ contains a subgroup $H=C_2$ then the target of $\Phi^H$ is $\Der(\Phi^{C_2}(\HA_{C_2}))$ which is not equivalent to $\Der(\HZ)$.  This is in stark contrast to examples like $\SH(G)$ or the category of derived Mackey functors $\Der(\HZ_G)$ where the geometric fixed point functors $\Phi^H$ always land in the same category, namely the category associated to the trivial group.
\end{Rem}


\bibliographystyle{alphasort}\bibliography{TG-articles}\end{document}